\documentclass[11pt,reqno]{amsart}
\usepackage[letterpaper,margin=1in,footskip=0.25in]{geometry}
\usepackage{microtype}
\usepackage{amssymb}
\usepackage{mathrsfs}
\usepackage[mathic=true]{mathtools}
\usepackage{enumitem}
\usepackage{booktabs,ltablex}
\usepackage[referable]{threeparttablex}
\usepackage[tableposition=top]{caption}
\usepackage[noadjust]{cite}
\renewcommand{\citepunct}{;\penalty\citemidpenalty\ }

\usepackage{trimclip}
\usepackage{tikz-cd}

\PassOptionsToPackage{pdfusetitle,colorlinks,pagebackref,bookmarksdepth=3,linktoc=all,pdfencoding=auto}{hyperref}
\usepackage{bookmark}
\hypersetup{citecolor=[HTML]{2E7E2A},linkcolor=[HTML]{800006},urlcolor=[HTML]{8A0087}}
\usepackage[hyphenbreaks]{breakurl}

\numberwithin{equation}{section}

\newtheorem{theorem}[equation]{Theorem}

\newtheorem{proposition}[equation]{Proposition}

\newtheorem{innercustomthm}{Theorem}
\newenvironment{customthm}[1]
  {\renewcommand\theinnercustomthm{#1}\innercustomthm}
  {\endinnercustomthm}

\newtheoremstyle{cited}{.5\baselineskip\@plus.2\baselineskip\@minus.2\baselineskip}{.5\baselineskip\@plus.2\baselineskip\@minus.2\baselineskip}{\itshape}{}{\bfseries}{\bfseries .}{5pt plus 1pt minus 1pt}{\thmname{#1}\thmnumber{ #2}\thmnote{ \normalfont#3}}
\theoremstyle{cited}
\newtheorem{citedthm}[equation]{Theorem}
\newtheorem{citedconj}[equation]{Conjecture}
\newtheorem{citedprop}[equation]{Proposition}

\theoremstyle{definition}
\newtheorem{definition}[equation]{Definition}
\newtheorem{example}[equation]{Example}
\newtheorem{setup}[equation]{Setup}

\newtheoremstyle{citeddef}{.5\baselineskip\@plus.2\baselineskip\@minus.2\baselineskip}{.5\baselineskip\@plus.2\baselineskip\@minus.2\baselineskip}{}{}{\bfseries}{\bfseries .}{5pt plus 1pt minus 1pt}{\thmname{#1}\thmnumber{ #2}\thmnote{ \normalfont#3}}
\theoremstyle{citeddef}
\newtheorem{citeddef}[equation]{Definition}

\theoremstyle{remark}
\newtheorem{remark}[equation]{Remark}

\newtheoremstyle{step}{.25\baselineskip\@plus.1\baselineskip\@minus.1\baselineskip}{.25\baselineskip\@plus.1\baselineskip\@minus.1\baselineskip}{\itshape}{}{\bfseries}{\bfseries .}{5pt plus 1pt minus 1pt}{\thmname{#1}\thmnumber{ #2}\thmnote{ \normalfont(#3)}}
\theoremstyle{step}
\newtheorem{step}{Step}[equation]

\DeclareMathOperator{\Ann}{Ann}
\DeclareMathOperator{\Ass}{Ass}
\DeclareMathOperator{\DD}{\mathbf{D}}
\DeclareMathOperator{\Frac}{Frac}
\DeclareMathOperator{\Hom}{Hom}
\DeclareMathOperator{\Perv}{Perv}
\DeclareMathOperator{\Reg}{Reg}
\DeclareMathOperator{\RRGGamma}{\underline{R\Gamma}}
\DeclareMathOperator{\RRHHom}{\underline{RHom}}
\DeclareMathOperator{\Spec}{Spec}
\DeclareMathOperator{\Tot}{Tot}
\DeclareMathOperator{\depth}{depth}
\DeclareMathOperator{\im}{im}

\newcommand{\Dp}{\prescript{\mathfrak{p}}{}{\mathcal{D}}}
\newcommand{\an}{\mathrm{an}}

\newcommand{\DR}{\mathrm{DR}}
\newcommand{\et}{\mathrm{\acute{e}t}}
\newcommand{\id}{\mathrm{id}}
\newcommand{\rh}{\mathit{rh}}

\newcommand{\CC}{\mathbf{C}}
\newcommand{\LL}{\mathrm{L}}
\newcommand{\QQ}{\mathbf{Q}}
\newcommand{\ZZ}{\mathbf{Z}}
\newcommand{\cC}{\mathcal{C}}
\newcommand{\cD}{\mathcal{D}}
\newcommand{\cH}{\mathcal{H}}
\newcommand{\cO}{\mathcal{O}}
\newcommand{\fm}{\mathfrak{m}}
\newcommand{\fp}{\mathfrak{p}}
\newcommand{\fq}{\mathfrak{q}}
\newcommand{\sD}{\mathscr{D}}
\newcommand{\sH}{\mathscr{H}}
\newcommand{\sM}{\mathscr{M}}
\newcommand{\sN}{\mathscr{N}}

\providecommand\given{}
\newcommand\SetSymbol[1][]{\nonscript\:#1\vert\allowbreak\nonscript\:\mathopen{}}
\DeclarePairedDelimiterX\Set[1]\{\}{\renewcommand\given{\SetSymbol[\delimsize]}#1}

\newcommand{\hooklongrightarrow}{\lhook\joinrel\longrightarrow}

\begin{document}
\title[Local cohomology modules of excellent locally unramified regular
rings]{Finiteness of associated primes for\\local cohomology modules of excellent locally\\
unramified regular rings of finite Krull dimension}
\author{Takumi Murayama}
\address{Department of Mathematics\\Purdue University\\150 N. University
Street\\West Lafayette, IN 47907-2067\\USA}
\email{\href{mailto:murayama@purdue.edu}{murayama@purdue.edu}}
\urladdr{\url{https://www.math.purdue.edu/~murayama/}}

\thanks{This material is based upon work supported by the National Science
Foundation under Grant No.\ DMS-2201251}


\makeatletter
  \hypersetup{
    pdftitle={Finiteness of associated primes for local cohomology modules of excellent locally unramified regular rings of finite Krull dimension},
    pdfsubject=\@subjclass,
    pdfkeywords={local cohomology, associated prime, regular ring, perverse sheaf, D-module}
  }
\makeatother

\begin{abstract}
  Thirty years ago, Huneke (for local rings) and Lyubeznik (in
  general) conjectured that for all regular rings $R$, the local cohomology
  modules $H^i_I(R)$ have finitely many associated prime ideals.
  We prove substantial new cases of their conjecture by proving that the
  local cohomology modules $H^i_I(R)$ have finitely many
  associated prime ideals whenever $R$ is an excellent regular
  ring of finite Krull dimension such that $R/pR$
  is regular and $F$-finite for every prime number $p$.
  Our result is new even for excellent regular $\mathbf{Q}$-algebras of
  finite Krull dimension, for example for finitely generated rings over formal
  power series rings over fields of characteristic zero.
  Our proof uses perverse sheaves, $\mathscr{D}$-modules,
  the Riemann--Hilbert correspondence for smooth complex varieties,
  N\'eron--Popescu desingularization,
  and
  a delicate Noetherian approximation argument.
\end{abstract}

\maketitle

\section{Introduction}\label{sect:intro}
\subsection{Background}
Let \(R\) be a Noetherian ring and let \(I\) be an ideal in \(R\).
In \cite[Expos\'e XIII, Conjecture 1.1]{SGA2}, Grothendieck conjectured that
\(\Hom_R(R/I,H^i_I(M))\) is finitely generated when \(M\) is finitely generated.
Hartshorne showed that Grothendieck's conjecture is false in
general \cite[\S3]{Har70}.
On the other hand, Hartshorne showed that Grothendieck's conjecture is true in
some cases when \(R\) is a regular local ring \cite[Corollaries 6.3 and
7.7]{Har70}.
\par Since then, an important and difficult open question in commutative algebra
has been to determine what finiteness properties hold for local
cohomology modules.
While Grothendieck's conjecture is false in general even for
regular local rings \citeleft\citen{HK91}\citemid Theorem 2.3 and Example 2.4\citeright,
Huneke asked whether
\(H^i_I(M)\) has
finitely many associated prime ideals when \(M\) is finitely generated
\cite[Problem 4]{Hun92}.
Huneke's question was motivated in part by Faltings's
local-global principle for finite generation of local cohomology modules
\cite{Fal81}.
Huneke's question is weaker than Grothendieck's conjecture since
\[
  \Ass_R\bigl(H^i_I(M)\bigl) =
  \Ass_R\Bigl(\Hom_R\bigl(R/I,H^i_I(M)\bigr)\Bigr)
\]
by \cite[Chapter IV, \S2, no.\ 1, Proposition 10]{Bou72}.\medskip
\par Huneke and Sharp proved that for regular rings \(R\) of equal characteristic
\(p > 0\), the modules \(H^i_I(R)\) have finitely many associated
prime ideals \cite[Corollary 2.3]{HS93}.
Lyubeznik showed the analogue of Huneke and Sharp's result in equal
characteristic zero if \(R\) is local or affine \cite[Theorem 3.4\((c)\) and Remark
3.7\((i)\)]{Lyu93}.
Motivated by these results, Huneke (in the local case) and Lyubeznik (in
general) made the following conjecture, which has now been open for
three decades.
\begin{citedconj}[{\citeleft\citen{Hun92}\citemid Conjecture 5.2\citepunct
  \citen{Lyu93}\citemid Remark 3.7\((i)\)\citepunct
  \citen{Lyu02}\citemid Open Question 1 on p.\
  132\citeright}]\label{conj:lyubeznik}
  Let \(R\) be a regular ring and let \(I \subseteq R\) be an ideal.
  Then, for every \(i \ge 0\), the local cohomology module \(H^i_I(R)\) has
  finitely many associated prime ideals.
\end{citedconj}
In addition to the results from \cite{HS93,Lyu93} listed above, Conjecture
\ref{conj:lyubeznik} is known for differentiably admissible rings over fields of
characteristic zero (in the sense of \cite[Definition 1.2.3.6]{NM14})
\citeleft\citen{Lyu00}\citemid pp.\ 5880--5882\citepunct
\citen{NB13}\citemid Theorem 4.4\citeright\ (see also
\citeleft\citen{MNM91}\citemid \S4\citepunct \citen{Put18}\citemid Theorem
1.2\citeright),
for unramified regular local rings of mixed
characteristic \cite[Theorem 1]{Lyu00}, and for smooth algebras over
Dedekind domains \citeleft\citen{BBLSZ14}\citemid Theorem 4.3\citepunct
\citen{CRS}\citemid Theorem 6.10\citeright.
Conjecture \ref{conj:lyubeznik} is also known when \(\dim(R)\),
\(i\), or \(\dim(R)-i\) is sufficiently small.
See Table \ref{tab:lowdim}.
The assumption that \(R\) is regular
cannot be weakened substantially \cite{Sin00,Kat02,SS04}.
\subsection{Main results}
\par In this paper, we prove Huneke and Lyubeznik's Conjecture
\ref{conj:lyubeznik} for all excellent
locally unramified
regular rings of
finite Krull dimension such that \(R/pR\) is \(F\)-finite for every prime number \(p\).
We note that \(R/pR\) is regular if \(R_\fm\) is an unramified regular local
ring for every maximal ideal \(\fm \subseteq R\) containing \(p\).
Following \cite[p.\ 88]{Coh46}, a regular local ring \((R,\fm)\) is
\textsl{unramified} if it contains a field or if it is of mixed
characteristic \((0,p)\) and \(p \notin \fm^2\).
\par In fact, we show the following more general statement which recovers the
formulation in terms of prime numbers in the previous paragraph
by setting \(A = \ZZ\).
Recall that a ring of prime characteristic \(p > 0\) is \textsl{\(F\)-finite} if
the (absolute) Frobenius map is module-finite \cite[p.\ 464]{Fed83}.
See Definition \ref{def:excellent} for the definition of an excellent ring.
\begin{customthm}{\ref{thm:mainnewtext}}\label{thm:mainnew}
  Let \(R\) be an excellent regular ring of finite Krull dimension that is flat
  over a regular domain \(A\) of dimension \(\le 1\).
  Suppose for every nonzero prime ideal \(\fp \subseteq A\) such that \(A/\fp\)
  is of prime characteristic, the quotient ring
  \(R/\fp R\) is regular and one of the following assumptions hold:
  \begin{enumerate}[label=\((\roman*)\)]
    \item \(R/\fp R\) is \(F\)-finite.
    \item \(A_\fp\) is excellent and
      \(R \otimes_A A_\fp\) is essentially of finite type over \(A_\fp\).
  \end{enumerate} 
  Let \(I \subseteq R\) be an ideal.
  Then, for every \(i \ge 0\), the local cohomology module \(H^i_I(R)\) has
  finitely many associated prime ideals.
\end{customthm}
\begin{table}[t]
  \begin{ThreePartTable}
    \begin{TableNotes}
      \note{The table excludes the cases proved in
      \cite{HS93,Lyu93,Lyu00,NB13,BBLSZ14,CRS}, which hold in arbitrary
      dimension.}
      \item[\(\QQ\)]\label{tn:put16q} \(R\) is excellent (more generally,
        \(\Reg(R/\fp)\) is open for all primes \(\fp\) of coheight \(1\)) and contains
        \(\QQ\).
    \end{TableNotes}
    {\footnotesize
      \renewcommand{\arraystretch}{1.125}
      \begin{longtable}{*{3}{c}}
        \caption{\normalsize Conjecture \ref{conj:lyubeznik} for
        small \(\dim(R)\), \(i\), or \(\dim(R)-i\).\label{tab:lowdim}}\\
        \toprule
        Assumptions & Conjecture \ref{conj:lyubeznik} for \(R\) local
        & Conjecture \ref{conj:lyubeznik} in general\\
        \cmidrule(lr){1-1} \cmidrule(lr){2-3}
        \(i=\depth_I(R)\) &
        \multicolumn{2}{c}{\citeleft\citen{KS99}\citemid Thm.\ 
        B\citepunct \citen{BRS00}\citemid Prop.\ 2.2\citeright}\\
        \(i\le1\) & \multicolumn{2}{c}{\citeleft\citen{KS99}\citemid Thm.\ 
        B\citepunct \citen{BLF00}\citemid Thm.\ 2.2\citeright}\\
        \(i=2\) & \cite[Thm.\ 2.4]{BN08} & \cite[Cor.\ 3.4]{DQ18}\\
        \(i=\dim(R)-1\) & \makebox[\widthof{\citeleft\citen{Put16}\citemid
        Thm.\ 1.4\citepunct \citen{Wan23}\citemid Thm.\ 2.89\citeright}][c]{\cite[Cor.\
        2.5]{Mar01}} &
        \citeleft\citen{Put16}\citemid Thm.\ 1.4\citepunct
        \citen{Wan23}\citemid Thm.\ 2.89\citeright\tnotex{tn:put16q}\\
        \(i=\dim(R)\) & \multicolumn{2}{c}{\cite[Rem.\ 3.11]{BRS00}}\\
        \midrule
        \(\dim(R)\le3\) & \citeleft\citen{Hel01}\citemid Cor.\ 3\citepunct
        \citen{Mar01}\citemid Thm.\ 2.9\citeright & \cite[Cor.\ 8.5]{Put16}\\
        \(\dim(R)=4\) & \citeleft\citen{Hel01}\citemid Cor.\ 3\citepunct
        \citen{Mar01}\citemid Thm.\ 2.9\citeright &
        \citeleft\citen{Put16}\citemid Thm.\ 1.4\citepunct
        \citen{Wan23}\citemid Thm.\ 2.89\citeright\tnotex{tn:put16q}\\
        \bottomrule
        \insertTableNotes
      \end{longtable}
    }
  \end{ThreePartTable}
\end{table}
For example, Theorem \ref{thm:mainnew} is new for all excellent regular rings of
finite Krull dimension satisfying one of the following conditions:
\begin{enumerate}[label=\((\alph*)\),ref=\alph*]
  \item\label{case:specialqalg}
    \(R\) is a non-local \(\QQ\)-algebra that is neither
      differentiably admissible over a field of characteristic zero nor smooth
      over a Dedekind domain.
    \item \(R\) does not contain a field and is not smooth over a Dedekind
      domain, and \(R/pR\) is regular and \(F\)-finite for every prime number
      \(p\).
\end{enumerate}
Case \((\ref{case:specialqalg})\)
includes the case when \(R\) is finitely generated over a formal
power series ring over a field of characteristic zero.
This case was pointed out as a substantial open case of Conjecture
\ref{conj:lyubeznik} in \cite[p.\ 2299]{Hoc19}.
\par Theorem \ref{thm:mainnew} also reproves the known cases of
Conjecture \ref{conj:lyubeznik} proved in
\cite{Lyu93,Lyu00,NB13,BBLSZ14,CRS}.
The only known case of Conjecture \ref{conj:lyubeznik} in arbitrary dimension
not recovered as part of Theorem \ref{thm:mainnew} is the case when \(R\) is of
prime characteristic \(p > 0\) and neither \(F\)-finite nor essentially of finite
type over a Noetherian local ring \cite{HS93}.\medskip
\par We also prove new cases of the following more general conjecture of
Lyubeznik.
The class of \textsl{Lyubeznik functors} on
\(R\) is a class of functors introduced in \cite{Lyu93} and named in
\cite{NB13} that contains all compositions of local cohomology functors
\(H^{i}_{Y}(\,\cdot\,)\) where \(Y\) is a locally closed subset of
\(\Spec(R)\).
See Definition \ref{def:lyubeznikfunctors} for the precise definition of a
Lyubeznik functor.
\begin{citedconj}[{\cite[Remark 3.7\((i)\)]{Lyu93}}]\label{conj:lyubeznikfunctor}
  Let \(R\) be a regular ring and let \(T(\,\cdot\,)\) be a Lyubeznik functor on
  \(R\).
  Then, the module \(T(R)\) has finitely many associated prime ideals.
\end{citedconj}
Conjecture \ref{conj:lyubeznikfunctor} is known for local or differentiably
admissible rings over fields of characteristic zero
\citeleft\citen{Lyu93}\citemid Theorem 3.4\((c)\) and Remark
3.7\((i)\)\citepunct \citen{Lyu00}\citemid pp.\ 5880--5882\citepunct
\citen{NB13}\citemid Theorem 4.4\citeright, in equal
characteristic \(p > 0\) \cite[Corollary 2.14]{Lyu97}, and in mixed
characteristic for unramified regular local rings \cite[Theorem 1]{Lyu00} or for
compositions of local cohomology functors with closed support on
\(\ZZ[x_1,x_2,\ldots,x_n]\)
\cite[Theorem F]{HNBPW19}.\medskip
\par We prove the following new cases of Conjecture \ref{conj:lyubeznikfunctor}.
\begin{customthm}{\ref{thm:lyubeznikfunctors}}
  \label{thm:lyubeznikfunctorsnew}
  Let \(R\) be an excellent regular \(\QQ\)-algebra of finite Krull dimension.
  Let \(T(\,\cdot\,)\) be a Lyubeznik functor on \(R\).
  Then, the module \(T(R)\) has finitely many associated prime ideals.
\end{customthm}
Theorem \ref{thm:lyubeznikfunctorsnew} reproves the cases of
Conjecture \ref{conj:lyubeznikfunctor} proved in
\cite{Lyu93,NB13}.
\subsection{Outline}
In \S\ref{sect:exclyu}, we review the definitions of excellent rings
\cite{EGAIV2,Mat80} and of Lyubeznik functors \cite{Lyu93,NB13}.
In \S\ref{sect:perversesheaves}, we review
background material on
\(t\)-structures and on perverse sheaves on excellent schemes following
\cite{BBDG18,Gab04,Far09,ILO14,Mor25}.
The new result is Theorem \ref{thm:fargues}, which shows that pullback
by a regular morphism of excellent schemes is \(t\)-exact with respect to the
perverse \(t\)-structure after an appropriate shift.
In \S\ref{sect:algdmod}, we review background material on algebraic
\(\sD\)-modules and the Riemann--Hilbert correspondence.
Perverse sheaves and the Rieman--Hilbert correspondence
were previously applied to questions on local cohomology in
commutative algebra in \cite{BBLSZ}.\medskip
\par In \S\ref{sect:equichar0} and \S\ref{sect:mixedchar}, we prove Theorems
\ref{thm:lyubeznikfunctorsnew} and \ref{thm:mainnew}, respectively.
In both Theorems \ref{thm:mainnew} and \ref{thm:lyubeznikfunctorsnew},
the assumption that \(R\) is excellent and of finite Krull dimension are
used to ensure that the category of perverse sheaves on \(\Spec(R \otimes_\ZZ
\QQ)\)
is both Noetherian and Artinian.
See Theorem \ref{thm:perv}.
These finiteness properties are due to Be\u{\i}linson, Bernstein, Deligne, and
Gabber \cite{BBDG18}, Gabber \cite{Gab04}, Fargues \cite{Far09}, and Morel
\cite{Mor25} at various levels of generality and for various rings of
coefficients.
In particular, we will need Gabber's construction of the truncation functor
\(\prescript{\fp}{}{\tau}_{\le0}\) from \cite[\S6]{Gab04},
which we recall in Remark \ref{rem:gabbertruncate}.\medskip
\par In \S\ref{sect:equichar0},
we prove Theorem \ref{thm:lyubeznikfunctorsnew} by showing that the set of
associated prime ideals of \(T(R)\) is contained in the set of generic points of
locally closed subsets of \(\Spec(R)\) associated to the simple components of a
certain perverse sheaf on \(\Spec(R)\).
See Setup \ref{setup:approxperverse} for the definition of this perverse sheaf
and see Theorem \ref{thm:assocaresimple} for the exact comparison statement we
prove.
We show Theorem \ref{thm:assocaresimple} using
N\'eron--Popescu desingularization \citeleft\citen{Pop86}\citemid Theorem
2.4\citepunct \citen{Pop90}\citemid p.\ 45\citepunct\citen{Swa98}\citemid
Theorem 1.1\citeright, Noetherian approximation techniques \cite[\S8]{EGAIV3},
and the Riemann--Hilbert correspondence for smooth
complex varieties (Theorem \ref{thm:riemannhilbert}).
We use the inductive nature of the definition of Lyubeznik functors
(Definition \ref{def:lyubeznikfunctors}) in a crucial way.\medskip
\par In \S\ref{sect:mixedchar}, we prove Theorem \ref{thm:mainnew}.
Theorem \ref{thm:mainnew} follows from Theorem \ref{thm:lyubeznikfunctorsnew} by
using the fact that \(H^i_I(R/pR)\) has finite length as an
\(F\)-module for every prime number \(p\).
\(F\)-modules were introduced by Lyubeznik \cite{Lyu97}.
The finite length property we use (Theorem \ref{thm:fmodfinlen})
is due to 
Blickle and B\"ockle
\cite[Theorem 5.13]{BB11}, which is related to an earlier result due Lyubeznik
\cite[Theorem 3.2]{Lyu97}.
\subsection*{Acknowledgments}
I am grateful to
Rankeya Datta,
Daniel Le,
Yinan Nancy Wang,
Farrah Yhee,
and
Bogdan Zavyalov
for helpful
conversations.
I thank Bhargav Bhatt for insightful comments
on preliminary versions of this paper.

\subsection*{Conventions}
All rings are commutative with identity and all ring maps are unital.
A \textsl{variety} is an integral scheme that is separated and of finite type
over a field.
We will sometimes use set-theoretic notation when talking about inclusions or
intersections of categories.
\par When working with constructible complexes and perverse sheaves,
we denote by \(f_*,f^*,f_!,f^!\) the triangulated versions of the sheaf
operations.
In other words, we drop ``\(\mathrm{R}\)'' and ``\(\mathrm{L}\)'' from our
notation for these functors.
We use the \(\ell\)-adic formalism of \citeleft\citen{Eke90}\citepunct
\citen{Far09}\citemid \S5\citepunct \citen{ILO14}\citemid Expos\'e XIII,
\S4\citeright\ or \cite{BS15}, which match
in our situation by \cite[Proposition 5.5.4]{BS15}.

\section{Excellent rings and Lyubeznik functors}\label{sect:exclyu}
\subsection{Excellent rings}
We define excellent rings.
\begin{citeddef}[{\citeleft\citen{EGAIV2}\citemid D\'efinition 7.8.2\citepunct
  \citen{Mat80}\citemid (32.B), (33.A), and (34.A)\citeright}]\label{def:excellent}
  Let \(R\) be a ring.
  We say that \(R\) is \textsl{excellent} if it is Noetherian and it satisfies
  the following conditions.
  \begin{enumerate}[label=\((\roman*)\)]
    \item \(R\) is universally catenary.
    \item \(R\) is a \textsl{G-ring}, that is, for every prime ideal \(\fp \subseteq R\),
      the \(\fp R_\fp\)-adic completion map
      \(R_\fp \to \hat{R}_\fp\)
      has geometrically regular fibers.
    \item \(R\) is \textsl{J-2}, that is, for every module-finite \(R\)-algebra
    \(S\), the regular locus \(\Reg(S)\) in \(\Spec(S)\) is open.
  \end{enumerate}
\end{citeddef}
\subsection{Lyubeznik functors}
The following class of functors was introduced by Lyubeznik \cite{Lyu93} and
named in \cite{NB13}.
\begin{citeddef}[{\citeleft\citen{Lyu93}\citemid \S1\citepunct
  \citen{NB13}\citemid Definition 4.1\citeright}]\label{def:lyubeznikfunctors}
  Let \(X\) be a scheme.
  A \textsl{Lyubeznik functor \(T\) on \(X\)} is a functor of the form
  \( T = T_r \mathop{\circ} \cdots \mathop{\circ} T_2 \mathop{\circ} T_1 \)
  where each \(T_j\) is one of the following functors.
  \begin{enumerate}[label=\((\roman*)\)]
    \item \(\sH^{i_j}_{Y_j}(\,\cdot\,)\) for a locally
      closed subset \(Y_j\) of \(X\).
    \item The kernel of a morphism in the long exact sequence
      \begin{equation}\label{eq:lyubeznikfunctorles}
        \cdots \longrightarrow \sH^i_{Y'}(\,\cdot\,) \longrightarrow
        \sH^i_{Y}(\,\cdot\,)
        \longrightarrow \sH^i_{Y-Y'}(\,\cdot\,) \longrightarrow \cdots
      \end{equation}
      of local cohomology functors from \cite[Expos\'e I, Th\'eor\`eme 2.8]{SGA2}
      where \(Y\) is a locally closed subset of \(X\) and \(Y'\) is a
      closed subset of \(Y\).
  \end{enumerate}
  If \(X = \Spec(R)\) is affine, we say that \(T\) is a \textsl{Lyubeznik
  functor on \(R\)}.
  In this case, we denote by \(T(R)\) the global sections of \(T(\cO_{X})\).
  This is allowed since \(T(\cO_{X})\) is quasi-coherent
  \cite[Expos\'e II, Proposition 1]{SGA2}.
\end{citeddef}

\section{\texorpdfstring{{\normalfont\textit{t}}}{t}-structures and perverse
sheaves}\label{sect:perversesheaves}
We review preliminaries on perverse sheaves
\cite{BBDG18}.
Perverse sheaves
were previously applied
to commutative algebra, in particular to questions about local cohomology, in
\cite{BBLSZ}.
\subsection{Dimension and perversity functions}
\begin{citeddef}[{\citeleft\citen{ILO14}\citemid Expos\'e XIV, D\'efinition
  2.1.2\citeright}]
  Let \(X\) be a scheme.
  An \textsl{immediate Zariski specialization} \(x \rightsquigarrow y\) is a
  pair of points \(x,y\in X\) such that
  \(y \in \overline{\{x\}}\) and the closure of the image of \(x\) in
  \(\Spec(\cO_{X,y})\) is of dimension 1.
\end{citeddef}
\begin{citeddef}[{\cite[Expos\'e XIV, Proposition 2.1.6 and D\'efinition
  2.1.10]{ILO14}}]
  \label{def:dimfunction}
  Let \(X\) be a universally catenary Noetherian scheme.
  A \textsl{dimension function on \(X\)} is a function \(\delta\colon X \to
  \ZZ\) such that for every immediate Zariski specialization \(x \rightsquigarrow y\)
  of points in \(X\), we have \(\delta(y) = \delta(x) - 1\).
\end{citeddef}
\begin{example}\label{ex:codimdimfunction}
  Let \(X\) be an equidimensional universally catenary Noetherian scheme.
  \begin{enumerate}[label=\((\roman*)\),ref=\roman*]
    \item\label{ex:codimfunction}
      The function \(\delta(x) = -{\dim(\cO_{X,x})}\)
      is a dimension function on \(X\).
      If \(X\) is irreducible, this is proved in \cite[Expos\'e XIV, Proposition
      2.2.2]{ILO14}.
      In general, \(\delta(x) = -{\dim(\cO_{X,x})}\) restricts to a
      dimension function on each irreducible component of \(X\), and hence is a
      dimension function on all of \(X\).
      See also \cite[Proposition 5.5]{Hei17}.
    \item\label{ex:dimfunction}
      Suppose additionally that \(X\) is equicodimensional.
      The function \(\delta(x) = \dim(\overline{\{x\}})\)
      is a dimension function on \(X\) by \cite[Expos\'e XIV, Corollaire
      2.4.4]{ILO14} when \(X\) is irreducible and by \cite[Proposition 3.18]{BH22} in
      general.
  \end{enumerate}
\end{example}
\begin{citeddef}[{\cite[\S1]{Gab04}}]
  Let \(X\) be a Noetherian scheme.
  A \textsl{strong perversity function on \(X\)} is a function \(\fp\colon X \to
  \ZZ \cup \{+\infty\}\) such that for all \(x,y \in X\) with \(y \in
  \overline{\{x\}}\), we have \(\fp(y) \ge \fp(x)\).
\end{citeddef}
We recall the standard example of a strong perversity function.
\begin{example}[Middle perversity {\citeleft\citen{BBDG18}\citemid (2.1.16) and
  (2.2.11)\citeright}]\label{ex:middleperversity}
  Let \(X\) be an equidimensional universally catenary Noetherian scheme of finite Krull
  dimension.
  Setting
  \begin{align*}
    \fp_{1/2}(x) \coloneqq -\bigl(\dim(X)+\delta(x)\bigr) ={} &{\dim(\cO_{X,x})} -
    \dim(X)
  \intertext{for \(\delta\) as in Example
  \ref{ex:codimdimfunction}\((\ref{ex:codimfunction})\) yields a strong perversity
  function, which we call the \textsl{middle perversity}.
  If \(X\) is biequidimensional in the sense of \cite[Definition
  1.2\((ii)\)]{Hei17}, the function \(\fp_{1/2}\) satisfies}
    \fp_{1/2}(x) = -\delta(x) = -&{\dim\bigl(\overline{\{x\}}\bigr)}
  \end{align*}
  for \(\delta\) as in Example
  \ref{ex:codimdimfunction}\((\ref{ex:dimfunction})\) since
  the dimension formula
  holds on \(X\) \citeleft\citen{EGAIV1}\citemid Chapitre 0, Corollaire
  14.3.5\citepunct \citen{Hei17}\citemid Proposition 4.1\citeright.
\end{example}
\subsection{\emph{t}-structures}
We collect some preliminaries on \(t\)-structures following
\cite[\S\S1.2--1.3]{BBDG18} and the summary of the material in \cite{BBDG18}
from \cite[\S2.2]{Jut09}.
\begin{citeddef}[{\cite[D\'efinition 1.3.1]{BBDG18}}]
  A \textsl{\(t\)-category} is a triangulated category \(\cD\) equipped with a
  pair \((\cD^{\le0},\cD^{\ge0})\) of
  strictly full subcategories of \(\cD\) such that setting \(\cD^{\le n}
  \coloneqq \cD^{\le 0}[-n]\) and \(\cD^{\ge n} \coloneqq \cD^{\ge 0}[-n]\) for
  every integer \(n\), the following conditions hold.
  \begin{enumerate}[label=\((\roman*)\)]
    \item For every \(X\) in \(\cD^{\le 0}\) and every \(Y\) in \(\cD^{\ge 1}\),
      we have \(\Hom_\cD(X,Y) = 0\).
    \item We have \(\cD^{\le 0} \subseteq \cD^{\le 1}\) and \(\cD^{\ge 0}
      \supseteq \cD^{\ge 1}\).
    \item For every \(X\) in \(\cD\), there exists a distinguished triangle
      \begin{equation}\label{eq:bbdgdef131iii}
        A \longrightarrow X \longrightarrow B \xrightarrow{+1}
      \end{equation}
      where \(A\) is in \(\cD^{\le 0}\) and \(B\) is in \(\cD^{\ge 1}\).
  \end{enumerate}
  We say that \((\cD^{\le0},\cD^{\ge0})\) is a \textsl{\(t\)-structure} on
  \(\cD\).
  The \textsl{heart} of the \(t\)-structure \((\cD^{\le0},\cD^{\ge0})\) is the
  full subcategory
  \[
    \cD^\heartsuit \coloneqq \cD^{\le 0} \cap \cD^{\ge 0} \subseteq \cD.
  \]
\end{citeddef}
\begin{citedprop}[{\cite[Proposition 1.3.3 and Proposition 1.3.5]{BBDG18}}]
  Let \(\cD\) be a \(t\)-category.
  \begin{enumerate}[label=\((\roman*)\)]
    \item The inclusion of \(\cD^{\le n}\) (resp.\ \(\cD^{\ge n}\)) in \(\cD\)
      admits a right adjoint \(\tau_{\le n}\) (resp.\ a left adjoint \(\tau_{\ge
      n}\)).
    \item For all \(X\) in \(\cD\), there exists a unique morphism
      \(
        d \in \Hom_\cD(\tau_{\ge1}X,\tau_{\le0}X[1])
      \)
      such that
      \[
        \tau_{\le 0} X \longrightarrow X \longrightarrow \tau_{\ge 1} X
        \xrightarrow[\raisebox{3pt}{\(\scriptstyle+1\)}]{d}
      \]
      is a distinguished triangle.
      Up to unique isomorphism, this is the unique distinguished triangle
      \eqref{eq:bbdgdef131iii} with \(A\) in \(\cD^{\le 0}\) and \(B\) in
      \(\cD^{\ge1}\).
    \item Let \(a \le b\).
      For every \(X\) in \(\cD\), there exists a unique morphism
      \begin{equation}\label{eq:tauabxtaubax}
        \tau_{\ge a}\,\tau_{\le b}X \longrightarrow \tau_{\le b}\,\tau_{\ge a}X
      \end{equation}
      making the diagram
      \[
        \begin{tikzcd}[column sep=small]
          \tau_{\le b}X \rar \dar & X \rar & \tau_{\ge a}X\\
          \tau_{\ge a}\,\tau_{\le b}X \arrow{rr}{\sim}[swap]{\eqref{eq:tauabxtaubax}}
          & & \tau_{\le b}\,\tau_{\ge a}X \uar
        \end{tikzcd}
      \]
      commute.
      Moreover, the morphism \eqref{eq:tauabxtaubax} is an isomorphism.
  \end{enumerate}
\end{citedprop}
\begin{citedprop}[{\cite[Proposition 1.2.2]{BBDG18}}]
  Let \(\cD\) be a triangulated category and let \(\cC\) be a full subcategory
  of \(\cD\) such that \(\Hom(X,Y[i]) = 0\) for all \(i < 0\) and \(X,Y\) in
  \(\cC\).
  Let \(f\colon X \to Y\) be a morphism of \(\cC\) and consider a distinguished
  triangle
  \[
    X \overset{f}{\longrightarrow} Y \longrightarrow S \xrightarrow{+1}
  \]
  containing \(f\).
  Suppose that \(S\) fits into a distinguished triangle
  \[
    N[1] \longrightarrow S \longrightarrow C \xrightarrow{+1}
  \]
  where \(N\) and \(C\) are in \(\cC\).
  We then obtain the diagram
  \begin{equation}\label{eq:bbdg1221}
    \begin{tikzcd}[column sep={between origins,4em},row
      sep=scriptsize,baseline=(S.base)]
      N[1]
      \arrow[dd,"\vphantom{\beta}\alpha"',""{name=leftcomm}] \arrow[dr] & & C
      \arrow[ll,"+1"'{pos=0.545},""{name=top1,pos=0.545}]\\
      &|[alias=S]| S \arrow[ur] \arrow[dl,"+1"'{sloped}]\\
      X \arrow[rr,"f"',""{name=bottom1}]
      & & Y \arrow[ul]
      \arrow[uu,"\beta"',""{name=rightcomm,yshift=-1.25pt}]
      \arrow[phantom,from={S.north},to=top1,"\scriptstyle \bigstar"{description,pos=0.55}]
      \arrow[phantom,from={S.south},to=bottom1,"\scriptstyle \bigstar"{description,pos=0.55}]
      \arrow[phantom,from=S,to=leftcomm,"\textstyle \circlearrowright"{description,pos=0.6}]
      \arrow[phantom,from=S,to=rightcomm,"\textstyle \circlearrowright"{description,pos=0.6}]
    \end{tikzcd}
  \end{equation}
  where the triangles marked with \(\circlearrowright\) commute and the triangles
  marked with \(\bigstar\) are distinguished.
  Then, \(\alpha[-1]\colon N \to X\) and \(\beta\colon Y \to C\) are
  respectively a kernel and a cokernel for \(f\) in \(\cC\).
\end{citedprop}
\begin{citeddef}[{\cite[(1.2.3)]{BBDG18}}]
  Let \(\cD\) be a triangulated category and let \(\cC\) be a full subcategory
  of \(\cD\) such that \(\Hom(X,Y[i]) = 0\) for all \(i < 0\) and \(X,Y\) in
  \(\cC\).
  A morphism \(f\colon X \to Y\) of \(\cC\) is \textsl{\(\cC\)-admissible} or
  simply \textsl{admissible} if it appears as the bottom horizontal morphism
  in a diagram of the form
  \eqref{eq:bbdg1221}.
  For every distinguished triangle
  \[
    X \overset{f}{\longrightarrow} Y \overset{g}{\longrightarrow} Z
    \xrightarrow[\raisebox{3pt}{\(\scriptstyle+1\)}]{d}
  \]
  such that \(X,Y,Z\) are in \(\cC\), the morphisms \(f,g\) are
  admissible, \(f\) is a kernel for \(g\), \(g\) is a cokernel for \(f\), and
  \(d\) is uniquely determined by \(f\) and \(g\).
  \par A sequence \(X \to Y \to Z\) in \(\cC\) is an \textsl{admissible short
  exact sequence} if it can be obtained from a distinguished triangle by
  suppressing the degree \(1\) morphism \(d\colon Z \to X[1]\).
\end{citeddef}
\begin{citedprop}[{\cite[Proposition 1.2.4]{BBDG18}}]\label{prop:bbdg124}
  Let \(\cD\) be a triangulated category and let \(\cC\) be a full subcategory
  of \(\cD\) such that \(\Hom(X,Y[i]) = 0\) for all \(i < 0\) and \(X,Y\) in
  \(\cC\).
  Suppose that \(\cC\) is stable under finite direct sums.
  Then, the following conditions are equivalent.
  \begin{enumerate}[label=\((\roman*)\)]
    \item \(\cC\) is Abelian and its short exact sequences are admissible.
    \item Every morphism of \(\cC\) is \(\cC\)-admissible.
  \end{enumerate}
\end{citedprop}
\begin{citeddef}[{\cite[D\'efinition 1.2.5]{BBDG18}}]
  Let \(\cD\) be a triangulated category and let \(\cC\) be a full subcategory
  of \(\cD\).
  We say that \(\cC\) is an \textsl{admissible Abelian subcategory} of \(\cD\)
  if the following conditions hold.
  \begin{enumerate}[label=\((\roman*)\)]
    \item \(\Hom(X,Y[i]) = 0\) for all \(i < 0\) and \(X,Y\) in \(\cC\).
    \item \(\cC\) is stable under finite direct sums.
    \item \(\cC\) satisfies one of the equivalent conditions in Proposition
      \ref{prop:bbdg124}.
  \end{enumerate}
\end{citeddef}
\begin{citedthm}[{\cite[Th\'eor\`eme 1.3.6]{BBDG18}}]\label{thm:bbdg136}
  Let \(\cD\) be a \(t\)-category.
  The heart \(\cD^\heartsuit\) of \(\cD\) is an admissible Abelian subcategory
  of \(\cD\) that is stable under extensions.
  The functor
  \[
    \cH^0 \coloneqq \tau_{\ge0}\,\tau_{\le0} \colon \cD \longrightarrow
    \cD^\heartsuit
  \]
  is a cohomological functor.
\end{citedthm}
We also define \(t\)-exact functors.
\begin{citeddef}[{\cite[(1.3.16)]{BBDG18}}]
  Let \(\cD_1,\cD_2\) be two \(t\)-categories.
  Let \(T\colon \cD_1 \to \cD_2\) be an exact functor of the underlying
  triangulated categories.
  \begin{enumerate}[label=\((\roman*)\)]
    \item We say that \(T\) is \textsl{right \(t\)-exact} if \(T(\cD_1^{\le0}) \subseteq
      \cD_2^{\le 0}\).
    \item We say that \(T\) is \textsl{left \(t\)-exact} if \(T(\cD_1^{\ge0}) \subseteq
      \cD_2^{\ge 0}\).
    \item We say that \(T\) is \textsl{\(t\)-exact} if \(T\) is both right
      \(t\)-exact and left \(t\)-exact.
  \end{enumerate}
\end{citeddef}
\begin{citedprop}[{\cite[Proposition 1.3.17\((i)\)]{BBDG18}}]
  Let \(\cD_1,\cD_2\) be two \(t\)-categories.
  Denote the inclusion functors by \(\varepsilon\colon \cD_i^\heartsuit
  \hookrightarrow \cD_i\).
  Let \(T\colon \cD_1 \to \cD_2\) be an exact functor of the underlying
  triangulated categories.
  If \(T\) is left (resp.\ right) \(t\)-exact, then the additive functor
  \[
    \prescript{\fp}{}{T} \coloneqq \cH^0 \mathop{\circ} T \mathop{\circ} \varepsilon\colon
    \cD_1^\heartsuit \longrightarrow \cD_2^\heartsuit
  \]
  is left (resp.\ right) exact.
\end{citedprop}

\subsection{Perverse sheaves}
We now define perverse sheaves with coefficients in rings such as \(\ZZ_\ell\)
or \(\QQ_\ell\) using the formalism of
\(\ell\)-adic complexes from \citeleft\citen{Eke90}\citepunct
\citen{Far09}\citemid \S5\citepunct \citen{ILO14}\citemid Expos\'e XIII,
\S4\citeright\ or \cite{BS15}, which match in our context by \cite[Proposition
5.5.4]{BS15}.
As noted in \cite[Remark 2.1.1]{Mor25}, this adic formalism extends to all
morphisms of finite type between excellent Noetherian
\(\ZZ[\frac{1}{\ell}]\)-schemes that are separated and of
finite Krull dimension.
In this situation, we have four functors \(f^*,f_*,f_!,f^!\) defined on
\(\DD^b_c\) satisfying the adjunctions
\[
  f^* \dashv f_* \quad\quad\text{and}\quad\quad f_! \dashv f^!.
\]
When the hypotheses of \cite[Expos\'e XVII, Th\'eor\`eme 0.2]{ILO14} hold, we
also have a dualizing functor \(\mathbf{D}_X\) interchanging \(f^*,f^!\) and
\(f_*,f_!\) by \cite[Expos\'e I, Proposition 1.12]{SGA5}.\medskip
\par See \cite[p.\ 198]{Eke90} for the definition of essentially zero \(\ell\)-adic
complexes used below.
\begin{citeddef}[{\citeleft\citen{Gab04}\citemid \S2\citepunct
  \citen{Far09}\citemid \S5.11\citepunct
  \citen{ILO14}\citemid Expos\'e XIII, \S4.1\citepunct
  \citen{Mor25}\citemid \S2.2\citeright}]
  Let \(\ell > 0 \) be a prime number.
  Let \(E\) be a finite extension of \(\QQ_\ell\) and let
  \(\cO\) be its ring of integers with uniformizer \(\varpi\).
  Let \(X\) be a Noetherian separated
  scheme and consider a strong perversity function
  \(\fp\) on \(X\).
  For every integer \(n\),
  we define strictly full subcategories \(\Dp^{\le n}\) and \(\Dp^{\ge n}\)
  of \(\DD^b_c(X,\cO)\) (resp.\ \(\DD^b_c(X,E) \coloneqq
  \DD^b_c(X,\cO) \otimes_{\cO} E\)) as follows.
  \begin{enumerate}[label=\((\roman*)\)]
    \item \(\Dp^{\le n}\)
      consists of all objects \(F^\bullet\) of
      \(\DD^b_c(X,\cO)\) (resp.\ \(\DD^b_c(X,E)\))
      such that for every \(x \in X\), 
      \(\cH^i(i_{\bar{x}}^*F^\bullet)\) is essentially zero (resp.\ zero)
      for all \(i > \fp(x) + n\).
    \item \(\Dp^{\ge n}\)
      consists of all objects \(F^\bullet\) of \(\DD^b_c(X,\cO)\)
      (resp.\ \(\DD^b_c(X,E)\)) such
      that for every \(x \in X\), \(\cH^i(i_{\bar{x}}^!F^\bullet)\) is
      essentially zero (resp.\ zero)
      for all \(i < \fp(x) + n\).
  \end{enumerate}
\end{citeddef}
We state the following result which says that in many cases, \((\Dp^{\le
0},\Dp^{\ge 0})\) defines a \(t\)-structure.
The statement in \cite{Far09} assumes the existence of a dualizing complex
for \(\DD^b_c(X,\cO/\varpi)\) in the sense of
\citeleft\citen{SGA5}\citemid Expos\'e I, D\'efinition 1.7\citepunct
\citen{ILO14}\citemid Expos\'e XVII, D\'efinition 0.1\citeright.
This assumption on \(\DD^b_c(X,\cO/\varpi)\) holds automatically for \(X\) as in
the statement below by \cite[Expos\'e XVII, Th\'eor\`eme 0.2]{ILO14}.
\par Below and in the rest of this paper, we set
\[
  \prescript{\fp}{}{\mathcal{H}}^i(\,\cdot\,) \coloneqq
  \mathcal{H}^0\bigl(\,\cdot\,[i]\bigr)
\]
where \(\mathcal{H}^0\) is computed in terms of
the \(t\)-structure defined by \(\fp\) (see Theorem \ref{thm:bbdg136}).
The superscript \(\fp\) on \(\prescript{\fp}{}{\mathcal{H}}^i\) is written to
differentiate this cohomology sheaf from the usual cohomology sheaf computed
using the ordinary \(t\)-structure.
\begin{citedthm}[{\citeleft\citen{Gab04}\citemid
  \S6 and Theorem 8.3\citepunct
  \citen{Far09}\citemid Th\'eor\`eme 5.29\citepunct
  \citen{Mor25}\citemid \S2.2 and Theorem 2.2.3\citeright\ (cf.\
  \citeleft\citen{BBDG18}\citemid \S2.2, Th\'eor\`eme 4.3.1\((i)\),
  and \((b)\) on p.\ 101\citeright)}]\label{thm:perv}
  Let \(\ell > 0 \) be a prime number.
  Let \(E\) be a finite extension of \(\QQ_\ell\) and let
  \(\cO\) be its ring of integers with uniformizer \(\varpi\).
  Let \(X\) be an equidimensional excellent Noetherian
  \(\ZZ[\frac{1}{\ell}]\)-scheme that is separated and of
  finite Krull dimension and consider the middle perversity
  \(\fp \coloneqq \fp_{1/2}\) on \(X\).
  Then, \((\Dp^{\le 0},\Dp^{\ge 0})\) defines a t-structure on
  \(\DD^b_c(X,\cO)\) (resp.\ \(\DD^b_c(X,E)\)).
  The heart
  \begin{align*}
    \Perv(X,\cO) &\coloneqq \Dp^{\le 0} \cap \Dp^{\ge0}
    \intertext{of the \(t\)-structure on \(\DD^b_c(X,\cO)\) is
    Noetherian.
    The heart}
    \Perv(X,E) &\coloneqq \Dp^{\le 0} \cap \Dp^{\ge0}
  \end{align*}
  of the \(t\)-structure on \(\DD^b_c(X,E)\) is both Artinian and
  Noetherian with simple objects of the form
  \[
    j_{!*}L[d] \coloneqq \im\Bigl(\prescript{\fp}{}{\cH}^0j_!L[d] \longrightarrow
    \prescript{\fp}{}{\cH}^0j_*L[d]\Bigr)
  \]
  where
  \(j\colon Z \hookrightarrow X\) is a locally closed immersion of a connected
  regular subscheme \(Z\) of dimension \(d\) and \(L\) is a lisse sheaf on
  \(Z\) corresponding to an irreducible representation of \(\pi_1^\et(Z)\).
\end{citedthm}
We can now define perverse sheaves as the objects living in
the categories \(\Perv(X,\cO)\) and \(\Perv(X,E)\) appearing in
Theorem \ref{thm:perv}.
\begin{citeddef}[{\citeleft\citen{BBDG18}\citemid p.\ 102\citepunct
  \citen{Gab04}\citemid \S7\citepunct
  \citen{Mor25}\citemid \S2.2\citeright}]
  Let \(\ell > 0 \) be a prime number.
  Let \(E\) be a finite extension of \(\QQ_\ell\) and let
  \(\cO\) be its ring of integers.
  Let \(X\) be an equidimensional excellent Noetherian
  \(\ZZ[\frac{1}{\ell}]\)-scheme that is separated and of
  finite Krull dimension and consider the middle perversity
  \(\fp \coloneqq \fp_{1/2}\) on \(X\).
  The hearts \(\Perv(X,\cO)\) and \(\Perv(X,E)\) of the
  \(t\)-structures
  appearing in Theorem
  \ref{thm:perv} are called the categories of 
  \textsl{perverse sheaves} with coefficients in \(\cO\) and \(E\),
  respectively.
\end{citeddef}
\begin{remark}\label{rem:gabbertruncate}
  For future reference, we recall Gabber's construction of the truncation
  functors \(\prescript{\fp}{}{\tau}_{\le0}\) from \cite[p.\ 712 and
  \S6]{Gab04}, which he
  used to prove Theorem \ref{thm:perv}.
  \begin{enumerate}[label=\((\arabic*)\),
    ref={\ref*{rem:gabbertruncate}.\arabic*}]
    \item Let \(F^\bullet\) be a complex that is bounded below.
      We can then construct the Godement resolution \(C(F^\bullet)\)
      of \(F^\bullet\) by taking the total complex
      \[
        C(F^\bullet) \coloneqq \Tot\bigl(C^\bullet(F^\bullet)\bigr)
      \]
      as in
      \citeleft\citen{SGA43}\citemid Expos\'e XVII, (4.2.9)\citepunct
      \citen{Gab04}\citemid p.\ 712\citeright,
      which works for \(\ell\)-adic complexes by
      \cite[pp.\ 129--130]{FK88}.
    \item We iterate this construction.
      Denote by \(\epsilon_{F^\bullet}\colon F^\bullet \to C(F^\bullet)\) the
      augmentation map.
      For every integer \(n \ge 1\), we then inductively define the augmentation
      maps
      \[
        \epsilon_{C^{(n)}(F^\bullet)}\colon
        C^{(n)}(F^\bullet) \longrightarrow C^{(n+1)}(F^\bullet).
      \]
      We can also iterate over ordinals by setting
      \begin{align*}
        C^{(\alpha+1)}(F^\bullet) &\coloneqq
        C\bigl(C^{(\alpha)}(F^\bullet)\bigr)
        \intertext{for successor ordinals \(\alpha+1\) and}
        C^{(\lambda)}(F^\bullet) &\coloneqq \varinjlim_{\alpha <\lambda}
        C^{(\alpha)}(F^\bullet)
      \end{align*}
      for limit ordinals \(\lambda\).
      The only infinite ordinal we will use in this construction is
      the first limit ordinal \(\omega\).
      Note that all terms in these iterated Godement resolutions are flasque.
    \item Let \(c \coloneqq -{\dim(X)}\), which is a lower bound for
      \(\fp_{1/2}\).
      For every \(d \ge c\), consider the strong perversity function
      \(\fp_d\coloneqq\min\{d,\fp\}\).
      We inductively define
      \(\tau_{\le \fp_d}F^\bullet\) as a
      subcomplex of \(C^{(d-c)}(F^\bullet)\) as follows.
      For \(d = c\), we have \(\fp_c = c\), and set
      \[
        \tau_{\le \fp_c}F^\bullet \coloneqq \tau_{\le
        c}F^\bullet.
      \]
      For the inductive step, we need to see that if
      \(\tau_{\le \fp_d}F^\bullet\) is a subcomplex of
      \(F^\bullet\), then we
      can define \(\tau_{\le \fp_{d+1}}F^\bullet\) as a
      subcomplex of \(C(F^\bullet)\) containing the image of
      \(C(\tau_{\le \fp_d}F^\bullet)\) in the commutative
      diagram
      \[
        \begin{tikzcd}
          0 \rar
          & \tau_{\le \fp_d}F^\bullet \rar \dar[hook]
          & F^\bullet \rar\dar[hook]
          & F^\bullet/\tau_{\le \fp_d}F^\bullet
          \rar
          & 0\\
          0 \rar
          & C(\tau_{\le \fp_d}F^\bullet) \rar
          & C(F^\bullet)
        \end{tikzcd}
      \]
      of complexes.
      It suffices to construct \(\tau_{\le \fp_{d+1}}\) for
      \(F^\bullet/\tau_{\le \fp_d}F^\bullet\).
      We therefore assume \(F^\bullet\) lies in
      \(\prescript{\fp_d}{}{\cD}^{\ge1} = \cD^{\ge
      \fp_d+1}\).
      Consider the ind-constructible set
      \[
        \Phi\coloneqq \Set[\big]{x \in X \given \fp_{d+1}(x) = d+1}.
      \]
      We then set
      \[
        \tau_{\fp_{d+1}}F^\bullet \coloneqq
        \tau_{d+1}
        \underline{\Gamma}_\Phi\bigl(C(F^\bullet)\bigr) \in \cD^{\le \fp_{d+1}},
      \]
      which is quasi-isomorphic to \(\sH_\Phi^{d+1}(F^\bullet)[-d-1]\) in the
      derived category.
    \item We set
      \[
        \prescript{\fp}{}{\tau}_{\le0} F^\bullet =
        \tau_{\le \fp} F^\bullet \coloneqq \varinjlim_{d \ge c}
        \tau_{\le \fp_{d}} F^\bullet
      \]
      which is a subcomplex of \(C^{(\omega)}(F^\bullet)\) by construction.
    \item Suppose\label{rem:gabbertruncatemodgodement} that \(F^\bullet\) is the
      direct limit of constructible complexes \(F^\bullet_\alpha\).
      Then, we can use the modified Godement resolution
      \[
        C_\ell(F^\bullet) \coloneqq \varinjlim_\alpha C(F^\bullet_\alpha)
      \]
      from \cite[Expos\'e XVIII, \S3.1]{SGA43} instead of \(C(F^\bullet)\) to
      make this construction compatible with direct limits.
  \end{enumerate}
\end{remark}

\subsection{Pullback along a regular morphism}
Pullback along a regular morphism, in particular, base change to the completion
on a locally excellent scheme, is \(t\)-exact with respect to the perverse
\(t\)-structure after taking an appropriate shift.
The case of base change to the completion is due to Fargues \cite[Th\'eor\`eme
6.1]{Far09}.
\begin{theorem}\label{thm:fargues}
  Let \(\ell > 0\) be a prime number.
  Let \(E\) be a finite extension of \(\QQ_\ell\) and let
  \(\cO\) be its ring of integers.
  Consider a regular morphism \(f\colon Y \to X\) between
  equidimensional excellent Noetherian
  \(\ZZ[\frac{1}{\ell}]\)-schemes that are separated and
  of finite Krull dimension.
  Then, the functors
  \begin{alignat*}{3}
    f^*\bigl[\dim(Y)-\dim(X)\bigr] &\colon \DD^b_c(X,\cO) &{}\longrightarrow{}&
    \DD^+_c(X,\cO)\\
    f^*\bigl[\dim(Y)-\dim(X)\bigr] &\colon \DD^b_c(X,E) &{}\longrightarrow{}&
    \DD^+_c(X,E)
  \end{alignat*}
  are \(t\)-exact with respect to the perverse \(t\)-structure.
\end{theorem}
\begin{proof}
  Since the question is local,
  we can replace \(X\) and \(Y\) by affine open subsets to assume that both
  \(X = \Spec(A)\) and \(Y = \Spec(B)\) are affine.
  \par For every \(y \in Y\), we have
  \begin{align*}
    \fp_{1/2}\bigl(f(y)\bigr) &= \dim(\cO_{X,f(y)}) - \dim(X)\\
    &\le \dim(\cO_{Y,y}) - \dim(Y) + \bigl(\dim(Y) - \dim(X)\bigr)\\
    &= \fp_{1/2}(y) + \bigl(\dim(Y) - \dim(X)\bigr).
  \end{align*}
  Since \(f^*\) is exact with respect to the ordinary \(t\)-structure,
  we see that
  \[
    f^*\bigl(\prescript{\fp}{}{\mathcal{D}}^{\le0}\bigr) \subseteq
    \prescript{\fp}{}{\mathcal{D}}^{\le{\dim(Y)-\dim(X)}}
  \]
  that is, \(f^*[\dim(Y)-\dim(X)]\) is right \(t\)-exact with respect to the
  perverse \(t\)-structure.\smallskip
  \par It remains to show that
  \(f^*[\dim(Y)-\dim(X)]\) is left \(t\)-exact with respect to the
  perverse \(t\)-structure.
  Let \(F^\bullet\) lie in \(\prescript{\fp}{}{\mathcal{D}}^{\ge0}\).
  By \cite[Lemma 3.1]{Gab04}, it suffices to show that
  \[
    \sH^i_{\overline{\{\fq\}}}\bigl(f^*F^\bullet\bigr) = 0
  \]
  for all \(i \le {\dim(Y)-\dim(X)}\) and for all \(\fq \in Y\).
  Since this vanishing is a local condition, it suffices by \cite[Proposition
  5.2]{Gab04} (see also \cite[p.\ 286]{Far09}) that
  \[
    H^i_{a^{-1}(\overline{\{\fq\}})}\bigl(a^*f^*F^\bullet\bigr) = 0
  \]
  for all \(i \le {\dim(Y)-\dim(X)}\), for all \(\fq \in Y\), and for all \(y
  \in U\) where \(a\colon U \to Y\) is \'etale and \(a(y) \in
  \overline{\{\fq\}}\).
  By N\'eron--Popescu desingularization
  \citeleft\citen{Pop86}\citemid Theorem 2.4\citepunct \citen{Pop90}\citemid p.\
  45\citepunct\citen{Swa98}\citemid Theorem 1.1\citeright,
  there exists a direct system \(\{B_\lambda\}_{\lambda \in \Lambda}\)
  of smooth \(A\)-algebras such that
  \[
    B \cong \varinjlim_{\lambda \in \Lambda} B_\lambda.
  \]
  Setting \(W_\lambda = \Spec(B_\lambda)\), we have the factorization
  \[
    \begin{tikzcd}[column sep=tiny]
      Y \arrow{rr}{f} \arrow{dr}[swap]{h_\lambda} & & X\\
      & W_\lambda\mathrlap{.} \arrow{ur}[swap]{\vphantom{h}g_\lambda}
    \end{tikzcd}
  \]
  We denote by \(d_\lambda=\dim(B_\lambda)-\dim(A)\) the relative dimension of \(g_\lambda\).
  \par Let \(\fq \in Y\), let \(a\colon U \to Y\) be an \'etale morphism, and let \(y
  \in U\) such that \(a(y) \in \overline{\{\fq\}}\).
  There exists \(\lambda_0\) such that setting \(\fp_\lambda = h_\lambda(\fq)\), we have
  \(\fp_\lambda B = \fq\) for all \(\lambda \ge \lambda_0\).
  By \cite[Expos\'e VII, Lemme 5.6]{SGA43}, there exists \(\lambda_1 \ge \lambda_0\) and an
  \'etale morphism \(a_{\lambda_1}\colon U_{\lambda_1} \to W_{\lambda_1}\)
  fitting into the Cartesian
  diagram
  \[
    \begin{tikzcd}
      U \rar{a}\dar & Y\dar{h_\lambda}\\
      U_{\lambda_1} \rar{a_{\lambda_1}} & W_{\lambda_1}\mathrlap{.}
    \end{tikzcd}
  \]
  Denote by \(U_\lambda \to W_\lambda\) the base change of \(a_{\lambda_1}\)
  along \(W_\lambda \to W_{\lambda_1}\).
  By Gabber's version of Grothendieck's limit theorem for local cohomology
  \cite[Proposition 5.2]{Gab04}, we then have
  \[
    H^i_{a^{-1}(\overline{\{\fq\}})}\bigl(a^*f^*F^\bullet\bigr)
    = \varinjlim_{\lambda \ge \lambda_1} H^i_{a_\lambda^{-1}(\overline{\{\fp_\lambda\}})}
    \bigl(a_\lambda^*\,g_\lambda^*\,F^\bullet\bigr).
  \]
  Note that
  \begin{align*}
    H^i_{a_\lambda^{-1}(\overline{\{\fp_\lambda\}})}
    \bigl(a_\lambda^*\,g_\lambda^*\,F^\bullet\bigr) &= 0
    \intertext{for all}
    i < \fp_{1/2}(\fp_\lambda) + d_\lambda
    = \dim\bigl((B_{\lambda})_{\fp_\lambda}\bigr) &- \dim(B_\lambda) + d_\lambda
  \end{align*}
  since \(g_\lambda^*[d_\lambda]\) is
  \(t\)-exact \citeleft\citen{BBDG18}\citemid Proposition 4.2.5\citepunct
  \citen{Mor25}\citemid Proposition 2.2.2\((v)\)\citeright.
  Thus, it suffices to show there exists \(\lambda_2 \ge \lambda_1\) such that
  \begin{align*}
    \MoveEqLeft[3]
    \dim\bigl((B_{\lambda})_{\fp_\lambda}\bigr) - \dim(B_\lambda) + d_\lambda\\
    &\ge \fp_{1/2}(\fq) + \bigl(\dim(B) - \dim(A)\bigr)\\\
    &= \bigl(\dim(B_\fq) - \dim(B)\bigr) + \bigl(\dim(B) - \dim(A)\bigr)
  \end{align*}
  for all \(\lambda \ge \lambda_2\).
  Since \(d_\lambda = \dim(B_\lambda) - \dim(A)\), it suffices to show there
  exists \(\lambda_2 \ge \lambda_1\) that
  \begin{equation}\label{eq:dimapproxbigger}
    \dim\bigl((B_{\lambda})_{\fp_\lambda}\bigr) \ge \dim(B_\fq)
  \end{equation}
  for all \(\lambda \ge \lambda_2\).
  Consider a strictly ascending chain of prime ideals in \(B_\fq\) of
  maximal length.
  Then, there exists \(\lambda_2 \ge \lambda_1\) for which this chain contracts
  to a strictly ascending chain of prime ideals in \((B_{\lambda})_{\fp_\lambda}\).
  For such a choice of \(\lambda_2\), the inequality \eqref{eq:dimapproxbigger}
  holds for all \(\lambda \ge \lambda_2\).
\end{proof}

\subsection{Exceptional pullback for morphisms between regular schemes}
We define exceptional pullbacks for morphisms betweeen equidimensional excellent
regular Noetherian \(\ZZ[\frac{1}{\ell}]\)-schemes.
\begin{definition}\label{def:cfpullback}
  Let \(\ell > 0\) be a prime number and
  let \(E\) be a finite extension of \(\QQ_\ell\).
  Let \(f\colon Y \to X\) be a morphism between equidimensional excellent
  regular Noetherian \(\ZZ[\frac{1}{\ell}]\)-schemes that are separated and of
  finite Krull dimension.
  We define the functor
  \[
    f^! \coloneqq \mathbf{D}_Y \mathop{\circ} f^* \mathop{\circ} \mathbf{D}_X
  \]
  using the dualizing complexes \(E[2\dim(X)]({\dim(X)})\)
  and \(E[2\dim(Y)]({\dim(Y)})\) on \(X\) and
  \(Y\), respectively.
  Note that \(E\) and its shifts are dualizing complexes by
  \citeleft\citen{ILO14}\citemid
  Expos\'e XVII, Th\'eor\`eme 0.2\citepunct \citen{BS15}\citemid Lemma
  6.7.20\citeright\ since \(X\) and \(Y\) are regular.
  This definition matches the usual definition for finite type morphisms by
  \cite[Expos\'e I, Proposition 1.12\((b)\)]{SGA5}.
\end{definition}
\begin{remark}\label{rem:cfpullbackql}
  \par The definition in Definition \ref{def:cfpullback} is chosen so that
  setting \(d_X \coloneqq \dim(X)\) and \(d_Y \coloneqq \dim(Y)\), we have
  \begin{align*}
    \MoveEqLeft[3]
    f^!
    [d_X-d_Y](d_X-d_Y)\,E[d_X]\\
    &= \RRHHom\Bigl(f^*\RRHHom\bigl(E[d_X],
    E[2d_X](d_X)\bigr),
    E[2d_Y](d_Y)\Bigr)
    [d_X-d_Y](d_X-d_Y)\\
    &\cong E[d_Y].
  \end{align*}
\end{remark}

\section{Algebraic \texorpdfstring{\(\mathscr{D}\)}{D}-modules and
the Riemann--Hilbert
correspondence}\label{sect:algdmod}
We review some aspects of the theory of \(\mathscr{D}\)-modules from an
algebraic point of view, mostly following \cite{Ber83,BGKHME87,HTT08}.
\subsection{Differential operators and \texorpdfstring{\(\mathscr{D}\)}{D}-modules}
\par For a scheme morphism \(X \to S\), we denote by \(\sD_{X/S}\) the
\textsl{sheaf of \(S\)-differential operators on \(X\)} as defined in
\citeleft\citen{EGAIV4}\citemid D\'efinition 16.8.1 and Corollaire
16.8.10\citepunct
\citen{SGA3}\citemid Expos\'e VII\textsubscript{A}, D\'efinition
1.4\citeright.
The sheaf \(\sD_{X/S}\) is a sheaf of (non-commutative) rings.
By a \textsl{\(\sD_{X/S}\)-module}, we will mean a left module over \(\sD_{X/S}\).
\par When \(X\) is a smooth variety over a field \(k\) of characteristic zero,
this definition matches
the definitions in \citeleft\citen{Ber83}\citemid Lecture 1, \S2\citepunct
\citen{BGKHME87}\citemid Chapter VI, \S1\citepunct \citen{HTT08}\citemid
\S1.1\citeright, and we set
\[
  \sD_X \coloneqq \sD_{X/{\Spec(k)}}.
\]
We will use the notions of \textsl{holonomic} and \textsl{regular holonomic}
\(\sD\)-modules.
In the analytic context, these notions were defined in \cite{Kas75,KK81}.
Algebraic versions of these definitions appear in \cite{Ber83,BGKHME87,HTT08}.
The analytic and algebraic definitions are compatible when \(k = \CC\)
by \cite[\S VII]{Bry86}.
\par For smooth varieties over a field \(k\) of characteristic zero,
we use the notation and conventions for
shifts from \cite{HTT08}.
For morphisms \(f\colon Y \to X\) between smooth varieties over \(k\), we have four
functors \(f^\bigstar,\int_f,\int_{f!},f^\dagger\) satisfying the adjunctions
\[
  f^\bigstar \dashv \int_f \qquad \text{and} \qquad \int_{f!} \dashv f^\dagger
\]
on categories of holonomic or regular holonomic
\(\sD\)-modules \citeleft\citen{Ber83}\citemid Lecture 3, \S9\citepunct
\citen{BGKHME87}\citemid Chapter VII, Theorem 10.2\citepunct
\citen{HTT08}\citemid Theorems 3.2.14 and
6.1.5\citeright.
For example, the (shifted) inverse image functor is
\begin{equation}\label{eq:deffdagger}
  f^\dagger\coloneqq \LL f^*\bigl[\dim(Y)-\dim(X)\bigr]
\end{equation}
where \(\LL f^*\) is the derived inverse image functor as \(\cO\)-modules \cite[p.\
33]{HTT08}.\medskip
\par Local cohomology is an important example of a \(\sD\)-module.
\begin{example}[Local cohomology {\citeleft\citen{Kas70}\citemid
  \S2.1\citeright}]\label{ex:lcisrh}
  Let \(X\) be a smooth variety over a field of characteristic zero
  and let \(Z \subseteq X\) be a locally
  closed subset.
  For every \(i\), the (algebraic) local cohomology functor \(\sH^i_Z(\,\cdot\,)\)
  sends \(\sD_X\)-modules to \(\sD_X\)-modules.
  Its associated derived functor \(\RRGGamma_Z\) is a functor on the bounded derived
  category of \(\sD_X\)-modules.
  These functors preserve strong finiteness properties: They preserve
  holonomicity \citeleft\citen{Meb77}\citemid Theorem 1.1\((i)\)\citepunct
  \citen{Kas78}\citemid \S1\citeright\ and regular holonomicity \cite[Theorem
  5.4.1]{KK81}.
  See \citeleft\citen{Ber83}\citemid Lecture 2, \S2\citepunct
  \citen{BGKHME87}\citemid Chapter VI, (7.9)\citepunct \citen{HTT08}\citemid
  \S1.7\citeright\ for some discussion in the algebraic case.
  The algebraic versions of the analytic results cited above also follow from the
  analytic versions \cite[\S VII]{Bry86}.
  \par Since \(\cO_X\) is regular holonomic \cite[Chapter VII, Corollary
  11.8]{BGKHME87} and the subcategory of regular holonomic \(\sD_X\)-modules is
  stable under subobjects, quotients, and extensions, we see that
  local cohomology sheaves \(\sH^i_Z(\cO_X)\) and the sheaves \(T(\cO_X)\)
  obtained by applying a Lyubeznik functor to \(\cO_X\) are
  regular holonomic \(\sD_X\)-modules for closed supports \(Z \subseteq X\).
\end{example}
\subsection{The Riemann--Hilbert correspondence}
Let \(X\) be a complex manifold.
The Riemann--Hilbert correspondence, due to Kashiwara
\cite{Kas75,Kas80,Kas84} and Mebkhout \cite{Meb80,Meb82,Meb84a,Meb84b}, states that the
de Rham functor
\[
  \begin{tikzcd}[row sep=0,column sep=1.475em]
    \DD^b_{\rh}(\sD_X) \rar & \DD^b_c(X,\CC)\\
    M^\bullet \rar[mapsto] &
    \Omega_X \otimes^{\mathrm{L}}_{\sD_X} M^\bullet
  \end{tikzcd}
\]
is an equivalence of categories that is \(t\)-exact with respect to the ordinary
\(t\)-structure on the left and the perverse \(t\)-structure on the right.
Here, \(\DD^b_{\rh}(\sD_X)\) is the bounded derived category of
\(\sD_X\)-modules with regular holonomic cohomology sheaves and
\(\DD^b_c(X,\CC)\) is the bounded derived category of sheaves of \(\CC\)-vector
spaces with constructible cohomology sheaves.\medskip
\par The algebraic version of this correspondence is the following result due to
Be\u{\i}linson and Bernstein \citeleft\citen{Ber83}\citemid Lecture 5\citepunct
\citen{BGKHME87}\citemid Chapter VIII\citeright, which can also
be deduced from the analytic version \cite[\S VII]{Bry86} (see also
\citeleft\citen{Meb89}\citemid Chapitre II, \S8\citepunct
\citen{HTT08}\citemid \S4.7\citeright).
Again, \(\DD^b_{\rh}(\sD_X)\) is the bounded derived category of
\(\sD_X\)-modules with regular holonomic cohomology sheaves and
\(\DD^b_c(X^\an,\CC)\) is the bounded derived category of sheaves of \(\CC\)-vector
spaces with constructible cohomology sheaves on the analytification \(X^\an\) of
\(X\).
\begin{theorem}[The Riemann--Hilbert correspondence for algebraic \(\sD\)-modules]
  \label{thm:riemannhilbert}
  Let \(X\) be a smooth complex variety.
  Then, there is a equivalence of categories
  \[
    \begin{tikzcd}[row sep=0,column sep=1.475em]
      \DR_X\colon &[-2.35em] \DD^b_{\rh}(\sD_X) \rar & \DD^b_c(X^\an,\CC)\\
      & M^\bullet \rar[mapsto] &
      \Omega_{X^\an} \otimes^{\mathrm{L}}_{\sD_{X^\an}} (M^\bullet)^\an
    \end{tikzcd}
  \]
  that is \(t\)-exact with respect to the ordinary
  \(t\)-structure on the left and the perverse \(t\)-structure on the right.
  This functor is compatible with the four functors
  \(f^\bigstar,\int_f,\int_{f!},f^\dagger\) and \(f^*,f_*,f_!,f^!\) and with
  \(\otimes^\LL\), \(\RRHHom\), and duality.
\end{theorem}
Note that \(\DD^b_c(X^\an,\CC)\) is defined with respect to the Euclidean
topology on the complex manifold \(X^\an\) and not for the \'etale topology on
the variety \(X\).
To move between the two topologies, we have the following.
\begin{citedthm}[{\cite[(6.1.2)]{BBDG18}}]\label{thm:bbdg612}
  Let \(X\) be a complex variety.
  Consider the morphism
  \begin{align*}
    \varepsilon\colon X^\an &\longrightarrow X_\et
  \intertext{of topoi.
  For every prime number \(\ell > 0\), the pullback functor}
    \varepsilon^*\colon \DD^b_c(X,\ZZ_\ell) &\overset{\sim}{\longrightarrow}
    \DD^b_c(X^\an,\ZZ_\ell)
  \end{align*}
  is an equivalence of categories that is \(t\)-exact with respect to the
  ordinary \(t\)-structure and the perverse \(t\)-structures.
  This equivalence commutes with the four functors
  \(f^*,f_*,f_!,f^!\) and with \(\otimes^\LL\) and \(\RRHHom\).
  The functor \(\varepsilon^*\) induces a fully faithful functor for \(\QQ_\ell\)
  and \(\bar{\QQ}_\ell\) coefficients.
\end{citedthm}

\section{Lyubeznik functors of excellent regular
\texorpdfstring{\(\QQ\)}{Q}-algebras}\label{sect:equichar0}
\par While there is no version of the Riemann--Hilbert correspondence available
in our generality, even for complete regular local \(\QQ\)-algebras, we show
that there exists a perverse sheaf encoding the data of the associated
primes of \(\sH^i_Z(\cO_X)\).
In fact, for smooth complex varieties, the perverse sheaf we use (after
extension of scalars) coincides with the
image of \(\sH^i_Z(\cO_X)\) under the Riemann--Hilbert correspondence (Theorem
\ref{thm:riemannhilbert}).
When \(X\) is a smooth complex variety, our main result below (Theorem
\ref{thm:assocaresimple}) says that the associated primes
of \(\sH^i_Z(\cO_X)\) are contained in the set of generic points of supports of
simple components of
\(\DR_X(\sH^i_Z(\cO_X))\).\medskip
\par Since we cannot apply the Riemann--Hilbert correspondence directly, we will
approximate \(X\) by smooth varieties using N\'eron--Popescu desingularization
\citeleft\citen{Pop86}\citemid Theorem 2.4\citepunct \citen{Pop90}\citemid p.\
45\citepunct\citen{Swa98}\citemid Theorem 1.1\citeright\ 
(proved earlier in \cite[Corollary 5.4]{Pop85} for \(\QQ\)-algebras).
\subsection{Approximation}
We will use the following construction throughout this section.
\begin{setup}\label{setup:approx}
  Let \(R\) be an excellent regular \(\QQ\)-algebra of finite Krull dimension
  \(d\) and set \(X \coloneqq \Spec(R)\).
  Let
  \[
    T = T_r \mathop{\circ} \cdots \mathop{\circ} T_2 \mathop{\circ} T_1
  \]
  be
  a Lyubeznik functor on \(X\).
  We approximate \(R\) and \(T\) as follows.
  \begin{enumerate}[label=\((\arabic*)\),ref={\ref*{setup:approx}.\arabic*}]
    \item By N\'eron--Popescu desingularization \citeleft\citen{Pop86}\citemid Theorem
      2.4\citepunct \citen{Pop90}\citemid p.\ 45\citepunct\citen{Swa98}\citemid
      Theorem 1.1\citeright\ (proved earlier in \cite[Corollary 5.4]{Pop85} for
      \(\QQ\)-algebras), we can write
      \[
        R = \varinjlim_{\lambda \in \Lambda} R_\lambda
      \]
      as a direct limit of integral smooth \(\QQ\)-algebras with transition and
      insertion maps
      \begin{align*}
        \varphi_{\lambda\mu}\colon R_\lambda &\longrightarrow R_\mu,\\
        \varphi_\lambda \colon R_\lambda &\longrightarrow R.
      \end{align*}
      We set
      \(X_\lambda \coloneqq \Spec(R_\lambda)\) and
      \(d_\lambda \coloneqq \dim(R_\lambda)\)
      for every \(\lambda \in \Lambda\).
    \item We\label{setup:approx2}
      find \(\lambda_0 \in \Lambda\) and approximate \(T_\lambda\) by Lyubeznik
      functors on \(R_\lambda\) for \(\lambda \ge \lambda_0\) as follows.
      If \(r = 0\), we set \(T_\lambda = \id\) as a functor on the category of
      regular holonomic \(\sD_{X_\lambda}\)-modules or quasi-coherent
      \(\cO_{X_\lambda}\)-modules
      and choose \(\lambda_0 \in \Lambda\)
      arbitrarily.
      Suppose \(r > 0\) and set
      \[
        T' \coloneqq T_{r-1} \mathop{\circ} \cdots \mathop{\circ} T_2
        \mathop{\circ} T_1.
      \]
      \begin{enumerate}[label=\((\roman*)\),ref=\roman*]
        \item
          Suppose\label{setup:approx2i} \(T_r = \sH^{i}_{Y}(\,\cdot\,)\) for a locally
          closed subset \(Y\) of \(X\).
          Decompose the inclusion \(Y
          \hookrightarrow X\) as the composition
          \[
            Y \overset{i_Y}{\hooklongrightarrow} U
            \overset{j_U}{\hooklongrightarrow} X
          \]
          of a closed immersion followed by an open immersion.
          By \cite[Proposition 8.6.3]{EGAIV3},
          after possibly replacing \(\lambda_0\) by a larger element in
          \(\Lambda\), we can find an open subset \(j_{U_{\lambda_0}}\colon U_{\lambda_0}
          \hookrightarrow
          X_{\lambda_0}\) such that
          \begin{align*}
            U &\cong U_{\lambda_0} \times_{X_{\lambda_0}} X.
          \intertext{After possibly replacing \(\lambda_0\) again by a larger element in
          \(\Lambda\), we can find a closed subset \(i_{Y_{\lambda_0}}\colon Y_{\lambda_0}
          \hookrightarrow
          U_{\lambda_0}\) such that}
            Y &\cong Y_{\lambda_0} \times_{U_{\lambda_0}} U.
          \end{align*}
          For every \(\lambda \ge \lambda_0\), denote by
          \(j_{U_\lambda}\colon U_\lambda \hookrightarrow X_\lambda\) and
          \(i_{Y_\lambda}\colon Y_\lambda \hookrightarrow U_\lambda\)
          the pullbacks of \(j_{\lambda_0}\) and \(i_{\lambda_0}\) to
          \(X_\lambda\) and \(U_\lambda\), respectively.
          We set
          \[
            T_{\lambda,r} \coloneqq \cH^i \mathop{\circ}
            \RRGGamma_{Y_\lambda}
          \]
          as a functor on the category of regular holonomic
          \(\sD_{X_\lambda}\)-modules.
          As a functor
          on the category of quasi-coherent \(\cO_{X_\lambda}\)-modules,
          we have \(T_{\lambda,r} = \sH^i_{Y_\lambda}(\,\cdot\,)\).
        \item
          Suppose\label{setup:approx2ii} \(T_r\) is the kernel of a morphism in the long exact sequence
          \[
            \cdots \longrightarrow \sH^i_{Y'}(\,\cdot\,) \longrightarrow
            \sH^i_{Y}(\,\cdot\,)
            \longrightarrow \sH^i_{Y-Y'}(\,\cdot\,) \longrightarrow \cdots
          \]
          from \citeleft\citen{SGA2}\citemid Expos\'e I, Th\'eor\`eme
          2.8\citepunct \citen{Kas78}\citemid (1.2.6)\citeright\ 
          where \(Y\) is a locally closed subset of \(X\) and \(Y'\) is a
          closed subset of \(Y\).
          With notation as in \((\ref{setup:approx2i})\), we also denote by
          \(i\colon Y' \hookrightarrow Y\) the closed inclusion and \(j\colon Y
          - Y' \hookrightarrow Y\) the open inclusion of its complement in \(Y\).
          After possibly replacing \(\lambda_0\) by a larger element in
          \(\Lambda\), we can find a closed subset \(i_{\lambda_0}\colon Y'_{\lambda_0}
          \hookrightarrow Y_{\lambda_0}\) with complement \(j_{\lambda_0}\colon
          Y_{\lambda_0} - Y'_{\lambda_0} \hookrightarrow Y_{\lambda_0}\) such that
          \begin{align*}
            Y' &\cong Y'_{\lambda_0} \times_{Y_{\lambda_0}} Y,\\
            Y - Y' &\cong \bigl(Y_{\lambda_0} - Y'_{\lambda_0}\bigr)
            \times_{Y_{\lambda_0}} Y.
          \end{align*}
          For every \(\lambda \ge \lambda_0\), denote by
          \(i_\lambda\colon Y'_\lambda \hookrightarrow Y_\lambda\) and
          \(j_\lambda\colon Y_\lambda -Y'_\lambda \hookrightarrow Y_\lambda\)
          the pullbacks of \(i_{\lambda_0}\) and \(j_{\lambda_0}\) to
          \(Y_\lambda\), respectively.
          We then set \(T_{\lambda,r}\) to be
          the kernel of the corresponding morphism in the long exact sequence
          \begin{align}
            \cdots \longrightarrow
            \cH^i\mathop{\circ}\RRGGamma_{Y'_\lambda}
            &\longrightarrow
            \cH^i\mathop{\circ}\RRGGamma_{Y_\lambda}
            \longrightarrow
            \cH^i\mathop{\circ}\RRGGamma_{Y_\lambda - Y'_\lambda}
            \longrightarrow \cdots\nonumber
          \intertext{of functors on the category of regular holonomic
          \(\sD_{X_\lambda}\)-modules.
          As a functor
          on the category of quasi-coherent \(\cO_{X_\lambda}\)-modules,
          \(T_{\lambda,r}\) is the
          kernel of the corresponding morphism in the long exact sequence}
            \cdots \longrightarrow \sH^i_{Y'_\lambda}(\,\cdot\,) &\longrightarrow
            \sH^i_{Y_\lambda}(\,\cdot\,) \longrightarrow \sH^i_{Y_\lambda -
            Y'_\lambda}(\,\cdot\,) \longrightarrow \cdots\label{eq:leswherekercomesfrom}
          \end{align}
      \end{enumerate}
      Finally, we set
      \(T_\lambda \coloneqq T_{r_\lambda} \mathop{\circ} T'_\lambda\).
  \end{enumerate}
\end{setup}
We show that the pullback maps connecting the \(T_\lambda(\cO_{X_\lambda})\)
are morphisms of \(\sD\)-modules, and the direct limit of these pullback
maps is \(T(\cO_X)\).
\begin{proposition}\label{prop:perversecohunderlimdxmod}
  Fix notation as in Setup \ref{setup:approx}.
  For every \(\mu \ge \lambda \ge \lambda_0\),
  the pullback maps
  \begin{equation}\label{eq:pullbackmaprh}
    \prescript{a}{}{\varphi}_{\lambda\mu}^*\,T_\lambda(\cO_{X_\lambda}) \longrightarrow
    T_\mu(\cO_{X_\mu})
  \end{equation}
  are morphisms of \(\sD_{X_\mu}\)-modules.
  Moreover, we have the isomorphisms
  \begin{align}
    T(\cO_X) &\overset{\sim}{\longleftarrow} \varinjlim_{\lambda \ge \lambda_0}
    \Bigl\{
      \Bigl(\prescript{a}{}{\varphi}_\lambda^{-1}
      \,T_\lambda(\cO_{X_\lambda}),
      \prescript{a}{}{\varphi}_\lambda^{-1}
      \,T_\lambda(\cO_{X_\lambda}) \longrightarrow
      \prescript{a}{}{\varphi}_\mu^{-1}
      \,T_\mu(\cO_{X_\mu})\Bigr)
    \Bigr\}\nonumber\\
    &\overset{\sim}{\longrightarrow} \varinjlim_{\lambda \ge \lambda_0}
    \varinjlim_{\mu \ge \lambda}
    \Bigl\{
      \Bigl(\prescript{a}{}{\varphi}_\mu^{-1}
      \prescript{a}{}{\varphi}_{\lambda\mu}^*
      \,T_\lambda(\cO_{X_\lambda}),
      \prescript{a}{}{\varphi}_\mu^{-1}
      \prescript{a}{}{\varphi}_{\lambda\mu}^*
      \,T_\lambda(\cO_{X_\lambda}) \longrightarrow
      \prescript{a}{}{\varphi}_\mu^{-1}
      \,T_\mu(\cO_{X_\mu})\Bigr)
    \Bigr\}\label{eq:toxdirlim}\\
    &\overset{\sim}{\longrightarrow} \varinjlim_{\lambda \ge \lambda_0}
    \Bigl\{
      \Bigl(\prescript{a}{}{\varphi}_{\lambda}^*
      \,T_\lambda(\cO_{X_\lambda}),
      \prescript{a}{}{\varphi}_{\lambda}^*
      \,T_\lambda(\cO_{X_\lambda}) \longrightarrow
      \prescript{a}{}{\varphi}_\mu^{*}
      \,T_\mu(\cO_{X_\mu})\Bigr)
    \Bigr\}\nonumber
  \end{align}
  as Abelian sheaves, whose composition is an isomorphism of
  \(\cO_X\)-modules.
\end{proposition}
\begin{proof}
  The isomorphisms in \eqref{eq:toxdirlim} follow from
  Gabber's version of Grothendieck's limit theorem for local cohomology
  \cite[Proposition 5.2]{Gab04} and by \cite[Chapter 0, Lemma 4.2.7]{FK18} and
  its proof.
  The composition is an isomorphism of \(\cO_X\)-modules by \cite[Theorem
  3.13]{Mur25}.\smallskip
  \par It remains to show that \eqref{eq:pullbackmaprh} is a morphism of \(\sD\)-modules.
  Fix notation as in Setup
  \ref{setup:approx}, where in particular, we denote
  \(T = T_r \mathop{\circ} \cdots
  \mathop{\circ} T_2 \mathop{\circ} T_1\).
  We induce on \(r\).
  If \(r = 0\), the map \eqref{eq:pullbackmaprh} is the identity map on
  \(\cO_{X_\mu}\), and hence there is nothing to show.
  Now suppose that \(r > 0\).
  Factoring \eqref{eq:pullbackmaprh} as
  \[
    \prescript{a}{}{\varphi}_{\lambda\mu}^*\,T_{\lambda,r}
    \bigl(T'_\lambda(\cO_{X_\lambda})\bigr)
    \longrightarrow
    T_{\mu,r}\bigl(\prescript{a}{}{\varphi}_{\lambda\mu}^*\,
    T'_\lambda(\cO_{X_\lambda})\bigr)
    \longrightarrow
    T_{\mu,r}\bigl(
    T'_\mu(\cO_{X_\mu})\bigr)
  \]
  the second map is a morphism of \(\sD_{X_\mu}\)-modules by the inductive
  hypothesis.
  It therefore suffices to show that the first map in this factorization is
  a morphism of \(\sD_{X_\mu}\)-modules.
  We will show that
  \begin{equation}\label{eq:firstmap}
    \prescript{a}{}{\varphi}_{\lambda\mu}^*\mathop{\circ}T_{\lambda,r}
    \longrightarrow
    T_{\mu,r} \mathop{\circ} \prescript{a}{}{\varphi}_{\lambda\mu}^*
  \end{equation}
  is a morphism of functors on \(\DD^b_\rh(\sD_{X_\lambda})\).
  For the proof below, we recall that
  \[
    f^\dagger \coloneqq \LL f^*\bigl[\dim(Y) - \dim(X)\bigr]
  \]
  for morphisms \(f\colon Y \to X\) by definition
  (see \eqref{eq:deffdagger}).
  With this notation, \eqref{eq:pullbackmaprh} is written as
  \[
    \cH^{\dim(X) - \dim(Y)}\prescript{a}{}{\varphi}_{\lambda\mu}^\dagger
    \,T_\lambda(\cO_{X_\lambda}) \longrightarrow
    T_\mu(\cO_{X_\mu}).\smallskip
  \]
  \par We first consider the case when \(T_{\lambda,r}\) is a local cohomology
  functor with support in a closed set \(Y'_\lambda\) or an open subset
  \(X_\lambda - Y'_\lambda\), using notation from Definition
  \ref{def:lyubeznikfunctors} and Setup
  \ref{setup:approx}.
  Consider the isomorphism of distinguished triangles
  \[
    \begin{tikzcd}
      \prescript{a}{}{\varphi}_{\lambda\mu}^\dagger \mathop{\circ}
      \dar[sloped]{\sim}
      \RRGGamma_{Y_\lambda'} \rar
      & \prescript{a}{}{\varphi}_{\lambda\mu}^\dagger \rar \dar[equal]
      & \prescript{a}{}{\varphi}_{\lambda\mu}^\dagger \mathop{\circ}
      \RRGGamma_{X_\lambda - Y_\lambda'}
      \rar{+1} \dar[sloped]{\sim}
      & {}\\
      \RRGGamma_{Y_\mu'} \mathop{\circ}
      \prescript{a}{}{\varphi}_{\lambda\mu}^\dagger \rar
      & \prescript{a}{}{\varphi}_{\lambda\mu}^\dagger \rar
      & \RRGGamma_{X_\mu - Y_\mu'} \mathop{\circ}
      \prescript{a}{}{\varphi}_{\lambda\mu}^\dagger \rar{+1}
      & {}
    \end{tikzcd}
  \]
  of functors on \(\DD^b_\rh(\sD_{X_\lambda})\) from \cite[Chapitre I,
  Proposition 6.3.5]{Meb89}.
  Pre- and post-composing by the edge maps for the \(E_2\) spectral sequence
  from \cite[Chapitre III, Proposition 4.4.6]{Ver67}, we see that
  \eqref{eq:firstmap} is a morphism of functors on
  \(\DD^b_\rh(\sD_{X_\lambda})\).\smallskip
  \par Next, we consider the case when \(T_{\lambda,r}\) is a local cohomology
  functor with support in a locally closed subset \(Y_\lambda - Y'_\lambda\) for
  \(Y_\lambda,Y'_\lambda\) closed in \(X_\lambda\),
  using notation from Definition \ref{def:lyubeznikfunctors} and Setup
  \ref{setup:approx}.
  We have the commutative diagram of distinguished triangles
  \[
    \begin{tikzcd}[baseline=(midsim.base)]
      \prescript{a}{}{\varphi}_{\lambda\mu}^\dagger \mathop{\circ}
      \arrow[d,"\sim"{sloped,name=midsim}]
      \RRGGamma_{Y_\lambda'} \rar
      & \prescript{a}{}{\varphi}_{\lambda\mu}^\dagger \mathop{\circ}
      \RRGGamma_{Y_\lambda} \rar \dar[sloped]{\sim}
      & \prescript{a}{}{\varphi}_{\lambda\mu}^\dagger \mathop{\circ}
      \RRGGamma_{Y_\lambda - Y_\lambda'}
      \rar{+1} \dar
      & {}\\
      \RRGGamma_{Y_\mu'} \mathop{\circ}
      \prescript{a}{}{\varphi}_{\lambda\mu}^\dagger \rar
      & \RRGGamma_{Y_\mu} \mathop{\circ}
      \prescript{a}{}{\varphi}_{\lambda\mu}^\dagger \rar
      & \RRGGamma_{Y_\mu - Y_\mu'} \mathop{\circ}
      \prescript{a}{}{\varphi}_{\lambda\mu}^\dagger \rar{+1}
      & {}
    \end{tikzcd}
  \]
  of functors on \(\DD^b_\rh(\sD_{X_\lambda})\), where the left and middle
  vertical maps are quasi-isomorphisms by \cite[Chapitre I, Proposition
  6.3.5]{Meb89}.
  Thus, the right vertical map is a quasi-isomorphism by \cite[Chapitre II,
  Corollaire 1.2.3]{Ver67}.
  Finally, the same argument as before using the \(E_2\) spectral sequence
  from \cite[Chapitre III, Proposition 4.4.6]{Ver67} shows that
  \eqref{eq:firstmap} is a morphism of functors on
  \(\DD^b_\rh(\sD_{X_\lambda})\).\smallskip
  \par Finally, we consider the case when \(T_{\lambda,r}\) is the kernel of a
  map in the long exact sequence \eqref{eq:leswherekercomesfrom}.
  Applying the functor \(\prescript{a}{}{\varphi}_{\lambda\mu}^*\), we then have
  the commutative diagram
  \[
    \begin{tikzcd}
      \cdots \rar
      & \prescript{a}{}{\varphi}_{\lambda\mu}^*\,
      \sH^i_{Y'_\lambda}(\,\cdot\,) \rar\dar
      & \prescript{a}{}{\varphi}_{\lambda\mu}^*\,
      \sH^i_{Y_\lambda}(\,\cdot\,) \rar\dar
      & \prescript{a}{}{\varphi}_{\lambda\mu}^*\,
      \sH^i_{Y_\lambda - Y'_\lambda}(\,\cdot\,) \rar\dar
      & \cdots\\
      \cdots \rar
      & \sH^i_{Y'_\mu}\bigl(\prescript{a}{}{\varphi}_{\lambda\mu}^*
      (\,\cdot\,)\bigr) \rar
      & \sH^i_{Y_\mu}\bigl(\prescript{a}{}{\varphi}_{\lambda\mu}^*
      (\,\cdot\,)\bigr) \rar
      & \sH^i_{Y_\mu - Y'_\mu}\bigl(\prescript{a}{}{\varphi}_{\lambda\mu}^*
      (\,\cdot\,)\bigr) \rar
      & \cdots
    \end{tikzcd}
  \]
  of \(\sD_{X_\mu}\)-modules 
  where the top row is a complex that is
  not necessarily exact and the bottom row is exact.
  Here, we use the fact that the \(E_2\) spectral sequence from \cite[Chapitre
  III, Proposition 4.4.6]{Ver67} is functorial to ensure that the vertical maps
  (i.e., the maps \eqref{eq:pullbackmaprh}) make the diagram commute.
  We can then construct the map \eqref{eq:pullbackmaprh} for kernels of maps
  apppearing in \eqref{eq:leswherekercomesfrom} by the universal property of
  kernels.
\end{proof}
We now construct the perverse sheaves corresponding to the \(T_\lambda\).
\begin{setup}\label{setup:approxperverse}
  We fix notation as in Setup \ref{setup:approx}.
  Let \(\ell > 0\) be a prime number.
  The perverse sheaves \(F_{T_\lambda}\) on \(X_\lambda\)
  are constructed for every \(\lambda \ge \lambda_0\) as follows.
  If \(r = 0\), then
  \[
    F_{T_\lambda} = F_\id \coloneqq \QQ_\ell[d_\lambda],
  \]
  Now suppose \(r > 0\).
  \begin{enumerate}[label=\((\roman*)\),ref=\roman*]
    \item Suppose\label{setup:approxperverse2i} \(T_r = \sH^{i}_{Y}(\,\cdot\,)\)
      for a locally closed subset \(Y\) of \(X\).
      With notation as in Setup
      \ref{setup:approx2}\((\ref{setup:approx2i})\),
      we set
      \[
        F_{T_\lambda} \coloneqq \prescript{\fp}{}{\mathcal{H}}^{i}
        j_{U_\lambda*}i_{Y_\lambda*}i_{Y_\lambda}^!j_{U_\lambda}^*
        \,F_{T'_\lambda}.
      \]
    \item
      Suppose\label{setup:approxperverse2ii} \(T_r\) is the kernel of a morphism in the long exact sequence
      \begin{align*}
        \cdots \longrightarrow \sH^i_{Y'}(\,\cdot\,) &\longrightarrow
        \sH^i_{Y}(\,\cdot\,)
        \longrightarrow \sH^i_{Y-Y'}(\,\cdot\,) \longrightarrow \cdots
      \intertext{with notation as in Setup
      \ref{setup:approx2}\((\ref{setup:approx2ii})\).
      We then set \(F_{T_\lambda}\) to be the kernel of
      the corresponding morphism in the long exact sequence}
        \cdots \longrightarrow
        \prescript{\fp}{}{\mathcal{H}}^{i}
        j_{U_\lambda*}(i_{Y_\lambda} \mathop{\circ} i_{\lambda})_*
        (i_{Y_\lambda} \mathop{\circ} i_\lambda)^!
        j_{U_\lambda}^*
        &\longrightarrow
        \prescript{\fp}{}{\mathcal{H}}^{i}
        j_{U_\lambda*}i_{Y_\lambda*}
        i_{Y_\lambda}^!
        j_{U_\lambda}^*
        \longrightarrow
        \prescript{\fp}{}{\mathcal{H}}^{i}
        j_{U_\lambda*}i_{Y_\lambda*}
        j_{\lambda*}j_\lambda^*
        i_{Y_\lambda}^!
        j_{U_\lambda}^*
        \longrightarrow \cdots
      \end{align*}
      of functors applied to \(F_{T_\lambda'}\).
  \end{enumerate}
\end{setup}
\par
By
the Riemann--Hilbert correspondence (Theorem \ref{thm:riemannhilbert})
and the description of local cohomology functors in
\cite[Chapter VI, Theorem 7.13\((ii)\)]{BGKHME87}, we have
\[
  \varepsilon^*\bigl(F_{T_\lambda} \otimes_{\QQ_\ell} \CC\bigr) \cong
  \DR_{X_\lambda}\bigl(T_\lambda(\cO_{X_\lambda})\bigr)
\]
where \(\varepsilon\colon X_\lambda^\an \to X_{\lambda,\et}\) is the morphism of topoi
from Theorem \ref{thm:bbdg612}.
\begin{theorem}\label{thm:perversecohunderlimperv}
  Fix notation as in Setup \ref{setup:approx} and Setup \ref{setup:approxperverse}.
  \begin{enumerate}[label=\((\roman*)\),ref=\roman*]
    \item For\label{thm:perversecohunderlimpervpullbacks}
      every \(\mu \ge \lambda \ge \lambda_0\),
      there are pullback maps
      \begin{equation}
          \begin{aligned}
            \prescript{a}{}{\varphi}_{\lambda\mu}^![d_\lambda-d_\mu]
            (d_\lambda-d_\mu)\,F_{T_\lambda}
            &\longrightarrow F_{T_\mu}\\
            \prescript{\fp}{}{\cH}^{d_\lambda-d_\mu}
            \prescript{a}{}{\varphi}_{\lambda\mu}^!(d_\lambda-d_\mu)\,F_{T_\lambda}
            &\longrightarrow F_{T_\mu}
          \end{aligned}
        \label{eq:pullbackafterrh}
      \end{equation}
      compatible with the pullback maps \eqref{eq:pullbackmaprh} under
      the Riemann--Hilbert correspondence (Theorem
      \ref{thm:riemannhilbert}) after
      analytification (Theorem \ref{thm:bbdg612}) and extending to scalars to
      \(\CC\).
      The maps \eqref{eq:pullbackafterrh} are isomorphic to maps
      \begin{equation}
        \begin{aligned}
          \prescript{a}{}{\varphi}_{\lambda\mu}^*[d_\mu-d_\lambda]
          \,F_{T_\lambda}
          &\longrightarrow F_{T_\mu}\\
          \prescript{\fp}{}{\cH}^{d_\mu-d_\lambda}
          \prescript{a}{}{\varphi}_{\lambda\mu}^*\,F_{T_\lambda}
          &\longrightarrow F_{T_\mu}\mathrlap{.}
        \end{aligned}\label{eq:pullbackafterrh2}
      \end{equation}
    \item The\label{thm:perversecohunderlimpervdirsys}
      pullback maps \eqref{eq:pullbackafterrh} and \eqref{eq:pullbackafterrh2}
      fit into commutative diagrams of the form
      \begin{equation}
        \begin{tikzcd}[column sep=-9.5em,
          ampersand replacement=\&]
          \prescript{\fp}{}{\cH}^{d_\lambda-d}
          \prescript{a}{}{\varphi}_{\lambda}^{!}(d_\lambda-d)\,
          F_{T_\lambda}
          \arrow{rr}
          \& \&
          \prescript{\fp}{}{\cH}^{d_\mu-d}
          \prescript{a}{}{\varphi}_{\mu}^{!}(d_\mu-d)\,
          F_{T_\mu}\\
          \& \prescript{\fp}{}{\cH}^{d_\mu-d}
          \prescript{a}{}{\varphi}_{\mu}^{!}(d_\mu-d)\,
          \prescript{\fp}{}{\cH}^{d_{\lambda}-d_\mu}
          \prescript{a}{}{\varphi}_{\lambda\mu}^!(d_{\lambda}-d_\mu)
          \,F_{T_\lambda}
          \arrow[dash]{ul}[sloped]{\sim}
          \arrow{ur}
          \\[-1.8em]          
          \prescript{\fp}{}{\cH}^{d-d_\lambda}
          \prescript{a}{}{\varphi}_{\lambda}^{*}\,
          F_{T_\lambda}
          \arrow{rr}
          \& \&
          \prescript{\fp}{}{\cH}^{d-d_\mu}
          \prescript{a}{}{\varphi}_{\mu}^{*}\,
          F_{T_\mu}\\
          \& \prescript{\fp}{}{\cH}^{d-d_\mu}
          \prescript{a}{}{\varphi}_{\mu}^{*}\,
          \prescript{\fp}{}{\cH}^{d_\mu-d_{\lambda}}
          \prescript{a}{}{\varphi}_{\lambda\mu}^*
          \,F_{T_\lambda}
          \arrow[dash]{ul}[sloped]{\sim}
          \arrow{ur}
        \end{tikzcd}
        \label{eq:dirsysperv}
      \end{equation}
      where the right diagonal maps are the results of applying
      \(\prescript{\fp}{}{\cH}^{d_\mu-d}
      \prescript{a}{}{\varphi}_{\lambda}^{!}(d_\mu-d)\)
      and
      \(\prescript{\fp}{}{\cH}^{d-d_\mu}
      \prescript{a}{}{\varphi}_{\mu}^{*}\)
      to the pullback maps in
      \eqref{eq:pullbackafterrh} and \eqref{eq:pullbackafterrh2}, respectively.
  \end{enumerate}
\end{theorem}
\begin{proof}
  We prove \((\ref{thm:perversecohunderlimpervpullbacks})\) and
  \((\ref{thm:perversecohunderlimpervdirsys})\)
  by induction on \(r\).
  If \(r = 0\), then the maps \eqref{eq:pullbackafterrh} and
  \eqref{eq:pullbackafterrh2} are the identity maps
  on \(\QQ_\ell[d_\mu]\)
  by \citeleft\citen{ILO14}\citemid
  Expos\'e XVI, Corollaire 3.1.2\citeright, proving
  \((\ref{thm:perversecohunderlimpervpullbacks})\).
  Moreover, the maps in \eqref{eq:dirsysperv} are the identity map on
  \(\QQ_\ell[d]\)
  by Remark \ref{rem:cfpullbackql},
  proving
  \((\ref{thm:perversecohunderlimpervdirsys})\).\smallskip
  \par For \(r > 0\), we first note that it suffices to construct the maps and
  diagrams involving \((\,\cdot\,)^!\).
  This is because
  \begin{align*}
    \prescript{a}{}{\varphi}_{\lambda\mu}^![d_\lambda-d_\mu] (d_\lambda-d_\mu)
    &=
    \RRHHom\Bigl(\prescript{a}{}{\varphi}_{\lambda\mu}^*\RRHHom\bigl(\,\cdot\,,
    E[2d_\lambda](d_\lambda)
    \bigr),E[2d_\mu](d_\mu)\Bigr)[d_\lambda-d_\mu] (d_\lambda-d_\mu)\\
    &\cong
    \RRHHom\Bigl(\RRHHom\bigl(\prescript{a}{}{\varphi}_{\lambda\mu}^*(\,\cdot\,),
    E[2d_\lambda]
    \bigr),E[2d_\mu]\Bigr)[d_\lambda-d_\mu]\\
    &\cong
    \RRHHom\Bigl(\RRHHom\bigl(\prescript{a}{}{\varphi}_{\lambda\mu}^*(\,\cdot\,),
    E[2d_\mu]
    \bigr),E[2d_\mu]\Bigr)[d_\mu-d_\lambda]\\
    &\cong \prescript{a}{}{\varphi}_{\lambda\mu}^*[d_\mu-d_\lambda]\\
    \prescript{a}{}{\varphi}_{\lambda}^![d_\lambda-d] (d_\lambda-d)
    &=
    \RRHHom\Bigl(\prescript{a}{}{\varphi}_{\lambda}^*\RRHHom\bigl(\,\cdot\,,
    E[2d_\lambda](d_\lambda)
    \bigr),E[2d](d)\Bigr)[d_\lambda-d] (d_\lambda-d)\\
    &\cong
    \RRHHom\Bigl(\RRHHom\bigl(\prescript{a}{}{\varphi}_{\lambda}^*(\,\cdot\,),
    E[2d_\lambda]
    \bigr),E[2d]\Bigr)[d_\lambda-d]\\
    &\cong
    \RRHHom\Bigl(\RRHHom\bigl(\prescript{a}{}{\varphi}_{\lambda}^*(\,\cdot\,),
    E[2d]
    \bigr),E[2d]\Bigr)[d-d_\lambda]\\
    &\cong \prescript{a}{}{\varphi}_{\lambda}^*[d-d_\lambda]
  \end{align*}
  where the first isomorphisms hold by \cite[Theorem 6.3\((iii)\)]{Eke90}
  and the last isomorphisms hold by duality \cite[Expos\'e XVII, Th\'eor\`eme
  0.2]{ILO14}.
  We construct the maps \eqref{eq:pullbackafterrh}
  as follows.
  We first prove \((\ref{thm:perversecohunderlimpervpullbacks})\) when
  \(T_{\lambda,r}\) is a local cohomology
  functor with support in a closed set \(i_\lambda\colon Y'_\lambda
  \hookrightarrow X_\lambda\) or an open subset
  \(j_\lambda\colon X_\lambda - Y'_\lambda \hookrightarrow
  X_\lambda\),
  using notation from Definition 
  \ref{def:lyubeznikfunctors}, Setup
  \ref{setup:approx2}\((\ref{setup:approx2i})\),
  and Setup
  \ref{setup:approxperverse}\((\ref{setup:approxperverse2i})\).
  We have the isomorphism of distinguished triangles of triangulated functors
  \begin{equation}\label{eq:lcclosedpervtri}
    \begin{tikzcd}[baseline=(isos.base)]
      \prescript{a}{}{\varphi}_{\lambda\mu}^!
      \mathop{\circ}
      i_{\lambda*} i_\lambda^{!}
      \dar[sloped]{\sim}
      \rar
      & \prescript{a}{}{\varphi}_{\lambda\mu}^!
      \rar \dar[equal]
      & \prescript{a}{}{\varphi}_{\lambda\mu}^!
      \mathop{\circ}
      j_{\lambda*} j_\lambda^{*}
      \rar{+1} \arrow[d,"\sim"{sloped,name=isos}]
      & {}\\
      i_{\mu*} i_\mu^{!} \mathop{\circ}
      \prescript{a}{}{\varphi}_{\lambda\mu}^!
      \rar
      & \prescript{a}{}{\varphi}_{\lambda\mu}^!
      \rar
      & j_{\mu*} j_\mu^{*} \mathop{\circ}
      \prescript{a}{}{\varphi}_{\lambda\mu}^!
      \rar{+1}
      & {}
    \end{tikzcd}
  \end{equation}
  by base change
  \citeleft\citen{SGA43}\citemid Expos\'e XVI, Th\'eor\`eme 1.1, Expos\'e XVII,
  Proposition 6.1.4\((iii)\), and
  Expos\'e XVIII, (3.1.14)\((a)\)\citepunct
  \citen{Eke90}\citemid Theorem 6.3\((iii)\)\citeright\ and
  \cite[Expos\'e I, Proposition 1.12]{SGA5}.
  Similarly to Proposition \ref{prop:perversecohunderlimdxmod}
  but using the cohomological functor associated to the perverse \(t\)-structure
  instead of the ordinary \(t\)-structure in the spectral sequence
  \cite[Chapitre III, Proposition 4.4.6]{Ver67} (see also
  \citeleft\citen{Del94}\citemid \S1\citepunct \citen{dCM10}\citemid
  Definition 3.6.1\citeright),
  the edge maps for the \(E_2\) spectral sequence yield the pullback maps
  \begin{alignat*}{5}
    &\prescript{\fp}{}{\cH}^{d_\lambda-d_{\mu}}
    \prescript{a}{}{\varphi}_{\lambda\mu}^!
    &{}\mathop{\circ}{}&
    \prescript{\fp}{}{\cH}^i
    i_{\lambda*} i_\lambda^{!}
    &{}\longrightarrow{}&& 
    \prescript{\fp}{}{\cH}^i
    i_{\mu*} i_\mu^{!} &{}\mathop{\circ}{}&
    \prescript{\fp}{}{\cH}^{d_\lambda-d_{\mu}}
    \prescript{a}{}{\varphi}_{\lambda\mu}^!
    ,\\
    &\prescript{\fp}{}{\cH}^{d_\lambda-d_{\mu}}
    \prescript{a}{}{\varphi}_{\lambda\mu}^!
    &{}\mathop{\circ}{}&
    \prescript{\fp}{}{\cH}^i
    j_{\lambda*} j_\lambda^{*}
    &{}\longrightarrow{}&& 
    \prescript{\fp}{}{\cH}^i
    j_{\mu*} j_\mu^{*} &{}\mathop{\circ}{}&
    \prescript{\fp}{}{\cH}^{d_\lambda-d_{\mu}}
    \prescript{a}{}{\varphi}_{\lambda\mu}^!.
  \end{alignat*}
  Twisting by \((d_\lambda-d_\mu)\),
  applying these maps to \(F_{T'_\lambda}\), and post-composing with the
  pullback map for \(F_{T'_\lambda}\), which exists by inductive hypothesis,
  we obtain the pullback maps \eqref{eq:pullbackafterrh} when \(T_{\lambda,r}\)
  is a local cohomology functor with closed or open support.\smallskip
  \par Next, we prove \((\ref{thm:perversecohunderlimpervpullbacks})\)
  when \(T_{\lambda,r}\) is a local cohomology
  functor with support in a locally closed subset \(Y_\lambda - Y'_\lambda\)
  using notation from Definition \ref{def:lyubeznikfunctors}, Setup
  \ref{setup:approx2}\((\ref{setup:approx2i})\),
  and Setup
  \ref{setup:approxperverse}\((\ref{setup:approxperverse2i})\).
  We then have the commutative diagram of distinguished triangles
  \begin{equation}\label{eq:locclosedlcperv}
    \begin{tikzcd}[baseline=(midsim.base)]
      \prescript{a}{}{\varphi}_{\lambda\mu}^!
      \mathop{\circ}
      (i_{Y_\lambda} \mathop{\circ} i_{\lambda})_*
      (i_{Y_\lambda} \mathop{\circ} i_{\lambda})^{!}
      \rar \dar[sloped]{\sim}
      & \prescript{a}{}{\varphi}_{\lambda\mu}^!
      \mathop{\circ}
      i_{Y_\lambda*}
      i_{Y_\lambda}^{!}
      \rar \arrow[d,"\sim"{sloped,name=midsim}]
      & \prescript{a}{}{\varphi}_{\lambda\mu}^!
      \mathop{\circ}
      i_{Y_\lambda*}
      j_{\lambda*}
      j_\lambda^*
      i_{Y_\lambda}^{!}
      \rar{+1} \dar
      & {}\\
      (i_{Y_\mu} \mathop{\circ} i_{\mu})_*
      (i_{Y_\mu} \mathop{\circ} i_{\mu})^{!}
      \mathop{\circ}
      \prescript{a}{}{\varphi}_{\lambda\mu}^!
      \rar
      & i_{Y_\mu*}
      i_{Y_\mu}^{!}
      \mathop{\circ}
      \prescript{a}{}{\varphi}_{\lambda\mu}^!
      \rar
      & i_{Y_\mu*}
      j_{\mu*}
      j_\mu^*
      i_{Y_\mu}^{!} \mathop{\circ}
      \prescript{a}{}{\varphi}_{\lambda\mu}^!
      \rar{+1}
      & {}
    \end{tikzcd}
  \end{equation}
  of functors on \(\DD^b_c(X_\lambda,\QQ_\ell)\), where the left and middle
  vertical maps are quasi-isomorphisms by the closed support case in the
  previous paragraph, and the right vertical map is constructed in the same
  manner by base change
  \citeleft\citen{SGA43}\citemid Expos\'e XVII, Proposition 2.1.3 and
  Expos\'e XVIII, (3.1.14)\((a)\)\citepunct
  \citen{Eke90}\citemid Theorem 6.3\((iii)\)\citeright.
  Thus, the right vertical map is a quasi-isomorphism by
  \cite[Chapitre II, Corollaire 1.2.3]{Ver67}.
  Finally, using the \(E_2\) spectral sequence as before
  from \cite[Chapitre III, Proposition 4.4.6]{Ver67}
  yields the pullback maps \eqref{eq:pullbackafterrh} when \(T_{\lambda,r}\) is
  a local cohomology functor with locally closed support.\smallskip
  \par We now show \((\ref{thm:perversecohunderlimpervdirsys})\)
  when \(T_{\lambda,r}\) is a local
  cohomology functor with closed, open, or locally closed support.
  Applying
  \(\prescript{a}{}{\varphi}_{\mu}^{*}[d_\mu-d]\) and
  \(\prescript{a}{}{\varphi}_{\mu}^{!}[d_\mu-d]\)
  to the right vertical maps in
  \eqref{eq:locclosedlcperv}, we obtain the
  right diagonal map in the diagram
  \begin{equation}\label{eq:dirsyspervfunctors}
    \begin{tikzcd}[column sep=-9em,row sep=large]
      \prescript{a}{}{\varphi}_{\lambda}^{!}[d_\lambda-d]
      \mathop{\circ}
      i_{Y_\lambda*}
      j_{\lambda*}
      j_\lambda^*
      i_{Y_\lambda}^{!}F_{T'_\lambda}
      \arrow{rr}
      & & \prescript{a}{}{\varphi}_{\mu}^{!}[d_\mu-d]
      \mathop{\circ}
      i_{Y_\mu*}
      j_{\mu*}
      j_\mu^*
      i_{Y_\mu}^{!} \mathop{\circ}
      \prescript{a}{}{\varphi}_{\lambda\mu}^!
      [d_\lambda-d_\mu]\,F_{T'_\lambda}\\
      & \prescript{a}{}{\varphi}_{\mu}^{!}[d_\mu-d] \mathop{\circ}
      \prescript{a}{}{\varphi}_{\lambda\mu}^![d_\lambda-d_\mu] \mathop{\circ}
      i_{Y_\lambda*}
      j_{\lambda*}
      j_\lambda^*
      i_{Y_\lambda}^{!}F_{T'_\lambda}
      \arrow[ul,dash,"\sim"{sloped}]
      \arrow[end anchor={[xshift=-2em]}]{ur}
    \end{tikzcd}
  \end{equation}
  The left diagonal map
  is obtained from the isomorphism
  \begin{align}
    \prescript{a}{}{\varphi}_{\mu}^{!} \mathop{\circ}
    \prescript{a}{}{\varphi}_{\lambda\mu}^!
    &= \mathbf{D}_X \mathop{\circ} \prescript{a}{}{\varphi}_{\mu}^*
    \mathop{\circ} \mathbf{D}_{X_\mu}
    \mathop{\circ} \prescript{a}{}{\varphi}_{\lambda\mu}^!\nonumber\\
    &\cong \mathbf{D}_X \mathop{\circ} \prescript{a}{}{\varphi}_{\mu}^*
    \mathop{\circ} \prescript{a}{}{\varphi}_{\lambda\mu}^*
    \mathop{\circ} \mathbf{D}_{X_\lambda}\nonumber\\
    &= \mathbf{D}_X \mathop{\circ} \prescript{a}{}{\varphi}_{\lambda}^*
    \mathop{\circ} \mathbf{D}_{X_\lambda}\nonumber\\
    &= \prescript{a}{}{\varphi}_{\lambda}^{!}\nonumber
  \intertext{where the middle isomorphism holds by \cite[Expos\'e I, Proposition
  1.12\((b)\)]{SGA5}.
  The left diagonal map in \eqref{eq:dirsysperv} is an isomorphism since the
  bottom term in \eqref{eq:dirsysperv} appears in the top right corner of the 
  \(E_2\) spectral sequence from \cite[Chapitre III, Proposition
  4.4.6]{Ver67}.
  The horizontal maps in \eqref{eq:dirsyspervfunctors} are constructed using
  base change isomorphisms as before and
  the isomorphisms}
  \begin{split}
    \prescript{a}{}{\varphi}_{\lambda}^{!}[d_\lambda-d]
    \mathop{\circ}
    i_{Y_\lambda*}
    j_{\lambda*}
    j_\lambda^*
    i_{Y_\lambda}^{!}
    &= \mathbf{D}_X \mathop{\circ} \prescript{a}{}{\varphi}_{\lambda}^*
    \mathop{\circ} \mathbf{D}_{X_\lambda}
    [d_\lambda-d]
    \mathop{\circ} i_{Y_\lambda*}
    j_{\lambda*}
    j_\lambda^*
    i_{Y_\lambda}^{!}\\
    &\cong \mathbf{D}_X \mathop{\circ} \prescript{a}{}{\varphi}_{\lambda}^*
    \mathop{\circ} i_{Y_\lambda*}
    j_{\lambda!}
    j_\lambda^*
    i_{Y_\lambda}^{*}
    \mathop{\circ} \mathbf{D}_{X_\lambda}
    [d_\lambda-d]\\
    &\cong \mathbf{D}_X \mathop{\circ} \prescript{a}{}{\varphi}_{\lambda}^*
    \mathop{\circ} i_{Y_\lambda*}
    j_{\lambda!}
    j_\lambda^*
    i_{Y_\lambda}^{*}
    \mathop{\circ} \mathbf{D}_{X_\lambda}
    [d_\lambda-d]\\
    &\cong \mathbf{D}_X
    \mathop{\circ} i_{Y*}
    j_{!}
    j^*
    i_{Y}^{*} \mathop{\circ} \prescript{a}{}{\varphi}_{\lambda}^*
    \mathop{\circ} \mathbf{D}_{X_\lambda}
    [d_\lambda-d]\\
    &\cong i_{Y*}j_*j_*i_Y^! \mathop{\circ} \mathbf{D}_X \mathop{\circ}
    \prescript{a}{}{\varphi}_{\lambda}^* \mathop{\circ}
    \mathbf{D}_{X_\lambda}[d_\lambda-d]\\
    &= i_{Y*}j_*j_*i_Y^! \mathop{\circ}
    \prescript{a}{}{\varphi}_{\lambda}^{!}[d_\lambda-d]
  \end{split}\label{eq:basechangevarphi}
  \end{align}
  where the isomorphism in the fourth row holds by
  proper base change \citeleft\citen{SGA43}\citemid Expos\'e XII,
  Th\'eor\`eme 5.1 and Expos\'e XVII, Th\'eor\`eme 5.2.6\citepunct
  \citen{Eke90}\citemid Theorem 6.3\((iii)\)\citeright\ 
  and the other
  isomorphisms hold by \cite[Expos\'e I, Proposition 1.12]{SGA5}.
  The diagram \eqref{eq:dirsyspervfunctors} commutes since these base change
  isomorphisms are compatible with composition \cite[Expos\'e XII, Proposition
  4.4 and p.\ 578]{SGA43}.
  Finally, we construct the commutative diagram \eqref{eq:dirsysperv}.
  Compose the
  maps in \eqref{eq:dirsyspervfunctors} with the pullback maps
  \eqref{eq:pullbackafterrh} and the pullback maps for \(T'\), which exist by
  induction on \(r\).
  Using the spectral
  sequence \cite[Chapitre III, Proposition 4.4.6]{Ver67} as before, we obtain
  \eqref{eq:dirsysperv}.\smallskip
  \par Next, we prove \((\ref{thm:perversecohunderlimpervdirsys})\)
  when \(T_{\lambda,r}\) is the kernel of a
  map in the long exact sequence \eqref{eq:leswherekercomesfrom}.
  Fix notation as in Definition \ref{def:lyubeznikfunctors}, Setup
  \ref{setup:approx2}\((\ref{setup:approx2ii})\),
  and Setup
  \ref{setup:approxperverse}\((\ref{setup:approxperverse2ii})\).
  We have the distinguished triangle
  \[
    j_{U_\lambda*}
    (i_{Y_\lambda} \mathop{\circ} i_{\lambda})_*
    (i_{Y_\lambda} \mathop{\circ} i_{\lambda})^{!} j_{U_\lambda}^*
    \longrightarrow
    j_{U_\lambda*}
    i_{Y_\lambda*}
    i_{Y_\lambda}^{!} j_{U_\lambda}^*
    \longrightarrow
    j_{U_\lambda*}
    i_{Y_\lambda*}
    j_{\lambda*}
    j_\lambda^*
    i_{Y_\lambda}^{!}j_{U_\lambda}^*
    \xrightarrow{+1}
  \]
  which induces the long exact sequence
  \begin{align}
    \cdots
    &\longrightarrow
    \prescript{\fp}{}{\cH}^i j_{U_\lambda*}
    (i_{Y_\lambda} \mathop{\circ} i_{\lambda})_*
    (i_{Y_\lambda} \mathop{\circ} i_{\lambda})^{!}
    j_{U_\lambda}^*\nonumber\\
    &\longrightarrow
    \prescript{\fp}{}{\cH}^i j_{U_\lambda*}
    i_{Y_\lambda*}
    i_{Y_\lambda}^{!} j_{U_\lambda}^*\label{eq:perverseleslc}\\
    &\longrightarrow
    \prescript{\fp}{}{\cH}^i j_{U_\lambda*}
    i_{Y_\lambda*}
    j_{\lambda*}
    j_\lambda^*
    i_{Y_\lambda}^{!}j_{U_\lambda}^*
    \longrightarrow \cdots.\nonumber
  \end{align}
  Applying the functor \(\prescript{\fp}{}{\cH}^{d_{\lambda}-d_\mu}
  \prescript{a}{}{\varphi}_{\lambda\mu}^!\),
  we then have the commutative diagram
  \[
    \mathclap{\begin{tikzcd}[ampersand replacement=\&,column sep=scriptsize]
      \vdots \dar \& \vdots \dar\\
      \prescript{\fp}{}{\cH}^{d_{\lambda}-d_\mu}
      \prescript{a}{}{\varphi}_{\lambda\mu}^!
      \mathop{\circ}
      \prescript{\fp}{}{\cH}^i j_{U_\lambda*}
      (i_{Y_\lambda} \mathop{\circ} i_{\lambda})_*
      (i_{Y_\lambda} \mathop{\circ} i_{\lambda})^{!} j_{U_\lambda}^*
      \dar \rar
      \& \prescript{\fp}{}{\cH}^i j_{U_\mu*}
      (i_{Y_\mu} \mathop{\circ} i_{\mu})_*
      (i_{Y_\mu} \mathop{\circ} i_{\mu})^{!} j_{U_\mu}^*
      \mathop{\circ} \prescript{\fp}{}{\cH}^{d_{\lambda}-d_\mu}
      \prescript{a}{}{\varphi}_{\lambda\mu}^! \dar\\
      \prescript{\fp}{}{\cH}^{d_{\lambda}-d_\mu}
      \prescript{a}{}{\varphi}_{\lambda\mu}^!
      \mathop{\circ}
      \prescript{\fp}{}{\cH}^i j_{U_\lambda*}
      i_{Y_\lambda*}
      i_{Y_\lambda}^{!} j_{U_\lambda}^* \dar \rar
      \& \prescript{\fp}{}{\cH}^i j_{U_\mu*}
      i_{Y_\mu*}
      i_{Y_\mu}^{!} j_{U_\mu}^*
      \mathop{\circ}
      \prescript{\fp}{}{\cH}^{d_{\lambda}-d_\mu}
      \prescript{a}{}{\varphi}_{\lambda\mu}^! \dar
      \\
      \prescript{\fp}{}{\cH}^{d_{\lambda}-d_\mu}
      \prescript{a}{}{\varphi}_{\lambda\mu}^!
      \mathop{\circ}
      \prescript{\fp}{}{\cH}^i j_{U_\lambda*}
      i_{Y_\lambda*}
      j_{\lambda*}
      j_\lambda^*
      i_{Y_\lambda}^{!}j_{U_\lambda}^* \dar \rar
      \& \prescript{\fp}{}{\cH}^i j_{U_\mu*}
      i_{Y_\mu*}
      j_{\mu*}
      j_\mu^*
      i_{Y_\mu}^{!}j_{U_\mu}^*
      \mathop{\circ}
      \prescript{\fp}{}{\cH}^{d_{\lambda}-d_\mu}
      \prescript{a}{}{\varphi}_{\lambda\mu}^!
      \dar\\
      \vdots \& \vdots
    \end{tikzcd}}
  \]
  where the left column is a complex that is
  not necessarily exact and the right column is exact.
  Here, we use the fact that the \(E_2\) spectral sequence from \cite[Chapitre
  III, Proposition 4.4.6]{Ver67} is functorial to ensure that the horizontal maps
  (i.e., the maps \eqref{eq:pullbackafterrh}) make the diagram commute.
  We can then construct \eqref{eq:pullbackafterrh} for kernels of maps
  apppearing in \eqref{eq:perverseleslc} by the universal property of
  kernels and twisting appropriately by \((d_\lambda-d_\mu)\).\smallskip
  \par For \((\ref{thm:perversecohunderlimpervdirsys})\), it remains to show
  \eqref{eq:dirsysperv} when \(T_{\lambda,r}\) is the kernel of a
  map in the long exact sequence \eqref{eq:leswherekercomesfrom}.
  For this, it suffices to take two copies of the commutative diagram
  \eqref{eq:dirsysperv} using local cohomology functors with different
  supports connected by the maps appearing in the long exact
  sequence \eqref{eq:perverseleslc} appropriately pulled back to \(X\).
  The statement \eqref{eq:dirsysperv}
  now follow by the universal property of
  kernels.\smallskip
  \par To finish the proof of \((\ref{thm:perversecohunderlimpervpullbacks})\),
  we note that the compatibility of the maps \eqref{eq:pullbackmaprh} and
  \(\varepsilon^*\eqref{eq:pullbackafterrh} \otimes_{\QQ_\ell} \CC\) under the
  Riemann--Hilbert correspondence (Theorem \ref{thm:riemannhilbert}, after
  extending the ground field to \(\CC\)) and the
  equivalence of categories in Theorem \ref{thm:bbdg612} follows by
  comparing the two constructions in
  Proposition \ref{prop:perversecohunderlimdxmod} and
  Theorem \ref{thm:perversecohunderlimperv}
  and the compatibility of the de Rham functor and
  \(\varepsilon^*\) with the sheaf operations.
\end{proof}

\subsection{The main comparison theorem}
\par We state our main theorem comparing associated points of local cohomology
modules to the simple components of an associated perverse sheaf.
When \(X\) is a smooth complex variety,
this comparison follows from \citeleft\citen{BBLSZ14}\citemid p.\ 516\citepunct
\citen{NB14}\citemid Remark 2.3\citeright\ together with the Riemann--Hilbert
correspondence (Theorem \ref{thm:riemannhilbert})
and
the comparison theorem for perverse sheaves in the analytic vs.\ \'etale
topologies (Theorem \ref{thm:bbdg612}).
\begin{theorem}\label{thm:assocaresimple}
  Fix notation as in Setup \ref{setup:approx} and Setup
  \ref{setup:approxperverse}.
  By Theorem
  \ref{thm:perversecohunderlimperv}\kern1pt\((\ref{thm:perversecohunderlimpervdirsys})\),
  we can consider the direct
  limit
  \[
    F_T \coloneqq \varinjlim_{\lambda \ge \lambda_0} 
    \prescript{\fp}{}{\cH}^{d_\lambda-d}
    \prescript{a}{}{\varphi}_{\lambda}^!(d_\lambda-d)\,
    F_{T_\lambda}.
  \]
  Then, we have
  \begin{equation}\label{eq:associnsimple}
    \Ass_{\cO_X}\bigl(T(\cO_X)\bigr) \subseteq
    \Set*{\eta \in X \given
      \begin{tabular}{@{}c@{}}
        \(\eta\) is the generic point of a locally closed\\
        connected regular subscheme \(j\colon Y \hookrightarrow X\) such that\\
        \(j_{!*}L[{\dim(Y)}]\)
        is a simple component of
        \(F_T\)\\
        for some lisse \(\QQ_\ell\)-sheaf \(L\) on \(Y\)
    \end{tabular}}.
  \end{equation}
\end{theorem}
Note that the category \(\Perv(X,\QQ_\ell)\) has all (small)
direct limits \cite[Proposition 7.1(2)]{Gab04}, and hence it makes sense to take
direct limits of perverse sheaves to define \(F_T\).\medskip
\par Theorem \ref{thm:assocaresimple} immediately implies Theorem
\ref{thm:lyubeznikfunctorsnew}.
\begin{customthm}{B}
  \label{thm:lyubeznikfunctors}
  Let \(R\) be an excellent regular \(\QQ\)-algebra of finite Krull dimension.
  Let \(T(\,\cdot\,)\) be a Lyubeznik functor on \(R\).
  Then, the module \(T(R)\) has finitely many associated prime ideals.
\end{customthm}
\begin{proof}
  Set \(X = \Spec(R)\).
  The quasi-coherence of \(T(\cO_X)\) \cite[Expos\'e II, Proposition 1]{SGA2}
  implies
  \[
    \Ass_R\bigl(T(R)\bigr) = \Ass_{\cO_X}\bigl(T(\cO_X)\bigr).
  \]
  The right-hand side
  of \eqref{eq:associnsimple} is finite by Theorem \ref{thm:perv}.
  We therefore see that
  Theorem \ref{thm:assocaresimple} implies Theorem \ref{thm:lyubeznikfunctors}.
\end{proof}
We now prove Theorem \ref{thm:assocaresimple}.
\begin{proof}[Proof of Theorem \ref{thm:assocaresimple}]
  Since associated points \cite[Chapter IV, \S1, no.\ 4, Proposition 5]{Bou72}
  and composition series for perverse sheaves (by \cite[Proposition
  2.2.1]{Mor25} and Theorem \ref{thm:fargues})
  are compatible with localization,
  we may work with one point \(x \in X\) at a
  time and replace \(X\) by the local scheme \(\Spec(\cO_{X,x})\), in which case
  it moreover suffices to show that if \(X\) is local with unique closed point
  \(x\) and
  \[
    x \in \Ass_{\cO_X}\bigl(T(\cO_X)\bigr),
  \]
  then
  \(x\) lies in the set on the right-hand side of
  \eqref{eq:associnsimple}.
  \par Set \(R = \cO_{X,x}\) with maximal ideal \(\fm = \fm_x\).
  We will denote the global sections of
  \(T(\cO_X)\) by \(T(R)\) since
  \(T(\cO_X)\)
  is quasi-coherent \cite[Expos\'e II, Proposition 1]{SGA2}.
  We replace
  \(R_\lambda\) by \((R_\lambda)_{\fm \cap R_\lambda}\) to assume that the
  \(R_\lambda\) are local with maximal ideals
  \[
    \fm_\lambda \coloneqq (\fm \cap R_\lambda) \cdot (R_\lambda)_{\fm \cap R_\lambda}.
  \]
  Note that the construction of \(F_T\) is compatible with this base change by
  regular base change \citeleft\citen{Fuj95}\citemid
  Corollary 7.1.6\citepunct \citen{ILO14}\citemid Expos\'e XIV, Lemme
  2.5.3\citeright\ and the \(t\)-exactness of \'etale base change (up to a
  shift) \cite[Proposition 2.2.1]{Mor25}.\smallskip
  \par Let \(x_1,x_2,\ldots,x_d \in R\) be a set of generators for \(\fm\).
  \begin{step}\label{step:xilambda0}
    Finding an element \(\xi_{\lambda_0} \in T_{\lambda_0}(R_{\lambda_0})\) 
    annihilated by \(x_1,x_2,\ldots,x_d\).
  \end{step}
  \par Suppose
  \(\xi \in T(R)\)
  is an element such that \(\Ann_R(\xi) = \fm\),
  which exists by our hypothesis that \(\fm \in \Ass_R(T(R))\).
  Note that \(\xi \ne 0\) since \(1 \notin \Ann_R(\xi)\).
  By the isomorphism \eqref{eq:toxdirlim} in Proposition
  \ref{prop:perversecohunderlimdxmod},
  after possibly replacing \(\lambda_0\) by a
  larger element in \(\Lambda\), \(\xi\) is the image of an element
  \[
    \xi_{\lambda_0} \in T_{\lambda_0}\bigl(R_{\lambda_0}\bigr).
  \]
  By the isomorphism \eqref{eq:toxdirlim} in Proposition
  \ref{prop:perversecohunderlimdxmod} again,
  after
  possibly replacing \(\lambda_0\), we may assume that \(x_1,x_2,\ldots,x_d \in
  R_{\lambda_0}\) and
  \[
    x_1 \cdot \xi_{\lambda_0}
    = x_2 \cdot \xi_{\lambda_0}
    = \cdots
    = x_d \cdot \xi_{\lambda_0}
    = 0.
  \]
  Setting \(\xi_\lambda \in T_\lambda(R_\lambda)\)
  to be the image of \(\xi_{\lambda_0}\) under the
  pullback maps \eqref{eq:pullbackmaprh}
  along \(\prescript{a}{}{\varphi}_{\lambda_0\lambda}\), we have
  \[
    \Ann_{R_\lambda}(\xi_{\lambda}) \cdot R =
    (x_1x_2,\ldots,x_d) \cdot R = \fm
  \]
  for all \(\lambda \ge \lambda_0\).
  Set \(I_\lambda \coloneqq (x_1,x_2,\ldots,x_d)R_\lambda\) and fix the notation
  \[
    \begin{tikzcd}[ampersand replacement=\&]
      \{\fm\} \rar[hook]{i_\fm} \dar[swap]{\prescript{a}{}{\bar{\varphi}}_\lambda}
      \& X \dar{\prescript{a}{}{\varphi}_\lambda}\\
      V(I_\lambda) \rar[hook]{i_{I_\lambda}} \& X_\lambda
    \end{tikzcd}
  \]
  where the square is Cartesian.
  We note that \(x_1,x_2,\ldots,x_d\) forms part of a regular system of
  parameters in \(R_\lambda\) since if their images were not linearly
  independent in \(\fm_\lambda/\fm_\lambda^2\), then their images in
  \(\fm/\fm^2\) would not be linearly independent either.
  Thus, \(R_\lambda/I_\lambda\) is regular for every \(\lambda \ge
  \lambda_0\).\smallskip
  \begin{step}\label{step:imagetodirlimnonzero}
    The map
    \[
      \prescript{\fp}{}{\cH}^{d_{\lambda_0}-d}
      \prescript{a}{}{\varphi}_{\lambda_0}^!(d_{\lambda_0}-d)\,
      \prescript{\fp}{}{\cH}^0
      i_{I_{\lambda_0}*}
      i_{I_{\lambda_0}}^!
      F_{T_{\lambda_0}} \longrightarrow
      \varinjlim_{\lambda \ge \lambda_0}
      \prescript{\fp}{}{\cH}^{d_\lambda-d}
      \prescript{a}{}{\varphi}_\lambda^!(d_\lambda-d)\,
      \prescript{\fp}{}{\cH}^0
      i_{I_\lambda*}
      i_{I_\lambda}^!
      F_{T_\lambda}
    \]
    is nonzero.
  \end{step}
  Suppose that the map is the \(0\) map.
  Since the category of perverse sheaves is Noetherian (Theorem \ref{thm:perv}),
  we see that there exists \(\mu_0 \ge \lambda_0\) such that
  \begin{align*}
    \prescript{\fp}{}{\cH}^{d_{\lambda_0}-d}
    \prescript{a}{}{\varphi}_{\lambda_0}^!(d_{\lambda_0}-d)\,
    \prescript{\fp}{}{\cH}^0
    i_{I_{\lambda_0}*}
    i_{I_{\lambda_0}}^!
    F_{T_{\lambda_0}} &\overset{0}{\longrightarrow}
    \prescript{\fp}{}{\cH}^{d_{\mu_0}-d}
    \prescript{a}{}{\varphi}_{\mu_0}^!(d_{\mu_0}-d)\,
    \prescript{\fp}{}{\cH}^0
    i_{I_{\mu_0}*}
    i_{I_{\mu_0}}^!
    F_{T_{\mu_0}}
  \intertext{is the \(0\) map.
  By the \(t\)-exactness of \(i_{I_\lambda*}\) and
  base change
  \citeleft\citen{SGA43}\citemid Expos\'e XVIII, (3.1.14)\((a)\)\citepunct
  \citen{Eke90}\citemid Theorem 6.3\((iii)\)\citeright, the map}
    i_{\fm*}
    \prescript{\fp}{}{\cH}^{d-d_{\lambda_0}}
    \prescript{a}{}{\bar{\varphi}}_{\lambda_0}^!(d_{\lambda_0}-d)\,
    \prescript{\fp}{}{\cH}^0
    i_{I_{\lambda_0}}^!
    F_{T_{\lambda_0}} &\overset{0}{\longrightarrow}
    i_{\fm*}
    \prescript{\fp}{}{\cH}^{d-d_{\mu_0}}
    \prescript{a}{}{\bar{\varphi}}_{\mu_0}^!(d_{\mu_0}-d)\,
    \prescript{\fp}{}{\cH}^0
    i_{I_{\mu_0}}^!
    F_{T_{\mu_0}}
  \intertext{is also the \(0\) map.
  Under the equivalence of categories between perverse sheaves
  on \(\{\fm\}\) and perverse sheaves on \(X\) supported on \(\{\fm\}\),
  analytification (Theorem \ref{thm:bbdg612}) after changing ground fields, and
  the Riemann--Hilbert correspondence (Theorem \ref{thm:riemannhilbert}),
  the map above corresponds to the map of the form}
    \cH^{d_\lambda-d}
    \bar{\varphi}_{\lambda_0}^\dagger\,
    \cH^0
    i_{I_{\lambda_0}}^\dagger
    T_{\lambda_0}(\cO_{X_{\lambda_0}})
    &\overset{0}{\longrightarrow}
    \cH^{d_{\mu_0}-d}
    \bar{\varphi}_{\mu_0}^\dagger\,
    \cH^0
    i_{I_{\mu_0}}^\dagger
    T_{\mu_0}(\cO_{X_{\mu_0}}),
  \intertext{where we omit the change of ground field from our notation.
  Since changing ground fields is faithfully flat, the map}
    \Hom_{R_{\lambda_0}}\bigl(R_{\lambda_0}/I_{\lambda_0},
    T_{\lambda_0}(R_{\lambda_0})\bigr) \otimes_{R_{\lambda_0}} R
    &\overset{0}{\longrightarrow}
    \Hom_{R_{\mu_0}}\bigl(R_{\mu_0}/I_{\mu_0},
    T_{\mu_0}(R_{\mu_0})\bigr) \otimes_{R_{\mu_0}} R
  \end{align*}
  of \(R\)-modules is the \(0\) map.
  Here, we use the fact that the rings \(R/I_\lambda\) are regular for every
  \(\lambda \ge \lambda_0\) (proved in Step \ref{step:xilambda0})
  and the description of restriction of
  \(\sD\)-modules in \citeleft\citen{Kas70}\citemid Theorem 2.3.4\citepunct
  \citen{BGKHME87}\citemid Chapter VI, Theorem 7.4\((ii)\)\citeright.
  This map fits into the commutative diagram
  \[
    \begin{tikzcd}
      \Hom_{R_{\lambda_0}}\bigl(R_{\lambda_0}/I_{\lambda_0},
    T_{\lambda_0}(R_{\lambda_0})\bigr) \otimes_{R_{\lambda_0}} R
      \rar{0} \dar[hook]
      & \prescript{a}{}{\bar{\varphi}}^*_{\mu_0}
      \Hom_{R_{\mu_0}}\bigl(R_{\mu_0}/I_{\mu_0},
      T_{\mu_0}(R_{\mu_0})\bigr) \otimes_{R_{\mu_0}} R \dar\\
      T_{\lambda_0}(R_{\lambda_0}) \otimes_{R_{\lambda_0}} R
      \rar
      & T(R)
    \end{tikzcd}
  \]
  where \(\xi_{\lambda_0} \otimes 1\) lies in
  the top left module by the choice of \(\lambda_0\) in Step
  \ref{step:xilambda0}.
  By the commutativity of the diagram, the image of
  \(\xi_{\lambda_0} \otimes 1\) in
  \(T(R)\) is \(0\), contradicting Step
  \ref{step:xilambda0}.\smallskip
  \setcounter{step}{2}
  \begin{step}\label{step:gtftmap}
    We have an isomorphism
    \[
      G_T \coloneqq \varinjlim_{\lambda \ge \lambda_0}
      \prescript{\fp}{}{\cH}^{d_\lambda-d}
      \prescript{a}{}{\varphi}_\lambda^!(d_\lambda-d)\,
      \prescript{\fp}{}{\cH}^0
      i_{I_\lambda*}
      i_{I_\lambda}^!\,
      F_{T_\lambda} \cong
      \prescript{\fp}{}{\cH}^0 i_{\fm*}i_{\fm}^!\,F_T.
    \]
  \end{step}
  We have the chain of isomorphisms
  \begin{align*}
    G_{T} &\cong \varinjlim_{\lambda \ge \lambda_0}
    \prescript{\fp}{}{\cH}^{d-d_\lambda}
    \prescript{a}{}{\varphi}_\lambda^*\,
    \prescript{\fp}{}{\cH}^0
    i_{I_\lambda*}
    i_{I_\lambda}^!\,
    F_{T_\lambda}\\
    &\cong \varinjlim_{\lambda \ge \lambda_0}
    \prescript{\fp}{}{\cH}^{0}
    \prescript{a}{}{\varphi}_\lambda^*\,
    \prescript{\fp}{}{\cH}^0
    i_{I_\lambda*}
    i_{I_\lambda}^!\,
    F_{T_\lambda}[d-d_\lambda]\\
    &\cong \prescript{\fp}{}{\cH}^{0}
    \varinjlim_{\lambda \ge \lambda_0}
    \prescript{a}{}{\varphi}_\lambda^*\,
    \prescript{\fp}{}{\cH}^0
    i_{I_\lambda*}
    i_{I_\lambda}^!\,
    F_{T_\lambda}[d-d_\lambda]\\
    &\cong \prescript{\fp}{}{\cH}^{0}
    i_{\fm*}
    i_{\fm}^!\,
    F_{T}.
  \end{align*}
  The first isomorphism holds by Theorem
  \ref{thm:perversecohunderlimperv}\((\ref{thm:perversecohunderlimpervpullbacks})\),
  and the second isomorphism holds by definition of perverse cohomology.
  The third isomorphism holds since
  perverse truncations commute with direct limits by the fact that direct
  systems of perverse sheaves lift to direct systems of actual complexes
  \cite[Proposition 7.1]{Gab04} and by the construction of perverse truncations
  using modified Godement resolutions in Remark \ref{rem:gabbertruncate}.
  Finally, the last isomorphism holds by Gabber's version of Grothendieck's
  limit theorem for local cohomology \cite[Proposition 5.2]{Gab04}.\smallskip
  \setcounter{step}{3}
  \begin{step}
    Conclusion of proof.
  \end{step}
  By Steps \ref{step:imagetodirlimnonzero} and \ref{step:gtftmap}, we see that
  the image of
  \[
    \prescript{\fp}{}{\cH}^{d_{\lambda_0}-d}
    \prescript{a}{}{\varphi}_{\lambda_0}^!(d_{\lambda_0}-d)\,
    \prescript{\fp}{}{\cH}^0
    i_{I_{\lambda_0}*}
    i_{I_{\lambda_0}}^!
    F_{T_{\lambda_0}}
    \longrightarrow
    i_{\fm*}i_{\fm}^! F_T
  \]
  is nonzero.
  Thus, there exists a simple component of \(F_T\) with support on \(\{\fm\}\).
\end{proof}

\section{Local cohomology modules of\texorpdfstring{\except{toc}{\\}{}}{}
excellent locally unramified regular rings}\label{sect:mixedchar}
In this section, we prove Theorem \ref{thm:mainnew}.
Since Theorem \ref{thm:lyubeznikfunctors} says that \(H^i_I(R) \otimes_\ZZ \QQ\)
has finitely many associated prime ideals, what remains is to show that
\[
  \bigcup_{p\ \text{prime}}
  \Ass_R\Bigl(\ker\bigl(H^i_I(R) \overset{p}{\longrightarrow}
  H^i_I(R)\bigr)\Bigr)
\]
is finite.
To do so, we will use Lyubeznik's theory of \(F\)-modules \cite{Lyu97}, the fact
that local cohomology modules are \(F\)-finite \(F\)-modules \cite[Example
2.2\((b)\)]{Lyu97}, and the fact that \(F\)-finite \(F\)-modules have finite
length in the category of \(F\)-modules.
This last finite length property is where the hypothesis on the quotients
\(R/pR\) is used, and is due to
Blickle and B\"ockle
\cite[Theorem 5.13]{BB11}.
See Theorem \ref{thm:fmodfinlen}, where we also state an earlier finite length
result for rings of
finite type over regular local rings of prime characteristic \(p > 0\)
due to Lyubeznik \cite[Theorem
3.2]{Lyu97}.
\subsection{\emph{F}-modules}
We recall the definition of \(F\)-modules due to Lyubeznik \cite{Lyu97}.
This notion relies on the following functor introduced by Peskine and Szpiro
\cite{PS73}.
\begin{citeddef}[{\cite[Chapitre I, D\'efinition 1.2]{PS73}}]
  Let \(R\) be a ring of prime characteristic \(p > 0\).
  Denote by
  \[
    \begin{tikzcd}[row sep=0,column sep=1.475em]
      F\colon &[-2.35em]  R \rar & F_*R\\
      & r \rar[mapsto] & r^p
    \end{tikzcd}
  \]
  the Frobenius map.
  The \textsl{Frobenius functor of Peskine--Szpiro} is the extension of scalars functor
  \[
    \begin{tikzcd}[row sep=0,column sep=1.475em]
      \mathsf{F}\colon
      &[-3.5em]
      \mathsf{Mod}_R \rar
      & \mathsf{Mod}_R\\
      & M \rar[mapsto] & M \otimes_R F_*R\\
      & \Bigl(M \overset{h}{\longrightarrow} N\Bigr) \rar[mapsto]
      & \Bigl(M \otimes_R F_*R \xrightarrow{\,h \otimes \id\,} N \otimes_R
      F_*R\Bigr)
    \end{tikzcd}
  \]
  along the Frobenius map \(F\), where we consider \(M \otimes_R F_*R\) as an
  \(R\)-module using the \(F_*R\)-module structure coming from the right factor
  \(F_*R\).
\end{citeddef}
The Peskine--Szpiro functor \(\mathsf{F}\) is exact when \(R\) is regular, since
in this case, the Frobenius map \(F\colon R \to F_*R\) is flat \cite[Theorem
2.1]{Kun69}.
\begin{citeddef}[{\cite[Definition 1.1]{Lyu97}}]
  Let \(R\) be a regular ring of prime characteristic \(p > 0\).
  An \textsl{\(F\)-module} over \(R\) is a pair \((\sM,\theta_\sM)\) where \(\sM\) is an
  \(R\)-module and
  \[
    \theta_\sM\colon \sM \overset{\sim}{\longrightarrow} \mathsf{F}(\sM)
  \]
  is a right \(R\)-module isomorphism, which we call the \textsl{structure morphism}
  of \(\sM\).
  A \textsl{morphism} \(f\colon (\sM,\theta_\sM) \to (\sN,\theta_\sN)\) of \(F\)-modules is
  an \(R\)-module homomorphism \(f\colon \sM \to \sN\) for which the diagram
  \[
    \begin{tikzcd}
      \sM \rar{f}\dar[swap]{\theta_\sM} & \sN\dar{\theta_\sN}\\
      \mathsf{F}(\sM) \rar{\mathsf{F}(f)}
      & \mathsf{F}(\sN)
    \end{tikzcd}
  \]
  commutes in \(\mathsf{Mod}_R\).
\end{citeddef}
The following class of \(F\)-modules satisfies good finiteness properties.
\begin{citeddef}[{\cite[Definitions 1.9 and 2.1]{Lyu97}}]
  Let \(R\) be a regular ring of prime characteristic \(p > 0\).
  Let \((\sM,\theta)\) be an \(F\)-module over \(R\).
  A \textsl{generating morphism} for \(\sM\) is an \(R\)-module map
  \[
    \beta\colon M \longrightarrow \mathsf{F}(M)
  \]
  for an \(R\)-module \(M\) such that \(\sM\) is the direct limit of the direct
  system in the top row of the commutative diagram
  \[
    \begin{tikzcd}[column sep=large]
      M \rar{f} \dar[swap]{\beta}
      & \mathsf{F}(M) \dar{\mathsf{F}(\beta)}\rar{\mathsf{F}(\beta)}
      & \mathsf{F}^2(M) \dar{\mathsf{F}^2(\beta)}\rar{\mathsf{F}^2(\beta)}
      & \cdots\\
      \mathsf{F}(M) \rar{\mathsf{F}(\beta)}
      & \mathsf{F}^2(M) \rar{\mathsf{F}^2(\beta)}
      & \mathsf{F}^3(M) \rar{\mathsf{F}^3(\beta)}
      & \cdots
    \end{tikzcd}
  \]
  and such that the structure isomorphism \(\theta\colon \sM \to
  \mathsf{F}(\sM)\) is the direct limit of the vertical maps in this diagram.
  Any morphism \(\beta\colon M \to \mathsf{F}(M)\) induces an \(F\)-module in
  this way, which we call the \(F\)-module \textsl{generated by} \(\beta\).
  We say that \(\sM\) is \textsl{\(F\)-finite} if \(\sM\) has a generating
  morphism \(\beta\colon M \to \mathsf{F}(M)\) where \(M\) is a finitely
  generated \(R\)-module.
\end{citeddef}
We now state the finite length property of \(F\)-finite \(F\)-modules that we need.
For the statement below, recall
that a ring \(R\) of prime characteristic \(p > 0\) is \textsl{\(F\)-finite} if
the Frobenius map \(F\colon R \to F_*R\) is module-finite \cite[p.\ 464]{Fed83}.
\begin{citedthm}[{\citeleft\citen{Lyu97}\citemid Theorem 3.2\citepunct
  \citen{BB11}\citemid Theorem 5.13\citeright}]\label{thm:fmodfinlen}
  Let \(R\) be a regular ring of prime characteristic \(p > 0\).
  Suppose one of the following assumptions hold.
  \begin{enumerate}[label=\((\roman*)\)]
    \item \(R\) is finitely generated over a regular local ring of prime
      characteristic \(p > 0\).
    \item \(R\) is \(F\)-finite.
  \end{enumerate}
  Then, every \(F\)-finite \(F\)-module over \(R\) has finite length in the
  category of \(F\)-modules over \(R\).
\end{citedthm}

\subsection{The proof of Theorem \ref{thm:mainnew}}
We are now ready to prove Theorem \ref{thm:mainnew} using Theorem
\ref{thm:lyubeznikfunctors} and a modification of
the strategy in \cite{BBLSZ14}.
\begin{customthm}{A}\label{thm:mainnewtext}
  Let \(R\) be an excellent regular ring of finite Krull dimension that is flat
  over a regular domain \(A\) of dimension \(\le 1\).
  Suppose for every nonzero prime ideal \(\fp \subseteq A\) such that \(A/\fp\)
  is of prime characteristic, the quotient ring
  \(R/\fp R\) is regular and one of the following assumptions hold:
  \begin{enumerate}[label=\((\roman*)\),ref=\roman*]
    \item\label{thm:mainnewtextffin} \(R/\fp R\) is \(F\)-finite.
    \item\label{thm:mainnewtexteft} \(A_\fp\) is excellent and
      \(R \otimes_A A_\fp\) is essentially of finite type over \(A_\fp\).
  \end{enumerate} 
  Let \(I \subseteq R\) be an ideal.
  Then, for every \(i \ge 0\), the local cohomology module \(H^i_I(R)\) has
  finitely many associated prime ideals.
\end{customthm}
\begin{proof}
  Working one connected component of \(\Spec(R)\) at a time, we may assume that
  \(R\) is a domain.
  Fix a set of generators \(f_1,f_2,\ldots,f_r\) for \(I\) and consider the
  Koszul cohomology module \(H^i(\underline{f};R)\).
  Since \(H^i(\underline{f};R)\) is finitely generated, we can write
  \[
    \Ass_R\bigl(H^i(\underline{f};R)\bigr) = \{\fq_1,\fq_2,\ldots,\fq_m\}
  \]
  as a finite set of prime ideals \cite[Chapter IV, \S1, no.\ 4, Corollary to
  Theorem 2]{Bou72}.\smallskip
  \par For every \(j\), let \(\fp_j = \fq_j \cap A\).
  We set
  \[
    R_\fp \coloneqq R \otimes_A A_\fp
  \]
  for every prime ideal \(\fp
  \subseteq A\).
  \begin{step}\label{step:koszulassoc}
    For every prime ideal
    \[
      \fp \in \Spec(A) - \bigl\{(0),\fp_1,\fp_2,\ldots,\fp_m\bigr\},
    \]
    the uniformizer of \(A_\fp\)
    is a nonzerodivisor on \(H^i_I(R) \otimes_A A_\fp\).
  \end{step}
  Let \(u\) be the uniformizer of \(A_\fp\).
  It suffices to show that \(u\) is a nonzerodivisor after localizing and
  completing \(R\) at every maximal ideal \(\fm \subseteq R\) containing \(\fp\).
  Since \(R_\fp/uR_\fp \cong R/\fp R\)
  is regular, we may choose \(u\) to be part of a regular system
  of parameters for \((R_\fm)^\land\), in which case \((R_\fm)^\land\) is a
  formal power series ring over a DVR \(V\) with uniformizer \(u\) by Cohen's
  structure theorem \cite[Theorem 15]{Coh46}.
  We then see that \(u\) is a nonzerodivisor on \(H^i_I((R_\fm)^\land)\) by
  \cite[Theorem 4.1(1)]{BBLSZ14}.\smallskip
  \par Since
  \begin{equation}\label{eq:basechangefraca}
    H^i_I(R) \otimes_A \Frac(A) \cong H^i_I\bigl(R \otimes_A \Frac(A)\bigr)
  \end{equation}
  has
  finitely many associated prime ideals by \cite[Corollary 2.3]{HS93} (if
  \(\Frac(A)\) is of characteristic \(p > 0\)) and Theorem \ref{thm:lyubeznikfunctors}
  (if \(\Frac(A)\) is of characteristic zero),
  it remains to show that
  \[
    \bigcup_{\fp \in \Spec(A) - \{(0)\}}
    \Ass_{R_\fp}
    \Bigl(\ker\bigl(H^i_I(R_\fp) \overset{u}{\longrightarrow}
    H^i_I(R_\fp)\bigr)\Bigr)
  \]
  is finite.
  Here, the isomorphism
  \eqref{eq:basechangefraca}
  holds by flat base change for local cohomology modules \cite[Theorem 4.3.2]{BS13}.
  By Step \ref{step:koszulassoc}, this set is equal to
  \[
    \bigcup_{\fp \in \{\fp_1,\fp_2,\ldots,\fp_m\} - \{0\}}  
    \Ass_{R_\fp}
    \Bigl(\ker\bigl(H^i_I(R_\fp) \overset{u}{\longrightarrow}
    H^i_I(R_\fp)\bigr)\Bigr)
  \]
  and hence we may work one prime \(\fp \in \{\fp_1,\fp_2,\ldots,\fp_m\} - \{0\}\) at a time.
  We may therefore replace \(A\) by \(A_\fp\) and assume that \((A,\fp)\) is a
  DVR with uniformizer \(u\).
  \par If \(A\) is of equal characteristic zero, then we are done by Theorem
  \ref{thm:lyubeznikfunctors}.
  It therefore suffices to consider the case when \(A/\fp A\) is of prime
  characteristic.
  \setcounter{step}{1}
  \begin{step}
    The set
    \[
      \Ass_{R}
      \Bigl(\ker\bigl(H^i_I(R) \overset{u}{\longrightarrow}
      H^i_I(R)\bigr)\Bigr)
    \]
    is finite.
  \end{step}
  \par We claim it suffices to show the case when \(R/\fp R\) is \(F\)-finite.
  We show that case \((\ref{thm:mainnewtexteft})\)
  can be reduced to case \((\ref{thm:mainnewtextffin})\).
  By the gamma construction of Hochster and Huneke \citeleft\citen{HH94}\citemid
  \S6\citepunct \citen{Has10}\citemid Lemma 3.23\citeright\ when \(A\) is of
  prime characteristic or by the mixed characteristic version of the gamma
  construction \cite[\S5.2]{HJ24} when \(A\) is of mixed characteristic \((0,p)\),
  there exists a
  faithfully flat extension
  \[
    A \longrightarrow \hat{A} \longrightarrow \hat{A}^\Gamma
  \]
  of \(A\) such
  that \(A/\fp\) is \(F\)-finite and \(R/\fp R \otimes_A \hat{A}^\Gamma\) is
  regular.
  Note that the assumption that \(A_\fp\) is excellent is used to ensure
  that \(A \to \hat{A}\) has geometrically regular fibers, and hence all base
  changes of \(A \to \hat{A}\) by maps essentially of finite type preserve
  regular loci \cite[Proposition 6.8.2 and 6.8.3]{EGAIV2}.
  Since associated prime ideals are compatible with faithfully flat base change
  \cite[Chapter IV, \S1, no.\ 4, Proposition 5]{Bou72}, we may replace \(R\) by
  \(R \otimes_A \hat{A}^\Gamma\)
  to assume that \(R/\fp R\) is
  \(F\)-finite.
  \par Consider the exact sequence
  \[
    0 \longrightarrow R \overset{u}{\longrightarrow} R
    \longrightarrow R/\fp R
    \longrightarrow 0.
  \]
  We then have the associated long exact sequence
  \begin{equation}\label{eq:bblsz14}
    \cdots \longrightarrow H^{i-1}_I(R) \longrightarrow
    H^{i-1}_I(R/\fp R) \overset{d}{\longrightarrow}
    H^i_I(R) \overset{u}{\longrightarrow}
    H^i_I(R) \longrightarrow \cdots.
  \end{equation}
  Then, \(H^{i-1}_I(R/\fp R)\) is an \(F\)-finite \(F\)-module by \cite[Example
  2.2\((b)\)]{Lyu97} and has finite length in the category of
  \(F\)-modules by Theorem \ref{thm:fmodfinlen}.
  Here, we use the assumption that \(R/\fp R\) is \(F\)-finite.
  We will show that
  \begin{equation}\label{eq:redmodpcont}
    \Ass_R\Bigl(\ker\bigl(H^i_I(R) \overset{u}{\longrightarrow}
    H^i_I(R)\bigr)\Bigr) \subseteq \Set*{\fq \subseteq R \given
      \begin{tabular}{@{}c@{}}
        \(\fq\) is an associated prime ideal\\
        of a simple \(F\)-module\\
        component
        of \(H^{i-1}_I(R/\fp R)\)
    \end{tabular}}.
  \end{equation}
  By the exactness of \eqref{eq:bblsz14}, we have the isomorphism
  \[
    \im(d) \cong \ker\bigl(H^i_I(R) \overset{u}{\longrightarrow}
    H^i_I(R)\bigr).
  \]
  Fix a composition series
  \begin{align}
    0 = \sM_0 \subsetneq \sM_1 \subsetneq \cdots &\subsetneq \sM_\ell =
    H^{i-1}_I(R/\fp R)\nonumber
  \intertext{for \(H^{i-1}_I(R/\fp R)\) as an \(F\)-module.
  By \cite[Theorem 2.12\((b)\)]{Lyu97}, each factor \(\sM_i/\sM_{i-1}\) has a
  unique associated prime ideal.
  Since Frobenius is compatible with localization and completion,
  we see that}
    0 = \sM_0 \otimes_R (R_\fq)^\wedge \subseteq \sM_1 \otimes_R (R_\fq)^\wedge
    \subseteq \cdots &\subseteq \sM_\ell \otimes_R (R_\fq)^\wedge \cong
    H^{i-1}_I\bigl((R_\fq)^\wedge/\fp(R_\fq)^\wedge\bigr)\label{eq:compserrphat}
  \intertext{is a filtration as \(F\)-modules.
  Since generating morphisms of \(F\)-modules on \(R\) pull back to
  generating morphisms of \(F\)-modules on \((R_\fq)^\wedge\), the module}
    H^{i-1}_I(R/\fp R) \otimes_R (R_\fq)^\wedge &\cong
    H^{i-1}_I\bigl((R_\fq)^\wedge/\fp(R_\fq)^\wedge\bigr)\nonumber
  \end{align}
  is an \(F\)-finite \(F\)-module over \((R_\fq)^\wedge/\fp(R_\fq)^\wedge\).
  We can therefore refine the
  filtration \eqref{eq:compserrphat} to become a composition series for
  \(H^{i-1}_I(R/\fp R) \otimes_R (R_\fq)^\wedge\).
  Moreover, if
  \[
    \bigl(\sM_i/\sM_{i-1}\bigr) \otimes_R (R_\fq)^\wedge
  \]
  is nonzero, its simple components each
  have a unique associated prime ideal by \cite[Theorem 2.12\((b)\)]{Lyu97}.
  The associated prime ideals of these simple components
  all contract to the unique associated prime ideal of \(\sM_i/\sM_{i-1}\) by
  the compatibility of associated prime ideals with flat base change
  \cite[Chapter IV, \S1, no.\ 4, Proposition 5]{Bou72}.
  \par Now consider \eqref{eq:bblsz14} after base change to
  \((R_\fq)^\wedge\).
  As in the proof of \cite[Theorem 4.1]{BBLSZ14}, we see that the map
  \[
    H^{i-1}_I\bigl((R_\fq)^\wedge/\fp(R_\fq)^\wedge\bigr) \xrightarrow{d \otimes_R
    (R_\fq)^\wedge}
    H^i_I\bigl((R_\fq)^\wedge\bigr)
  \]
  is a map of \(\sD_{(R_\fq)^\wedge/V}\)-modules where \(V\) is a coefficient ring for
  \((R_\fq)^\wedge\) in the sense of \cite[Definition on p.\ 84]{Coh46}.
  Since
  \begin{equation}\label{eq:redmodplc}
    H^{i-1}_I\bigl((R_\fq)^\wedge/\fp(R_\fq)^\wedge\bigr) \cong
    H^{i-1}_I(R/\fp R) \otimes_R (R_\fq)^\wedge
  \end{equation}
  is of finite length as a
  \(\sD_{(R_\fq)^\wedge/V}\)-module by \cite[Theorem 5.7]{Lyu97}, we see that
  \(\im(d) \otimes_R (R_\fq)^\wedge\) is also of finite length and that
  each simple \(\sD_{(R_\fq)^\wedge/V}\)-module component
  \(\im(d) \otimes_R (R_\fq)^\wedge\) is a simple
  \(\sD_{(R_\fq)^\wedge/V}\)-module component of
  \eqref{eq:redmodplc}.
  In particular, since \(\fq\cdot(R_\fq)^\wedge\)
  is an associated prime ideal of
  \(\im(d) \otimes_R (R_\fq)^\wedge\), we see that by
  \citeleft\citen{BBLSZ14}\citemid p.\ 516\citepunct
  \citen{NB14}\citemid Remark 2.3\citeright, there is a simple
  \(\sD_{(R_\fq)^\wedge/V}\)-module component
  \(\sN\) of \eqref{eq:redmodplc}
  whose unique maximal
  associated prime ideal is \(\fq\cdot(R_\fq)^\wedge\).
  By \cite[Theorem 5.6]{Lyu97}, this simple component \(\sN\) is a direct
  summand of a simple \(F\)-module component \(\bar{\sN}\) of
  \eqref{eq:redmodplc}
  and \(\bar{\sN}\) has 
  \(\fq\cdot(R_\fq)^\wedge\) as its unique associated
  prime ideal by \cite[Theorem 2.12\((b)\)]{Lyu97}.
  This simple component \(\bar{\sN}\) appears as a simple component of one of
  the factors
  \[
    \bigl(\sM_i/\sM_{i-1}\bigr) \otimes_R (R_\fq)^\wedge
  \]
  in \eqref{eq:compserrphat}.
  By the previous paragraph, we see that \(\fq\) is the unique associated prime
  ideal of \(\sM_i/\sM_{i-1}\).
  This completes the proof of \eqref{eq:redmodpcont} and therefore of
  Theorem \ref{thm:mainnewtext}.
\end{proof}

\bookmarksetup{startatroot}


\begin{thebibliography}{BGKHME87}

  \bibitem[BB11]{BB11}
    Manuel Blickle and Gebhard B\"ockle, \textit{Cartier modules: finiteness
    results}, J. Reine Angew. Math. \textbf{661} (2011), 85--123;
    \burlalt{https://doi.org/10.1515/crelle.2011.087}{doi:10.1515/crelle.2011.087};
    MR \href{https://mathscinet.ams.org/mathscinet-getitem?mr=2863904}{2863904}.

  \bibitem[BBDG18]{BBDG18}
    Alexander Be\u{\i}linson, Joseph Bernstein, Pierre Deligne, and Ofer Gabber,
    \textit{Faisceaux pervers}, Second edition,
    Analyse et topologie sur les espaces singuliers. I
    ({L}uminy, 1981), Ast\'erisque, vol. 100, Soc. Math. France, Paris, 2018,
    vi+180 pp.; 1982 edition available at
    \url{https://numdam.org/item/AST_1982__100__1_0};
    \burlalt{https://doi.org/10.24033/ast.1042}{doi:10.24033/ast.1042};
    MR \href{https://mathscinet.ams.org/mathscinet-getitem?mr=4870047}{4870047}.

  \bibitem[BBLSZ14]{BBLSZ14}
    Bhargav Bhatt, Manuel Blickle, Gennady Lyubeznik, Anurag K. Singh, and
    Wenliang Zhang, \textit{Local cohomology modules of a smooth
    \(\mathbb{Z}\)-algebra have finitely many associated primes}, Invent. Math.
    \textbf{197} (2014), no. 3, 509--519;
    \burlalt{https://doi.org/10.1007/s00222-013-0490-z}{doi:10.1007/s00222-013-0490-z};
    MR \href{https://mathscinet.ams.org/mathscinet-getitem?mr=3251828}{3251828}.

  \bibitem[BBLSZ]{BBLSZ}
    \bysame, \textit{Applications of perverse sheaves in commutative
    algebra}, J. Reine Angew. Math. Ahead of Print (2025), 58 pp.;
    \burlalt{https://doi.org/10.1515/crelle-2025-0028}{doi:10.1515/crelle-2025-0028}.

  \bibitem[Ber83]{Ber83}
    Joseph Bernstein, \textit{Algebraic theory of \(D\)-modules}, Unpublished
    notes, 1983;
    Available at \url{https://math.uchicago.edu/~drinfeld/langlands.html}.

  \bibitem[BGKHME87]{BGKHME87}
    Armand Borel, Pierre-Paul Grivel, Burchard Kaup, Andr\'{e} Haefliger,
    Bernard Malgrange, and Fritz Ehlers, \textit{Algebraic \(D\)-modules},
    Perspect. Math., vol. 2, Academic Press, Boston, MA, 1987;
    \burlalt{https://n2t.net/ark:/13960/s2w8xr7bmrp}{ark:/13960/s2w8xr7bmrp};
    MR \href{https://mathscinet.ams.org/mathscinet-getitem?mr=882000}{882000}.

  \bibitem[BH22]{BH22}
    Bhargav Bhatt and David Hansen, \textit{The six functors for
    {Z}ariski-constructible sheaves in rigid geometry}, Compos. Math. \textbf{158}
    (2022), no. 2, 437--482;
    \burlalt{https://doi.org/10.1112/s0010437x22007291}{doi:10.1112/s0010437x22007291};
    MR \href{https://mathscinet.ams.org/mathscinet-getitem?mr=4413751}{4413751}.

  \bibitem[BLF00]{BLF00}
    Markus P. Brodmann and A. Lashgari Faghani, \textit{A finiteness result for
    associated primes of local cohomology modules}, Proc. Amer. Math. Soc.
    \textbf{128} (2000), no. 10, 2851--2853;
    \burlalt{https://doi.org/10.1090/S0002-9939-00-05328-4}{doi:10.1090/S0002-9939-00-05328-4};
    MR \href{https://mathscinet.ams.org/mathscinet-getitem?mr=1664309}{1664309}.

  \bibitem[BN08]{BN08}
    Kamal Bahmanpour and Reza Naghipour, \textit{Associated primes of local cohomology
    modules and Matlis duality}, J. Algebra \textbf{320} (2008), no. 6, 2632--2641;
    \burlalt{https://doi.org/10.1016/j.jalgebra.2008.05.014}{doi:10.1016/j.jalgebra.2008.05.014};
    MR \href{https://mathscinet.ams.org/mathscinet-getitem?mr=2441778}{2441778}.

  \bibitem[Bou72]{Bou72}
    Nicolas Bourbaki, \textit{Elements of mathematics. {C}ommutative algebra},
    {T}ranslated from the {F}rench, Hermann, Paris; Addison-Wesley, Reading,
    MA, 1972;
    \burlalt{https://n2t.net/ark:/13960/t56f3ng94}{ark:/13960/t56f3ng94};
    MR \href{https://mathscinet.ams.org/mathscinet-getitem?mr=360549}{360549}.

  \bibitem[BRS00]{BRS00}
    Markus P. Brodmann, Christel Rotthaus, and Rodney Y. Sharp, \textit{On
    annihilators and associated primes of local cohomology modules}, J. Pure
    Appl. Algebra \textbf{153} (2000), no. 3, 197--227;
    \burlalt{https://doi.org/10.1016/S0022-4049(99)00104-8}{doi:10.1016/S0022-4049(99)00104-8};
    MR \href{https://mathscinet.ams.org/mathscinet-getitem?mr=1783166}{1783166}.

  \bibitem[Bry86]{Bry86}
    Jean-Luc Brylinski, \textit{Transformations canoniques, dualit\'e{}
    projective, th\'eorie de {L}efschetz, transformations de {F}ourier et sommes
    trigonom\'etriques}, G\'eom\'etrie et analyse microlocales, Ast\'erisque,
    vol. 140-141, Soc. Math. France, Paris, 1986, pp. 3--134;
    \url{https://www.numdam.org/item/AST_1986__140-141__3_0};
    MR \href{https://mathscinet.ams.org/mathscinet-getitem?mr=864073}{864073}.

  \bibitem[BS13]{BS13}
    Markus P. Brodmann and Rodney Y. Sharp, \textit{Local cohomology. An
    algebraic introduction with geometric applications}, Second edition,
    Cambridge Stud. Adv. Math., vol. 136, Cambridge Univ. Press, Cambridge,
    2013;
    \burlalt{https://doi.org/10.1017/CBO9781139044059}{doi:10.1017/CBO9781139044059};
    MR \href{https://mathscinet.ams.org/mathscinet-getitem?mr=3014449}{3014449}.

  \bibitem[BS15]{BS15}
    Bhargav Bhatt and Peter Scholze, \textit{The pro-étale topology for
    schemes}, De la g\'eom\'etrie alg\'ebrique aux formes automorphes, I,
    Ast\'erisque, vol. 369, Soc. Math. France, Paris, 2015, pp. 99--201;
    \burlalt{https://doi.org/10.24033/ast.960}{doi:10.24033/ast.960};
    MR \href{https://mathscinet.ams.org/mathscinet-getitem?mr=3379634}{3379634}.

  \bibitem[Coh46]{Coh46}
    I. S. Cohen, \textit{On the structure and ideal theory of complete local
    rings}, Trans. Amer. Math. Soc. \textbf{59} (1946), 54--106;
    \burlalt{https://doi.org/10.2307/1990313}{doi:10.2307/1990313};
    MR \href{https://mathscinet.ams.org/mathscinet-getitem?mr=16094}{16094}.

  \bibitem[CRS24]{CRS}
    Yairon Cid-Ruiz and Ilya Smirnov, \textit{Effective generic freeness and
    applications to local cohomology}, J. Lond. Math. Soc. (2) \textbf{110}
    (2024), no. 4, Paper No. e12995, 31 pp.;
    \burlalt{https://doi.org/10.1112/jlms.12995}{doi:10.1112/jlms.12995};
    MR \href{https://mathscinet.ams.org/mathscinet-getitem?mr=4801897}{4801897}.

  \bibitem[dCM10]{dCM10}
    Mark Andrea A. de Cataldo and Luca Migliorini, \textit{The perverse
    filtration and the Lefschetz hyperplane theorem}, Ann. of Math. (2)
    \textbf{171} (2010), no. 3, 2089--2113;
    \burlalt{https://doi.org/10.4007/annals.2010.171.2089}{doi:10.4007/annals.2010.171.2089};
    MR \href{https://mathscinet.ams.org/mathscinet-getitem?mr=2680404}{2680404}.

  \bibitem[Del94]{Del94}
    Pierre Deligne, \textit{D\'ecompositions dans la cat\'egorie d\'eriv\'ee},
    Motives (Seattle, WA, 1991), Proc. Sympos. Pure Math., vol. 55, pt. 1,
    Amer. Math. Soc., Providence, RI, 1994, pp. 115--128;
    \burlalt{https://doi.org/10.1090/pspum/055.1/1265526}{doi:10.1090/pspum/055.1/1265526};
    MR \href{https://mathscinet.ams.org/mathscinet-getitem?mr=1265526}{1265526}.

  \bibitem[DQ18]{DQ18}
    Hailong Dao and Ph\d am H\`ung Qu\'y, \textit{On the associated primes of local
    cohomology}, Nagoya Math. J. \textbf{237} (2020), 1--9;
    \burlalt{https://doi.org/10.1017/nmj.2017.44}{doi:10.1017/nmj.2017.44};
    MR \href{https://mathscinet.ams.org/mathscinet-getitem?mr=4059782}{4059782}.

  \bibitem[EGAIV\textsubscript{1}]{EGAIV1}
    Alexander Grothendieck and Jean A. Dieudonn\'{e}, \textit{\'El\'ements de
    g\'eom\'etrie alg\'ebrique. IV. \'Etude locale des sch\'emas et des
    morphismes de sch\'emas. I}, Inst. Hautes \'Etudes Sci. Publ. Math.
    \textbf{20} (1964), 259 pp.;
    Available at \url{https://www.numdam.org/item/PMIHES_1964__20__5_0};
    \burlalt{https://doi.org/10.1007/BF02684747}{doi:10.1007/BF02684747};
    MR \href{https://mathscinet.ams.org/mathscinet-getitem?mr=173675}{173675}.

  \bibitem[EGAIV\textsubscript{2}]{EGAIV2}
    \bysame, \textit{\'{E}l\'{e}ments de g\'{e}om\'{e}trie alg\'{e}brique. {IV}.
    \'{E}tude locale des sch\'{e}mas et des morphismes de sch\'{e}mas. {II}},
    Inst. Hautes \'{E}tudes Sci. Publ. Math. \textbf{24} (1965), 231 pp.;
    Available at \url{https://www.numdam.org/item/PMIHES_1965__24__5_0};
    \burlalt{https://doi.org/10.1007/BF02684322}{doi:10.1007/BF02684322};
    MR \href{https://mathscinet.ams.org/mathscinet-getitem?mr=199181}{199181}.

  \bibitem[EGAIV\textsubscript{3}]{EGAIV3}
    \bysame, \textit{\'{E}l\'{e}ments de g\'{e}om\'{e}trie alg\'{e}brique. {IV}.
    \'{E}tude locale des sch\'{e}mas et des morphismes de sch\'{e}mas. {III}},
    Inst. Hautes \'{E}tudes Sci. Publ. Math. \textbf{28} (1966), 255 pp.;
    Available at \url{https://www.numdam.org/item/PMIHES_1966__28__5_0};
    \burlalt{https://doi.org/10.1007/BF02684343}{doi:10.1007/BF02684343};
    MR \href{https://mathscinet.ams.org/mathscinet-getitem?mr=217086}{217086}.

  \bibitem[EGAIV\textsubscript{4}]{EGAIV4}
    \bysame, \textit{\'{E}l\'{e}ments de g\'{e}om\'{e}trie alg\'{e}brique. {IV}.
    \'{E}tude locale des sch\'{e}mas et des morphismes de sch\'{e}mas. {IV}},
    Inst. Hautes \'{E}tudes Sci. Publ. Math. \textbf{32} (1967), 361 pp.;
    Available at \url{https://www.numdam.org/item/PMIHES_1967__32__5_0};
    \burlalt{https://doi.org/10.1007/BF02732123}{doi:10.1007/BF02732123};
    MR \href{https://mathscinet.ams.org/mathscinet-getitem?mr=238860}{238860}.

  \bibitem[Eke90]{Eke90}
    Torsten Ekedahl, \textit{On the adic formalism}, The Grothendieck
    Festschrift, Vol. II, Progr. Math., vol. 87, Birkh\"auser Boston,
    Boston, MA, 1990, pp. 197--218;
    \burlalt{https://doi.org/10.1007/978-0-8176-4575-5_4}{doi:10.1007/978-0-8176-4575-5_4};
    MR \href{https://mathscinet.ams.org/mathscinet-getitem?mr=1106899}{1106899}.

  \bibitem[Fal81]{Fal81}
    Gerd Faltings, \textit{Der {E}ndlichkeitssatz in der lokalen {K}ohomologie},
    Math. Ann. \textbf{255} (1981), no. 1, 45--56;
    \burlalt{https://doi.org/10.1007/BF01450555}{doi:10.1007/BF01450555};
    MR \href{https://mathscinet.ams.org/mathscinet-getitem?mr=611272}{611272}.

  \bibitem[Far09]{Far09}
    Laurent Fargues, \textit{Filtration de monodromie et cycles évanescents
    formels}, Invent. Math. \textbf{177} (2009), no. 2, 281--305;
    \burlalt{https://doi.org/10.1007/s00222-009-0184-8}{doi:10.1007/s00222-009-0184-8};
    MR \href{https://mathscinet.ams.org/mathscinet-getitem?mr=2511743}{2511743}.

  \bibitem[Fed83]{Fed83}
    Richard Fedder, \textit{\(F\)-purity and rational singularity}, Trans. Amer.
    Math. Soc. \textbf{278} (1983), no. 2, 461--480;
    \burlalt{https://doi.org/10.2307/1999165}{doi:10.2307/1999165};
    MR \href{https://mathscinet.ams.org/mathscinet-getitem?mr=701505}{701505}.

  \bibitem[FK88]{FK88}
    Eberhard Freitag and Reinhardt Kiehl, \textit{\'Etale cohomology and the
    Weil conjecture}, Translated from the German by Betty S. Waterhouse and
    William C. Waterhouse. With an historical introduction by Jean A.
    Dieudonn\'e, Ergeb. Math. Grenzgeb. (3), vol. 13, Springer-Verlag, Berlin,
    1988;
    \burlalt{https://doi.org/10.1007/978-3-662-02541-3}{doi:10.1007/978-3-662-02541-3};
    MR \href{https://mathscinet.ams.org/mathscinet-getitem?mr=926276}{926276}.

  \bibitem[FK18]{FK18}
    Kazuhiro Fujiwara and Fumiharu Kato, \textit{Foundations of rigid geometry.
    I}, EMS Monogr. Math., Eur. Math. Soc., Z\"urich, 2018;
    \burlalt{https://doi.org/10.4171/135}{doi:10.4171/135};
    MR \href{https://mathscinet.ams.org/mathscinet-getitem?mr=3752648}{3752648}.

  \bibitem[Fuj95]{Fuj95}
    Kazuhiro Fujiwara, \textit{Theory of tubular neighborhood in \'{e}tale
    topology}, Duke Math. J. \textbf{80} (1995), no. 1, 15--57;
    \burlalt{https://doi.org/10.1215/S0012-7094-95-08002-8}{doi:10.1215/S0012-7094-95-08002-8};
    MR \href{https://mathscinet.ams.org/mathscinet-getitem?mr=1360610}{1360610}.

  \bibitem[Gab04]{Gab04}
    Ofer Gabber, \textit{Notes on some \(t\)-structures}, Geometric aspects of
    Dwork theory, Vol. II, De Gruyter, Berlin, 2004,
    pp. 711--734;
    \burlalt{https://doi.org/10.1515/9783110198133.2.711}{doi:10.1515/9783110198133.2.711};
    MR \href{https://mathscinet.ams.org/mathscinet-getitem?mr=2099084}{2099084}.

  \bibitem[Har70]{Har70}
    Robin Hartshorne, \textit{Affine duality and cofiniteness}, Invent. Math.
    \textbf{9} (1970), no. 2, 145--164;
    \burlalt{https://doi.org/10.1007/BF01404554}{doi:10.1007/BF01404554};
    MR \href{https://mathscinet.ams.org/mathscinet-getitem?mr=257096}{257096}.

  \bibitem[Has10]{Has10}
    Mitsuyasu Hashimoto, \textit{\(F\)-pure homomorphisms, strong
    \(F\)-regularity, and \(F\)-injectivity}, Comm. Algebra \textbf{38} (2010),
    no. 12, 4569--4596;
    \burlalt{https://doi.org/10.1080/00927870903431241}{doi:10.1080/00927870903431241};
    MR \href{https://mathscinet.ams.org/mathscinet-getitem?mr=2764840}{2764840}.

  \bibitem[Hei17]{Hei17}
    Katharina Heinrich, \textit{Some remarks on biequidimensionality of
    topological spaces and Noetherian schemes}, J. Commut. Algebra \textbf{9}
    (2017), no. 1, 49--63;
    \burlalt{https://doi.org/10.1216/JCA-2017-9-1-49}{doi:10.1216/JCA-2017-9-1-49};
    MR \href{https://mathscinet.ams.org/mathscinet-getitem?mr=3631826}{3631826}.
    
  \bibitem[Hel01]{Hel01}
    Michael Hellus, \textit{On the set of associated primes of a local
    cohomology module}, J. Algebra \textbf{237} (2001), no. 1, 406--419;
    \burlalt{https://doi.org/10.1006/jabr.2000.8580}{doi:10.1006/jabr.2000.8580};
    MR \href{https://mathscinet.ams.org/mathscinet-getitem?mr=1813886}{1813886}.

  \bibitem[HH94]{HH94}
    Melvin Hochster and Craig Huneke, \textit{\(F\)-regularity, test elements, and smooth base
    change}, Trans. Amer. Math. Soc. \textbf{346} (1994), no. 1, 1--62;
    \burlalt{https://doi.org/10.2307/2154942}{doi:10.2307/2154942};
    MR \href{https://mathscinet.ams.org/mathscinet-getitem?mr=1273534}{1273534}.

  \bibitem[HJ24]{HJ24}
    Melvin Hochster and Jack Jeffries, \textit{A Jacobian criterion for
    nonsingularity in mixed characteristic}, Amer. J. Math. \textbf{146} (2024),
    no. 6, 1749--1780;
    \burlalt{https://doi.org/10.1353/ajm.2024.a944363}{doi:10.1353/ajm.2024.a944363};
    MR \href{https://mathscinet.ams.org/mathscinet-getitem?mr=4855866}{4855866}.

  \bibitem[HK91]{HK91}
    Craig Huneke and Jee Koh, \textit{Cofiniteness and vanishing of local cohomology
    modules}, Math. Proc. Cambridge Philos. Soc. \textbf{110} (1991), no. 3,
    421--429;
    \burlalt{https://doi.org/10.1017/S0305004100070493}{doi:10.1017/S0305004100070493};
    MR \href{https://mathscinet.ams.org/mathscinet-getitem?mr=1120477}{1120477}.

  \bibitem[HNBPW19]{HNBPW19}
    Daniel J. Hern\'andez, Luis N\'u\~nez-Betancourt, Felipe P\'erez, and Emily
    E. Witt, \textit{Lyubeznik numbers and injective dimension in mixed
    characteristic}, Trans. Amer. Math. Soc. \textbf{371} (2019), no. 11, 7533--7557;
    \burlalt{https://doi.org/10.1090/tran/7310}{doi:10.1090/tran/7310};
    MR \href{https://mathscinet.ams.org/mathscinet-getitem?mr=3955527}{3955527}.

  \bibitem[Hoc19]{Hoc19}
    Melvin Hochster, \textit{Finiteness properties and numerical behavior
    of local cohomology}, Comm. Algebra \textbf{47} (2019), no. 6, 1--11;
    \burlalt{https://doi.org/10.1080/00927872.2019.1574807}{doi:10.1080/00927872.2019.1574807};
    MR \href{https://mathscinet.ams.org/mathscinet-getitem?mr=3941632}{3941632}.

  \bibitem[HS93]{HS93}
    Craig L. Huneke and Rodney Y. Sharp, \textit{Bass numbers of local
    cohomology modules}, Trans. Amer. Math. Soc. \textbf{339} (1993), no. 2, 765--779;
    \burlalt{https://doi.org/10.2307/2154297}{doi:10.2307/2154297};
    MR \href{https://mathscinet.ams.org/mathscinet-getitem?mr=1124167}{1124167}.

  \bibitem[HTT08]{HTT08}
    Ryoshi Hotta, Kiyoshi Takeuchi, and Toshiyuki Tanisaki,
    \textit{\(D\)-modules, perverse sheaves, and representation theory},
    Translated from the 1995 Japanese edition by Takeuchi, Progr. Math., vol.
    236, Birkh\"auser Boston, Inc., Boston, MA, 2008;
    \burlalt{https://doi.org/10.1007/978-0-8176-4523-6}{doi:10.1007/978-0-8176-4523-6};
    MR \href{https://mathscinet.ams.org/mathscinet-getitem?mr=2357361}{2357361}.

  \bibitem[Hun92]{Hun92}
    Craig Huneke, \textit{Problems on local cohomology}, Free resolutions in
    commutative algebra and algebraic geometry (Sundance, UT, 1990), Res. Notes
    Math., vol. 2, Jones and Bartlett, Boston, MA, 1992, pp. 93--108;
    \burlalt{https://doi.org/10.1201/9781003420187}{doi:10.1201/9781003420187};
    MR \href{https://mathscinet.ams.org/mathscinet-getitem?mr=1165320}{1165320}.

  \bibitem[ILO14]{ILO14}
    Luc Illusie, Yves Laszlo, and Fabrice Orgogozo, eds.,
    With the collaboration of Fr\'ed\'eric D\'eglise, Alban
    Moreau, Vincent Pilloni, Michel Raynaud, Jo\"el Riou, Beno\^it
    Stroh, Michael Temkin, and Weizhe Zheng, \textit{Travaux de
    {G}abber sur l'uniformisation locale et la cohomologie \'etale des
    sch\'emas quasi-excellents}, S\'eminaire \`a{} l'\'Ecole Polytechnique 2006--2008,
    Ast\'erisque, vol. 363-364, Soc. Math. France, Paris, 2014, xxiv+625 pp.;
    Corrected version of Nov. 14, 2016 available at
    \url{http://fabrice.orgogozo.perso.math.cnrs.fr/travaux_de_Gabber};
    \burlalt{https://doi.org/10.24033/ast.935}{doi:10.24033/ast.935};
    MR \href{https://mathscinet.ams.org/mathscinet-getitem?mr=3309086}{3309086}.

  \bibitem[Jut09]{Jut09}
    Daniel Juteau, \textit{Decomposition numbers for perverse sheaves}, Ann.
    Inst. Fourier (Grenoble) \textbf{59} (2009), no. 3, 1177--1229;
    \burlalt{https://doi.org/10.5802/aif.2461}{doi:10.5802/aif.2461};
    MR \href{https://mathscinet.ams.org/mathscinet-getitem?mr=2543666}{2543666}.

  \bibitem[Kas70]{Kas70}
    Masaki Kashiwara, \textit{Algebraic study of systems of partial differential
    equations}, Master's thesis, University of Tokyo, 1970, Translated by
    Jean-Pierre Schneiders and Andrea D'Agnolo, M\'em. Soc. Math. France (N.S.)
    \textbf{63} (1995), xiv+72 pp.;
    \burlalt{https://doi.org/10.24033/msmf.377}{doi:10.24033/msmf.377};
    MR \href{https://mathscinet.ams.org/mathscinet-getitem?mr=1384226}{1384226}.

  \bibitem[Kas75]{Kas75}
    \bysame, \textit{On the maximally overdetermined system of linear
    differential equations. I}, Publ. Res. Inst. Math. Sci. \textbf{10} (1975),
    563--579;
    \burlalt{https://doi.org/10.2977/prims/1195192011}{doi:10.2977/prims/1195192011};
    MR \href{https://mathscinet.ams.org/mathscinet-getitem?mr=370665}{370665}.

  \bibitem[Kas78]{Kas78}
    \bysame, \textit{On the holonomic systems of linear differential equations.
    II}, Invent. Math. \textbf{49} (1978), no. 2, 121--135;
    \burlalt{https://doi.org/10.1007/BF01403082}{doi:10.1007/BF01403082};
    MR \href{https://mathscinet.ams.org/mathscinet-getitem?mr=511186}{511186}.

  \bibitem[Kas80]{Kas80}
    \bysame, \textit{Faisceaux constructibles et syst\`emes holon\^omes
    d'\'equations aux d\'eriv\'ees partielles lin\'eaires \`a{} points
    singuliers r\'eguliers}, S\'eminaire {G}oulaouic-{S}chwartz, 1979--1980,
    \'Ecole Polytech., Palaiseau, 1980, Exp. No. 19, 7 pp.;
    \url{https://www.numdam.org/item/SEDP_1979-1980____A20_0};
    MR \href{https://mathscinet.ams.org/mathscinet-getitem?mr=600704}{600704}.

  \bibitem[Kas84]{Kas84}
    \bysame, \textit{The {R}iemann-{H}ilbert problem for holonomic systems}, Publ.
    Res. Inst. Math. Sci. \textbf{20} (1984), no. 2, 319--365;
    \burlalt{https://doi.org/10.2977/prims/1195181610}{doi:10.2977/prims/1195181610};
    MR \href{https://mathscinet.ams.org/mathscinet-getitem?mr=743382}{743382}.

  \bibitem[Kat02]{Kat02}
    Mordechai Katzman, \textit{An example of an infinite set of associated primes of a local
    cohomology module}, J. Algebra \textbf{252} (2002), no. 1, 161--166;
    \burlalt{https://doi.org/10.1016/S0021-8693(02)00032-7}{doi:10.1016/S0021-8693(02)00032-7};
    MR \href{https://mathscinet.ams.org/mathscinet-getitem?mr=1922391}{1922391}.

  \bibitem[KK81]{KK81}
    Masaki Kashiwara and Takahiro Kawai, \textit{On holonomic systems of
    microdifferential equations. III. Systems with regular singularities}, Publ.
    Res. Inst. Math. Sci. \textbf{17} (1981), no. 3, 813--979;
    \burlalt{https://doi.org/10.2977/prims/1195184396}{doi:10.2977/prims/1195184396};
    MR \href{https://mathscinet.ams.org/mathscinet-getitem?mr=650216}{650216}.

  \bibitem[KS99]{KS99}
    Kazem Khashyarmanesh and Shokrollah Salarian, \textit{On the associated primes of local
    cohomology modules}, Comm. Algebra \textbf{27} (1999), no. 12, 6191--6198;
    \burlalt{https://doi.org/10.1080/00927879908826816}{doi:10.1080/00927879908826816};
    MR \href{https://mathscinet.ams.org/mathscinet-getitem?mr=1726302}{1726302}.

  \bibitem[Kun69]{Kun69}
    Ernst Kunz, \textit{Characterizations of regular local rings of
    characteristic \(p\)}, Amer. J. Math. \textbf{91} (1969), 772--784;
    \burlalt{https://doi.org/10.2307/2373351}{doi:10.2307/2373351};
    MR \href{https://mathscinet.ams.org/mathscinet-getitem?mr=252389}{252389}.

  \bibitem[Lyu93]{Lyu93}
    Gennady Lyubeznik, \textit{Finiteness properties of local cohomology modules (an
    application of \(D\)-modules to commutative algebra)}, Invent. Math.
    \textbf{113} (1993), no. 1, 41--55;
    \burlalt{https://doi.org/10.1007/BF01244301}{doi:10.1007/BF01244301};
    MR \href{https://mathscinet.ams.org/mathscinet-getitem?mr=1223223}{1223223}.

  \bibitem[Lyu97]{Lyu97}
    \bysame, \textit{\(F\)-modules: applications to local cohomology and
    \(D\)-modules in characteristic \(p>0\)}, J. Reine Angew. Math. \textbf{491}
    (1997), 65--130;
    \burlalt{https://doi.org/10.1515/crll.1997.491.65}{doi:10.1515/crll.1997.491.65};
    MR \href{https://mathscinet.ams.org/mathscinet-getitem?mr=1476089}{1476089}.

  \bibitem[Lyu00]{Lyu00}
    \bysame, \textit{Finiteness properties of local cohomology modules for
    regular local rings of mixed characteristic: the unramified case},
    Special issue in honor of Robin Hartshorne,
    Comm. Algebra \textbf{28} (2000), no. 12, 5867--5882;
    \burlalt{https://doi.org/10.1080/00927870008827193}{doi:10.1080/00927870008827193};
    MR \href{https://mathscinet.ams.org/mathscinet-getitem?mr=1808608}{1808608}.

  \bibitem[Lyu02]{Lyu02}
    \bysame, \textit{A partial survey of local cohomology}, Local cohomology and
    its applications (Guanajuato, 1999), Lecture Notes in Pure and Appl.
    Math., vol. 226, Dekker, New York, 2002, pp. 121--154;
    \burlalt{https://doi.org/10.1201/9781482275766}{doi:10.1201/9781482275766};
    MR \href{https://mathscinet.ams.org/mathscinet-getitem?mr=1888197}{1888197}.

  \bibitem[Mar01]{Mar01}
    Thomas Marley, \textit{The associated primes of local cohomology modules over rings of
    small dimension}, Manuscripta Math. \textbf{104} (2001), no. 4, 519--525;
    \burlalt{https://doi.org/10.1007/s002290170024}{doi:10.1007/s002290170024};
    MR \href{https://mathscinet.ams.org/mathscinet-getitem?mr=1836111}{1836111}.

  \bibitem[Mat80]{Mat80}
    Hideyuki Matsumura, \textit{Commutative algebra}, Second edition, Math.
    Lecture Note Ser., vol. 56, Benjamin/Cummings Publishing Co., Inc., Reading,
    MA, 1980; 1970 edition available at
    \burlalt{https://n2t.net/ark:/13960/s2gv47f226m}{ark:/13960/s2gv47f226m};
    MR \href{https://mathscinet.ams.org/mathscinet-getitem?mr=575344}{575344}.

  \bibitem[Meb77]{Meb77}
    Zoghman Mebkhout, \textit{Local cohomology of analytic spaces},
    \textit{Publ. Res. Inst. Math. Sci.} \textbf{12} (1977), suppl.,
    247--256;
    \burlalt{https://doi.org/10.2977/prims/1195196607}{doi:10.2977/prims/1195196607};
    MR \href{https://mathscinet.ams.org/mathscinet-getitem?mr=454056}{454056}.

  \bibitem[Meb80]{Meb80}
    \bysame, \textit{Sur le probl\`eme de {H}ilbert-{R}iemann},
    Complex analysis, microlocal calculus and relativistic quantum theory
    ({P}roc. {I}nternat. {C}olloq., {C}entre {P}hys., {L}es {H}ouches, 1979),
    Lecture Notes in Phys., vol. 126, Springer, Berlin-New York, pp. 90--110;
    \burlalt{https://doi.org/10.1007/3-540-09996-4_31}{doi:10.1007/3-540-09996-4_31};
    MR \href{https://mathscinet.ams.org/mathscinet-getitem?mr=579742}{579742}.

  \bibitem[Meb82]{Meb82}
    \bysame, \textit{Th\'eor\`emes de bidualit\'e{} locale pour les
    \(\mathcal{D}_X\)-modules holonomes}, Ark. Mat. \textbf{20} (1982), no. 1,
    111--124;
    \burlalt{https://doi.org/10.1007/BF02390502}{doi:10.1007/BF02390502};
    MR \href{https://mathscinet.ams.org/mathscinet-getitem?mr=660129}{660129}.

  \bibitem[Meb84a]{Meb84a}
    \bysame, \textit{Une \'equivalence de cat\'egories}, Compositio
    Math. \textbf{51} (1984), no. 1, 51--62;
    \url{https://www.numdam.org/item/CM_1984__51_1_51_0};
    MR \href{https://mathscinet.ams.org/mathscinet-getitem?mr=734784}{734784}.

  \bibitem[Meb84b]{Meb84b}
    \bysame, \textit{Une autre \'equivalence de cat\'egories}, Compositio
    Math. \textbf{51} (1984), no. 1, 63--88;
    \url{https://www.numdam.org/item/CM_1984__51_1_63_0};
    MR \href{https://mathscinet.ams.org/mathscinet-getitem?mr=734785}{734785}.

  \bibitem[Meb89]{Meb89}
    \bysame, \textit{Le formalisme des six op\'erations de {G}rothendieck pour les
    {$\mathscr{D}_X$}-modules coh\'erents}, With supplementary material by the
    author and Luis Narv\'aez Macarro, Travaux en Cours, vol. 35, Hermann,
    Paris, 1989;
    MR \href{https://mathscinet.ams.org/mathscinet-getitem?mr=1008245}{1008245}.

  \bibitem[MNM91]{MNM91}
    Zoghman Mebkhout and Luis Narv\'aez-Macarro, \textit{La th\'eorie du
    polyn\^ome de {B}ernstein-{S}ato pour les alg\`ebres de {T}ate et de
    {D}work-{M}onsky-{W}ashnitzer}, Ann. Sci. \'Ecole Norm. Sup. (4) \textbf{24}
    (1991), no. 2, 227--256;
    \burlalt{https://doi.org/10.24033/asens.1627}{doi:10.24033/asens.1627};
    MR \href{https://mathscinet.ams.org/mathscinet-getitem?mr=1097693}{1097693}.

  \bibitem[Mor25]{Mor25}
    Sophie Morel, \textit{Mixed \(\ell\)-adic complexes for schemes over number
    fields}, Doc. Math. \textbf{30} (2025), no. 1, 105--181;
    \burlalt{https://doi.org/10.4171/DM/990}{doi:10.4171/DM/990};
    MR \href{https://mathscinet.ams.org/mathscinet-getitem?mr=4855495}{4855495}.

  \bibitem[Mur25]{Mur25}
    Takumi Murayama, \textit{Relative vanishing theorems for
    \(\mathbf{Q}\)-schemes}, Algebr. Geom. \textbf{12} (2025), no. 1, 84--144;
    \burlalt{https://doi.org/10.14231/AG-2025-003}{doi:10.14231/AG-2025-003};
    MR \href{https://mathscinet.ams.org/mathscinet-getitem?mr=4841227}{4841227}.

  \bibitem[NB13]{NB13}
    Luis N\'u\~nez-Betancourt, \textit{On certain rings of differentiable type and
    finiteness properties of local cohomology}, J. Algebra \textbf{379} (2013),
    1--10;
    \burlalt{https://doi.org/10.1016/j.jalgebra.2012.12.010}{doi:10.1016/j.jalgebra.2012.12.010};
    MR \href{https://mathscinet.ams.org/mathscinet-getitem?mr=3019242}{3019242}.

  \bibitem[NB14]{NB14}
    \bysame, \textit{Associated primes of local cohomology of flat extensions with
    regular fibers and \(\Sigma\)-finite \(D\)-modules}, J. Algebra \textbf{399}
    (2014), 770--781;
    \burlalt{https://doi.org/10.1016/j.jalgebra.2013.10.010}{doi:10.1016/j.jalgebra.2013.10.010};
    MR \href{https://mathscinet.ams.org/mathscinet-getitem?mr=3144611}{3144611}.

  \bibitem[NM14]{NM14}
    Luis Narv\'aez Macarro, \textit{Differential structures in commutative
    algebra}, Mini-course at the XXIII Brazilian Algebra Meeting, July 27--August 1, 2014,
    Maring\'a, Brazil, Version of Nov. 10, 2014;
    \url{https://personal.us.es/narvaez/DSCA_course_2014.pdf}.

  \bibitem[Pop85]{Pop85}
    Dorin Popescu, \textit{General N\'eron desingularization}, Nagoya Math. J.
    \textbf{100} (1985), 97--126;
    \burlalt{https://doi.org/10.1017/S0027763000000246}{doi:10.1017/S0027763000000246};
    MR \href{https://mathscinet.ams.org/mathscinet-getitem?mr=818160}{818160}.

  \bibitem[Pop86]{Pop86}
    \bysame, \textit{General N\'eron desingularization and approximation},
    Nagoya Math. J. \textbf{104} (1986), 85--115;
    \burlalt{https://doi.org/10.1017/S0027763000022698}{doi:10.1017/S0027763000022698};
    MR \href{https://mathscinet.ams.org/mathscinet-getitem?mr=868439}{868439}.

  \bibitem[Pop90]{Pop90}
    \bysame, \textit{Letter to the editor: ``General N\'eron desingularization and
    approximation''}, Nagoya Math. J. \textbf{118} (1990), 45--53;
    \burlalt{https://doi.org/10.1017/S0027763000002981}{doi:10.1017/S0027763000002981};
    MR \href{https://mathscinet.ams.org/mathscinet-getitem?mr=1060701}{1060701}.

  \bibitem[PS73]{PS73}
    Christian Peskine and Lucien Szpiro, \textit{Dimension projective finie et
    cohomologie locale. Applications \`a{} la d\'emonstration de conjectures
    de M. Auslander, H. Bass et A.  Grothendieck}, Inst. Hautes \'Etudes Sci.
    Publ. Math. \textbf{42} (1973), 47--119;
    Available at \url{https://www.numdam.org/item/PMIHES_1973__42__47_0};
    \burlalt{https://doi.org/10.1007/BF02685877}{doi:10.1007/BF02685877};
    MR \href{https://mathscinet.ams.org/mathscinet-getitem?mr=374130}{374130}.

  \bibitem[Put16]{Put16}
    Tony J. Puthenpurakal, \textit{Associated primes of local cohomology modules over
    regular rings}, Pacific J. Math. \textbf{282} (2016), no. 1, 233--255;
    \burlalt{https://doi.org/10.2140/pjm.2016.282.233}{doi:10.2140/pjm.2016.282.233};
    MR \href{https://mathscinet.ams.org/mathscinet-getitem?mr=3463431}{3463431}.

  \bibitem[Put18]{Put18}
    \bysame, \textit{On the ring of differential operators of certain
    regular domains}, Proc. Amer. Math. Soc. \textbf{146} (2018), no. 8, 3333--3343;
    \burlalt{https://doi.org/10.1090/proc/14039}{doi:10.1090/proc/14039};
    MR \href{https://mathscinet.ams.org/mathscinet-getitem?mr=3803659}{3803659}.

  \bibitem[SGA2]{SGA2}
    Alexander Grothendieck,
    \textit{Cohomologie locale des faisceaux coh\'erents et th\'eor\`emes
    de Lefschetz locaux et globaux (SGA 2)}, S\'eminaire de G\'eom\'etrie
    Alg\'ebrique du Bois Marie, 1962, With an expos\'e by Mich\`{e}le Raynaud,
    Revised reprint of the 1968 French original with a preface and edited by Yves Laszlo,
    Doc. Math. (Paris), vol. 4, Soc. Math. France, Paris, 2005;
    Available at \url{https://www.cmls.polytechnique.fr/perso/laszlo/sga2/sga2-smf.pdf};
    MR \href{https://mathscinet.ams.org/mathscinet-getitem?mr=2171939}{2171939}.

  \bibitem[SGA3\textsubscript{1}]{SGA3}
    \textit{{Sch\'emas en groupes ({SGA} 3). {T}ome {I}. {P}ropri\'et\'es
    g\'en\'erales des sch\'emas en groupes}}, S\'eminaire de G\'eom\'etrie
    Alg\'ebrique du Bois Marie, 1962--64, A seminar directed by Michel Demazure
    and Alexander Grothendieck with the collaboration of Michael Artin,
    Jean-\'Etienne Bertin, Pierre Gabriel, Michel Raynaud, and Jean-Pierre
    Serre, Revised and annotated edition of the 1970 French original edited by
    Philippe Gille and Patrick Polo, Doc. Math. (Paris), vol. 7, Soc. Math.
    France, Paris;
    Corrected version of Oct. 13, 2024 available at
    \url{https://webusers.imj-prg.fr/~patrick.polo/SGA3};
    MR \href{https://mathscinet.ams.org/mathscinet-getitem?mr=2867621}{2867621}.

  \bibitem[SGA4\textsubscript{3}]{SGA43}
    \textit{Th\'{e}orie des topos et cohomologie \'{e}tale des sch\'{e}mas (SGA 4).
    {T}ome 3}, S\'{e}minaire de G\'{e}om\'{e}trie Alg\'{e}brique du Bois Marie,
    1963--1964, A seminar directed by Michael Artin, Alexander Grothendieck, and
    Jean-Louis Verdier with the collaboration of Pierre
    Deligne and Bernard Saint-Donat, Lecture Notes in Math., vol. 305,
    Springer-Verlag, Berlin-New York, 1973;
    Revised version of Jul. 30, 2024 available at
    \url{http://fabrice.orgogozo.perso.math.cnrs.fr/SGA4};
    \burlalt{https://doi.org/10.1007/BFb0070714}{doi:10.1007/BFb0070714};
    MR \href{https://mathscinet.ams.org/mathscinet-getitem?mr=354654}{354654}.

  \bibitem[SGA5]{SGA5}
    \textit{Cohomologie {\(l\)}-adique et fonctions {\(L\)}}, S\'{e}minaire de
    G\'{e}ometrie Alg\'{e}brique du Bois-Marie 1965--1966 (SGA 5), Dirig\'{e}
    par Alexander Grothendieck. Avec la collaboration de Ionel Bucur, Christian
    Houzel, Luc Illusie, Jean-Pierre Jouanolou et Jean-Pierre Serre, Lecture
    Notes in Math., vol. 589, Springer-Verlag, Berlin-New York, 1977;
    \burlalt{https://doi.org/10.1007/BFb0096802}{doi:10.1007/BFb0096802};
    MR \href{https://mathscinet.ams.org/mathscinet-getitem?mr=491704}{491704}.

  \bibitem[Sin00]{Sin00}
    Anurag K. Singh, \textit{\(p\)-torsion elements in local cohomology modules},
    Math. Res. Lett. \textbf{7} (2000), no. 2-3, 165--176;
    \burlalt{https://doi.org/10.4310/MRL.2000.v7.n2.a3}{doi:10.4310/MRL.2000.v7.n2.a3};
    MR \href{https://mathscinet.ams.org/mathscinet-getitem?mr=1764314}{1764314}.

  \bibitem[SS04]{SS04}
    Anurag K. Singh and Irena Swanson, \textit{Associated primes of local cohomology modules
    and of Frobenius powers}, Int. Math. Res. Not. \textbf{33} (2004), 1703--1733;
    \burlalt{https://doi.org/10.1155/S1073792804133424}{doi:10.1155/S1073792804133424};
    MR \href{https://mathscinet.ams.org/mathscinet-getitem?mr=2058025}{2058025}.

  \bibitem[Swa98]{Swa98}
    Richard G. Swan, \textit{N\'eron-Popescu desingularization}, Algebra and
    geometry (Taipei, 1995), Lect. Algebra Geom., vol. 2, Int. Press, Cambridge,
    MA, 1998, pp. 135--192;
    MR \href{https://mathscinet.ams.org/mathscinet-getitem?mr=1697953}{1697953}.

  \bibitem[Ver67]{Ver67}
    Jean-Louis Verdier, \textit{Des cat\'egories d\'eriv\'ees des cat\'egories
    ab\'eliennes}, With a preface by Luc Illusie, Edited and with a note by
    Georges Maltsiniotis, Ast\'erisque, vol. 239, Soc. Math. France, Paris,
    1996, xii+253 pp.;
    \url{https://www.numdam.org/item/AST_1996__239__R1_0};
    MR \href{https://mathscinet.ams.org/mathscinet-getitem?mr=1453167}{1453167}.

  \bibitem[Wan23]{Wan23}
    Yinan Nancy Wang, \textit{Local cohomology modules and motivic Chern class
    computations}, Thesis (Ph.D.)--University of Michigan, 2023, vi+113 pp.;
    \burlalt{https://doi.org/10.7302/8176}{doi:10.7302/8176};
    MR \href{https://mathscinet.ams.org/mathscinet-getitem?mr=4675907}{4675907}.

\end{thebibliography}
\end{document}